\documentclass[a4paper,11pt]{article}

\usepackage[T1]{fontenc}
\usepackage[utf8]{inputenc}
\usepackage{amsmath, amssymb, mathtools, amsthm}
\mathtoolsset{showonlyrefs}
\usepackage[margin=0.9in]{geometry}
\usepackage{bbm}
\usepackage[numbers,sort&compress]{natbib}
\usepackage[colorlinks=true]{hyperref}
\hypersetup{urlcolor=blue, citecolor=blue, linkcolor=blue}

\DeclarePairedDelimiter{\abs}{\lvert}{\rvert}
\DeclarePairedDelimiter{\norm}{\lVert}{\rVert}
\DeclarePairedDelimiter{\set}{\{}{\}}
\DeclarePairedDelimiter{\skp}{\langle}{\rangle}

\DeclarePairedDelimiter{\prt}{(}{)}
\DeclarePairedDelimiter{\brk}{[}{]}

\DeclareMathOperator*{\esssup}{ess\,sup}
\DeclareMathOperator*{\supp}{\operatorname{supp}}
\DeclareMathOperator*{\tr}{\operatorname{tr}}
\DeclareMathOperator*{\Id}{\operatorname{Id}}

\DeclareMathAlphabet{\mathup}{OT1}{\familydefault}{m}{n}
\newcommand{\dx}[1]{\mathop{}\!\mathup{d} #1}
\newcommand{\dd}{\,\mathup{d}}

\newcommand{\inb}{\stackrel{b}{\in}}
\newcommand{\ddt}{\frac{\dx{}}{\dx{t}}}
\renewcommand{\div}{\operatorname{div}}
\newcommand{\R}{\mathbb{R}}
\newcommand{\hpsi}{\hat\psi}
\newcommand{\intO}{\int_{\Omega}}
\newcommand{\intOT}{\int_{\Omega_{t}}}
\newcommand{\intOTf}{\int_{\Omega_{T}}}
\newcommand{\intOD}{\int_{\Omega\times D}}
\newcommand{\OD}{\Omega\times D}
\newcommand{\F}{\mathcal{F}}

\newcommand\blfootnote[1]{%
  \begingroup
  \renewcommand\thefootnote{}\footnote{#1}%
  \addtocounter{footnote}{-1}%
  \endgroup
}

\theoremstyle{plain}
\newtheorem{theorem}{Theorem}[section]

\newtheorem{corollary}[theorem]{Corollary}
\theoremstyle{remark}
\newtheorem{remark}[theorem]{\bf Remark}
\newtheorem{definition}[theorem]{\bf Definition}

\begin{document}

\title{On a class of generalised solutions to the kinetic Hookean dumbbell model for incompressible dilute polymeric fluids: existence and macroscopic closure
}

\author{Tomasz D\k{e}biec\thanks{Institute of Applied Mathematics and Mechanics, University of Warsaw, Banacha 2, 02-097 Warsaw, Poland.}
\and{Endre S\"{u}li\thanks{Mathematical Institute, University of Oxford, Woodstock Road, Oxford OX2 6GG, United Kingdom.}}}

\maketitle
\begin{abstract}
		We consider the Hookean dumbbell model, a system of nonlinear PDEs arising in the kinetic theory of homogeneous dilute polymeric fluids. It consists of the unsteady incompressible Navier--Stokes equations in a bounded Lipschitz domain, coupled to a Fokker--Planck-type parabolic equation with a centre-of-mass diffusion term, for the probability density function, modelling the evolution of the configuration of noninteracting polymer molecules in the solvent. 
		The micro-macro interaction is reflected by the presence of a drag term in the Fokker--Planck equation and the divergence of a polymeric extra-stress tensor in the Navier--Stokes balance of momentum equation.
		We introduce the concept of \emph{generalised dissipative solution} -- a relaxation of the usual notion of weak solution, allowing for the presence of a, possibly nonzero, defect measure in the momentum equation. This defect measure accounts for the lack of compactness in the polymeric extra-stress tensor. We prove global existence of generalised dissipative solutions satisfying additionally an energy inequality for the macroscopic deformation tensor. Using this inequality, we establish a conditional regularity result:\ any generalised dissipative solution with a sufficiently regular velocity field is a weak solution to the Hookean dumbbell model. Additionally, in two space dimensions we provide a rigorous derivation of the macroscopic closure of the Hookean model and discuss its relationship with the Oldroyd-B model with stress
diffusion. Finally, we improve a result by Barrett and S\"uli~[Barrett, J. W. and S\"{u}li, E., Existence of global weak solutions to the kinetic Hookean dumbbell model for incompressible dilute polymeric fluids. Nonlinear Anal. Real World Appl. 39 (2018), 362–395] by establishing the global existence of weak solutions for a larger class of initial data.
\end{abstract}

\vskip .4cm
\begin{flushleft}
    \noindent{\makebox[1in]\hrulefill}
\end{flushleft}
\begin{flushright}
    \noindent{\makebox[1in]\hrulefill}
\end{flushright}
\vskip .4cm

\blfootnote{{Email addresses:} {t.debiec@mimuw.edu.pl}, {suli@maths.ox.ac.uk}.}


\section{Description of the model}

In this paper, we study generalised solutions to a Navier--Stokes--Fokker--Planck (NSFP) model describing the flow of a dilute polymeric fluid. We establish the existence of solutions admitting a defect measure in the Navier--Stokes equation but obeying an energy-type inequality in such a way that the energy defect controls the defect in the flow equation. Thus, we extend and improve an earlier result of Barrett and S\"uli~\cite{BaSu18} in three space dimensions, where no energy inequality was proved. In the case of two spatial dimensions with additional regularity of the initial data, we show that our generalised solution is in fact a weak solution of the NSFP system. Furthermore, we show that the ``defected'' extra-stress tensor satisfies the Oldroyd-B equations -- thus providing a new existence proof for the macroscopic system. 

Polymer molecules are long chains of monomers and their presence in a fluid induces a number of complex phenomena in the resulting mixture, making their mathematical analysis challenging. The simplest kinetic model for a polymeric fluid is the \emph{dumbbell model} wherein each polymer molecule is idealised to consist of two small massless beads connected by an elastic spring. Furthermore, the solution is assumed to be dilute, i.e., it is supposed that 
distinct polymer molecules do not interact with each other, and individual polymer molecules do not exhibit self-interaction. A crucial modelling choice is the form of the elastic spring force:
\begin{equation}
\label{eq:SpringPotential}
	F:D\to\R^d,\qquad F\prt*{\frac{|q|^2}{2}} = H U'\prt*{\frac{|q|^2}{2}}\, q,
\end{equation}
where $U$ is the spring potential, $H$ is the spring constant, and $D$ is the space of admissible conformation vectors. We shall choose $D=\R^d$ and assume that the springs obey Hooke's law, i.e., $U(s)=s$, with the spring constant $H$ scaled to $1$. While it is obviously nonphysical that the polymer molecules are allowed to extend indefinitely, this simple description has practical appeal. In particular, the resulting micro-macro model has a (formal) macroscopic closure, the Oldroyd-B model~\cite{Oldroyd}.
We shall therefore focus our attention on the following Hookean dumbbell model: 
\begin{alignat}{2}
\begin{aligned}
\label{eq:NSFP}
	\partial_t u + \prt*{u\cdot \nabla_x}u - \nu \Delta_x u + \nabla_x p &= \div_x \tau &&\quad \mbox{on $(0,T) \times \Omega$},\\
	\div_x u &=0&&\quad \mbox{on $(0,T) \times \Omega$},\\
	\partial_t \psi + u\cdot\nabla_x\psi + \div_q\prt*{(\nabla_x u) q\psi} - \mu\Delta_x\psi &= \div_q\prt*{M\nabla_q\prt*{\frac{\psi}{M}}}&&\quad \mbox{on $(0,T) \times \Omega \times D$},
\end{aligned}
\end{alignat}
where $T>0$, 
$\Omega\subset\R^d$ is a bounded open domain, and $D=\R^d$ is the conformation domain.
The \emph{polymeric extra stress tensor} $\tau$ appearing on the right-hand side of the conservation of linear momentum equation \eqref{eq:NSFP}$_1$ arises because of the presence of the polymer molecules in the Newtonian fluid, with 
constant viscosity $\nu>0$, acting as a solvent, and $\mu>0$ is the so-called centre-of-mass diffusion coefficient. The extra stress tensor
$\tau$ is related
to the probability density function $\psi$ by the \emph{Kramers expression} (again, with all physical constants scaled to 1):
\begin{equation}
	\tau(t,x) := \int_D \psi(t,x,q)\, (q\otimes q) \dx{q} - \prt*{\int_D \psi(t,x,q) \dx{q}}\Id.
\end{equation}
Here $q \otimes q \in \mathbb{R}^{d \times d}_{\textrm{sym}}$, with entries $(q \otimes q)_{i,j} = q_i q_j$ for $i,j=1,\ldots,d$, and $\Id \in \mathbb{R}^{d \times d}$ is the identity matrix. 
Thus, $\tau$ is a symmetric tensor.
The factor $M$, present in the last term of the Fokker--Planck equation  \eqref{eq:NSFP}$_3$, is the normalised Maxwellian, defined by
\begin{equation}
	M(q) := \frac{1}{\mathcal{Z}} \mathrm{e}^{-\frac{|q|^2}{2}},\quad \mathcal{Z}:= \int_D \mathrm{e}^{-\frac{|q|^2}{2}}\dx{q}.
\end{equation}
We note in passing that the triple $(u,p,\psi)=(0,c,M)$ is an equilibrium solution of the system \eqref{eq:NSFP} for any constant $c$.
We observe immediately some crucial properties of the Maxwellian, which we will frequently use without reference:
\begin{equation}
	\nabla_q M = -M q,\quad \int_D M(q)|q|^r \dx{q} < \infty,\quad \sup_{q\in D}M^\alpha|q|^r <\infty,  \;\;\text{for any $\alpha>0$, $r\geq 0$}.
\end{equation}
The Maxwellian is a convenient weight for the solution of the Fokker--Planck equation, and so we will actually consider the normalised probability density function
\begin{equation}
	\hpsi := \frac{\psi}{M},\quad \hpsi_0 := \frac{\psi_0}{M}, 
\end{equation}
where $\psi_0$ is the initial value of $\psi$, i.e., $\psi(0,x,q)=\psi_0(x,q)$ for all $(x,q) \in \Omega \times D$.
Consequently, it is useful to rewrite the Fokker--Planck equation as an equation for $\hpsi$:
\begin{equation}
\label{eq:FPhat}
	M\partial_t \hpsi + Mu\cdot \nabla_x \hpsi + \div_q\prt*{(\nabla_x u) qM\hpsi} - \mu M\Delta_x\hpsi = \div_q\prt*{M\nabla_q\hpsi}.
\end{equation}
We supplement system~\eqref{eq:NSFP} with the following initial and boundary/decay conditions:
\begin{alignat}{2}
	u &= 0 &&\text{on $(0,T)\times\partial\Omega$},\\
	u(0,x) &= u_0(x)\quad &&\forall\, x\in\Omega,
\end{alignat}
and 
\begin{alignat}{2}
\left|M\brk*{\nabla_q\prt*{\frac{\psi}{M}} - (\nabla_x u) q\frac{\psi}{M}} \right| &\rightarrow 0\;\;\text{as $|q|\to\infty$} \quad &&\text{on $(0,T)\times \Omega$},\\
\nabla_x\psi \cdot \hat{n} &= 0 && \text{on $(0,T)\times\partial\Omega\times D$},\\
\psi(0,x,q) &= \psi_0(x,q) &&\forall\, (x,q)\in \Omega\times D,
\end{alignat}
where $\hat n$ denotes the unit outward normal vector to $\partial\Omega$. If $\psi_0$ is such that 
\begin{equation}
	\psi_0 \geq 0,\quad \int_D\psi_0(x,q)\dx{q} = 1\;\; \text{for a.e. $x\in\Omega$},
\end{equation}
the above boundary/decay conditions guarantee that these properties are propagated in time.

Let us also recall the following Oldroyd-B model with stress diffusion:
\begin{equation}
\begin{split}
\label{eq:OB}
	\partial_t u + \prt*{u\cdot \nabla_x}u - \nu \Delta_x u + \nabla_x p &= \div_x \tau,\\
	\div_x u &=0,\\
	\partial_t \tau + \prt*{u\cdot\nabla_x}\tau - \brk*{(\nabla_x u) \tau + \tau (\nabla_x^{\textrm{T}} u)} - \mu\Delta_x\tau + 2\tau &= 2 D(u),
\end{split}
\end{equation}
posed on $(0,T) \times \Omega$,
which, at least formally, is the macroscopic closure of the Navier--Stokes--Fokker--Planck system \eqref{eq:NSFP}, upon multiplying the Fokker--Planck equation \eqref{eq:NSFP}$_3$ by $q\otimes q$ and integrating over $q \in D$. Note, however, that this assertion has not been rigorously proved for weak solutions, except in two space dimensions (cf.~\cite{BaSu18}). Here we denote by $D(u):=\frac12((\nabla_x u) + (\nabla_x^{\mathrm T} u))$ the symmetric velocity gradient.
For further reference, we introduce the conformation tensor $\sigma:=\tau + \Id$, which satisfies the following evolution equation
\begin{equation}
\label{eq:OB-sigma}
		\partial_t \sigma + \prt*{u\cdot\nabla_x}\sigma - \brk*{(\nabla_x u) \sigma + \sigma (\nabla_x^{\mathrm T} u)} - \mu\Delta_x\sigma + 2\prt*{\sigma-\Id} = 0.
\end{equation}
To make a clear distinction between solutions 
to~\eqref{eq:OB}$_3$ and~\eqref{eq:OB-sigma} 
on the one hand, and the tensors $\sigma$ and $\tau$ defined through a solution $\hpsi$ to the system~\eqref{eq:NSFP} via the (formal) macroscopic closure described above, we denote the latter by, respectively,
\begin{equation}
	\sigma(\hpsi):=\int_D M\hpsi\, (q\otimes q) \dx{q},\quad \tau(\hpsi):=\sigma(\hpsi) - \prt*{\int_D M\hpsi \dx{q}} \Id.
\end{equation}

\medskip

The main difficulty in proving existence results for the Hookean dumbbell model~\eqref{eq:NSFP} is the lack of sufficiently strong \emph{a priori} bounds on the extra stress tensor $\tau(\hpsi)$. 
The global existence of weak solutions is known only in two spatial dimensions:\ it was proved by Barrett and S\"{u}li~\cite{BaSu18} (for smooth enough initial data) through a rigorous derivation of the macroscopic closure of the NSFP model, i.e., by rigorously establishing the connection of the NSFP model to the macroscopic Oldroyd-B model; see~\cite{BaSu18}. This result was recently extended to a larger class of data by La~\cite{La2020}, who introduced the concept of moment solutions to the Fokker--Planck equation. 

Choosing a suitable nonlinear spring potential in~\eqref{eq:SpringPotential} instead of the linear Hookean spring potential leads to a more satisfactory existence theory. A large body of literature is by now available in the case of finitely-extensible nonlinear elastic (FENE) type models; see, for instance,~\cite{BaSu2011, BaSu2012, BaSu2012a, BaSu2012b, Ma2013, Ma2011, JoLeDB}. In the classical FENE model due to Warner \cite{Warner} the configuration domain $D$ is a bounded open ball in $\mathbb{R}^d$ of radius $\sqrt{b}$, with $b>2$, centred at the origin, and the elastic spring-potential is given by
\begin{equation}
	U(s) = -\frac{b}{2}\ln\prt*{1-\frac{2s}{b}},\quad s\in \left[0,\frac{b}{2}\right).
\end{equation}
Thus, as the function $q\mapsto U(|q|^2/2)$ and the associated spring-force both blow up at the boundary $\partial D$ of the configuration domain $D$, polymer chains are prevented from exhibiting indefinite stretching.
Global existence of large-data weak solutions to a general class of FENE-type models with centre-of-mass diffusion, $\mu>0$, and the equilibration of weak solutions to the model as $t \rightarrow +\infty$, were proved by Barrett \& S\"uli in \cite{BaSu2011}. The global existence of weak solutions for the nondiffusive (i.e., when $\mu=0$) FENE model was proved in the paper of Masmoudi~\cite{Ma2013}. It is not currently  known whether weak solutions of the model with stress diffusion ($\mu>0$) converge, in any sense, to weak solutions of the nondiffusive model as $\mu \rightarrow 0_+$. Existence results 
are also available for diffusive Hookean-type models, where the Maxwellian is super-Gaussian, i.e., when $D=\mathbb{R}^d$ and $M$ exhibits more rapid decay to $0$ as $|q| \rightarrow \infty$ than the Gaussian Maxwellian associated with the standard Hookean model considered herein; see~\cite{BaSu2012}.

Another variant of~\eqref{eq:NSFP}, which has been widely studied in the mathematical literature, is the corotational model. It arises by replacing the velocity gradient $\nabla_x u$ in the $q$-convective term in the Fokker--Planck equation by the vorticity tensor 
$\omega(u):=\frac12((\nabla_x u) - (\nabla_x^{\mathrm T} u))$. In this case, one can derive stronger \emph{a priori} estimates on the probability density function $\hpsi$ than in the noncorotational case studied here.
In the presence of centre-of-mass diffusion, the global existence of weak solutions was shown for a general class of corotational FENE models in~\cite{BaSu2008}, and for the corotational Hookean model in~\cite{DebiecSuli}. In the nondiffusive case, $\mu=0$, global well-posedness is known only in two space dimensions~\cite{Ma2011, MasmoudiZhangZhang}.

Macroscopic models for polymeric flows are also of significant interest. The global existence of weak solutions to the two-dimensional Oldroyd-B model~\eqref{eq:OB} with stress diffusion was shown by Barrett and Boyaval in~\cite{barrett-boyaval-09}. Global regularity was then shown by Constantin and Kliegl in~\cite{CK12}. For the nondiffusive case, local well-posedness was shown in critical Besov spaces by Chemin and Masmoudi~\cite{CheminMasmoudi} under a Beale--Kato--Majda type assumption. Lions and Masmoudi~\cite{LiMa2000} proved global existence of weak solutions to the corotational Oldroyd-B model.

Mindful of the difficulties associated with proving the global existence of weak solutions to~\eqref{eq:NSFP}, we propose to relax the definition of weak solution and consider a class of \emph{generalised} solutions, referred to as \emph{generalised dissipative solutions}. The main idea is to allow for a defect to develop in the Navier--Stokes equation -- mathematically, this entails substituting the extra-stress tensor, $\tau$, by $\tau + m_{NS}$, where $m_{NS}$ is a tensor-valued measure. This approach was recently suggested by Barrett and S\"uli in~\cite{BaSu18}, who showed the global existence of weak \emph{subsolutions} to the Hookean Navier--Stokes--Fokker--Planck model in three dimensions. Here, we revisit the idea of~\cite{BaSu18} giving a different proof. Crucially, we are able to identify a \emph{dissipation defect} measure, which dominates the aforementioned measure $m_{NS}$ and satisfies an appropriate energy inequality -- namely, the energy inequality associated with the macroscopic Oldroyd-B model. Armed with this energy inequality we are able to prove a compatibility result: any generalised dissipative solution with appropriate regularity is in fact a weak solution.
Altogether, our framework of generalised dissipative solutions is similar to the concept of dissipative solutions introduced in~\cite{AbFe, AbFeNo} in the context of incompressible and compressible viscous fluids, as well as to the dissipative measure-valued solutions; see, for instance,~\cite{FGSW2016}. 
One particularly interesting open problem is to investigate whether our generalised dissipative solutions (or a subclass of these) satisfy the weak-strong uniqueness property, i.e., whether they necessarily coincide with the strong solution emanating from the same initial data. The usual tool in the quest for such a result is to establish a relative energy/entropy inequality.
This, however, seems to be a particularly challenging task, even on a formal level (again, because of the lack of satisfactory information concerning the derivatives of $\hpsi$), and is left for future work.

\medskip\noindent{{\it Outline of the paper.}} In the rest of this section we introduce the necessary notation and function spaces, recall the definition of \textit{weak solution} to the Navier--Stokes--Fokker--Planck system~\eqref{eq:NSFP}, and formally explain the associated energy inequality. Then, in Section~\ref{sec:GeneralisedSolutions} we define \emph{generalised dissipative solutions} to~\eqref{eq:NSFP}. For the Fokker--Planck equation this does not differ from the definition of a weak solution, but the Navier--Stokes equation contains a potentially nontrivial defect measure accounting for possible lack of compactness in the extra stress tensor $\tau$. We postulate however that a macroscopic energy-type inequality is satisfied. This inequality also contains a defect due to the presence of a defect in the stress tensor (or, more precisely, the conformation tensor), which dominates the defect present in the momentum equation. This property allows us to derive a conditional regularity property: any generalised dissipative solution, which enjoys suitable additional regularity of the velocity field, is in fact a weak solution of the NSFP system. This is proved in Section~\ref{sec:Conditional}. In Section~\ref{sec:MainResults} we state the main results of this paper. In Section~\ref{sec:Existence} we prove global existence of generalised dissipative solutions for any number of space dimensions $d\geq 2$. We first truncate the problem in two ways: we restrict the configuration space to a bounded domain, and we truncate the probability density function in the $q$-convective term in order to guarantee its boundedness. The right-hand side of the Navier--Stokes momentum balance equation is modified accordingly in order to preserve the energy identity for the truncated system. We then employ a two-step Galerkin approximation to show the existence of solutions to the truncated system, following which we remove the truncations. In Section~\ref{sec:Existence2d} we establish additional regularity of solutions for the two-dimensional case. In particular, we derive an evolution equation for the strong limit of the sequence of approximate conformation tensors. We cannot guarantee that this limit coincides with the conformation tensor defined by the solution $\hpsi$ of the Fokker--Planck equation, but we can ensure that the macroscopic equation for the stress tensor is the Oldroyd-B equation. 
Furthermore, we can use that equation to formulate a simple condition when the generalised dissipative solution whose existence we have shown is in fact a weak solution of the NSFP system, by showing that the defect measure in the stress tensor vanishes. For instance, this can be achieved by assuming some additional regularity of the initial data. In this way, we extend and improve the existence result from~\cite{BaSu18}, as we require less regular initial data.

\subsection{Preliminaries}

Before stating the definition of a weak solution to the system~\eqref{eq:NSFP}, let us introduce the necessary notation regarding the function spaces we will be working with. Let $\Omega\subset \R^d$ be a bounded open set with a Lipschitz boundary, and let $D=\R^d$ be the space of elongations/configurations. Let
\begin{align}
	&L^2_{\mathrm{div}}(\Omega;\R^d) :=\set*{v \in L^2(\Omega;\R^d)\;|\; \div_x v =0} ,\quad V := \set*{ v \in H_0^1(\Omega;\R^d)\;|\; \div_x v = 0 },
\end{align}
and let $L^p_M(\Omega\times D)$, $p\in [1,\infty)$, denote the Maxwellian-weighted $L^p$ space with the norm
\begin{equation}
	\norm{\phi}_{L^p_{M}(\Omega\times D)} := \prt*{\intOD M|\phi|^p\dx{q}\dx{x}}^{1/p}.
\end{equation}
Similarly, we introduce the Maxwellian-weighted Sobolev space $H^1_M(\Omega\times D)$ and its norm by
\begin{align}
	&H^1_M(\Omega\times D) := \set*{ \phi \in L^1_{\mathrm{loc}}(\Omega\times D)\; |\; \norm{\phi}_{H^1_M(\Omega\times D)} <\infty},\\
	&\norm{\phi}_{H^1_M(\Omega\times D)}:=\set*{\intOD M\brk*{|\phi|^2 + |\nabla_x\phi|^2 + |\nabla_q\phi|^2} \dx{q}\dx{x}}^{1/2}.
\end{align}
Of course, $L^p(\Omega\times D)$ and $H^1(\Omega\times D)$ are continuously embedded into the Maxwellian-weighted spaces
$L^p_M(\Omega\times D)$ and $H^1_M(\Omega\times D)$, respectively. We recall from~\cite{BaSu2012a} the following density and compact embedding results:
\begin{align}
	& C^\infty(\overline{\Omega};C_c^\infty(D))\; \text{is dense in $H_M^1(\Omega\times D)$},\\
	& H_M^1(\Omega\times D) \hookrightarrow\hookrightarrow L^2_{M}(\Omega\times D).
\end{align} 
We will also make use of the fractional-order Sobolev space $W^{1/2,4/3}_n(\Omega;\R^{d\times d})$ defined as follows. Let $W^{1/2,4/3}(\Omega;\R^{d\times d})$ denote the space of those $\zeta\in L^{4/3}(\Omega;\R^{d\times d})$ for which
\begin{equation}
    \norm{\zeta}_{W^{1/2,4/3}(\Omega;\R^{d\times d})} := \norm{\zeta}_{L^{4/3}(\Omega;\R^{d\times d})} + \prt*{\intO\intO \frac{|\zeta(x)-\zeta(y)|^{4/3}}{|x-y|^{d+2/3}}\dx{y}\dx{x}}^{3/4} < \infty.
\end{equation}
Then $W^{1/2,4/3}_n(\Omega;\R^{d\times d})$ is defined as the completion of $\{\zeta\in C^\infty(\overline\Omega;\R^{d\times d})\; |\; \nabla_x\zeta\cdot n = 0\;\text{on}\;\partial\Omega\}$ in the norm of $W^{1/2,4/3}(\Omega;\R^{d\times d})$.

Given a normed space $X$, we denote by $M^{-1}X$ the space of all those functions $\phi$ for which $M\phi \in X$. By $X'$ we denote the continuous dual of the space $X$.
Given a time $t \geq 0$, we define $\Omega_t := [0,t] \times\Omega$.

Let $\mathcal{F}:[0,\infty)\to [0,\infty)$ be defined by
\begin{equation}
	\F(0) := 1,\quad \F(s) := s\ln{s} - s + 1\quad \mbox{for $s \in (0,\infty)$}.
\end{equation}
Note that $\F$ is a nonnegative, strictly convex function with $\F'(s)=\ln s$ and $\F''(s)=1/s$ for $s \in (0,\infty)$, and $\F(s)/s \rightarrow \infty$ as $s \rightarrow \infty$.

By $\mathcal{M}^{+}(\overline\Omega;\R^{3\times 3}_{\mathrm{sym}})$ we denote the space of matrix-valued Borel measures on $\overline\Omega$, which take values in the cone of symmetric positive semidefinite matrices. We shall generally identify this space with a subspace of the dual of the space of continuous functions defined on $\overline\Omega$; by $\skp*{\cdot,\cdot}$ we then denote the corresponding duality pairing.
Similarly, the duality pairing between the Sobolev space $H^1(\Omega \times D)$ and its dual space $[H^1(\Omega \times D)]'$ will be denoted by $\langle \cdot, \cdot \rangle_{([H^1(\Omega \times D)]', H^1(\Omega \times D))}$; we shall use an analogous notation 
for the duality pairing between $H^1_M(\Omega \times D)$ and $[H^1_M(\Omega \times D)]'$, as well as for
the duality pairing between $H^1(\Omega)$ and $[H^1(\Omega)]'$.

\subsection{Weak solutions}

\begin{definition}[Weak solution of NSFP]
\label{def:NSFPweak}
	Let $d\in\{2,3\}$, and let $(u_0,\hpsi_0)$ be initial data such that
	\begin{align}
		u_0 \in L^2_{\mathrm{div}}(\Omega;\R^d),\;\; & \;\F(\hpsi_0) \in L_M^1(\Omega\times D),\\
		\hpsi_0(x,q)\geq 0\quad \mbox{for a.e. $(x,q) \in \Omega \times D$},& \;\;\;\int_D M(q)\hpsi_0(x,q)\dx{q} = 1\;\;\; \text{for a.e. $x\in\Omega$}.
	\end{align}
	A pair $(u,\hpsi)$ is a weak solution of the Navier--Stokes--Fokker--Planck system if
	\begin{align}
			u &\in C_{\mathrm{weak}}([0,T];L^2(\Omega;\R^d)) \cap L^2(0,T;V),\\
		\hpsi &\in L^1(0,T;L_M^1(\Omega\times D))\cap H^1(0,T; M^{-1}[H^s(\Omega\times D]'),\;\; s>d+1,\\
        &\quad\hpsi(t,x,q) \geq 0 \quad \mbox{for a.e. $(t,x,q) \in (0,T) \times \Omega \times D$},\\
&\quad \int_D M(q)\hpsi(t,x,q) \dx{q}  = 1\;\;\; \text{for a.e. $(t,x)\in (0,T)\times\Omega$},\\
		\mathcal{F}(\hpsi) &\in L^\infty(0,T; L^1_M(\OD)),\quad \sqrt{\hpsi} \in L^2(0,T;H^1_M(\OD)),
	\end{align}
and the equations are satisfied in the following weak sense for each $t\in[0,T]$:
	\begin{align}
		\intOT & u\cdot \partial_t\vartheta \dx{x}\dx{s} + \intOT \prt*{u\otimes u}:\nabla_x\vartheta \dx{x}\dx{s} - \nu \intOT \nabla_x u : \nabla_x \vartheta \dx{x}\dx{s}\\
		&= \intOT \tau : \nabla_x\vartheta\dx{x}\dx{s} + \intO u(t)\cdot\vartheta(t)\dx{x} - \intO u_0\cdot \vartheta(0) \dx{x}\\
		&\forall\, \vartheta \in \set*{ \vartheta \in L^2(0,T;V) \;|\; \partial_t\vartheta\in L^1(0,T;L^2(\Omega;\R^d))},
	\end{align}
and
	\begin{align}
		\intOT\int_D & M\hpsi\partial_t\phi \dx{q}\dx{x}\dx{s} + \int_{\Omega\times D} M\hpsi_0 \phi(0)\dx{q}\dx{x} - \int_{\Omega\times D} M\hpsi(t)\phi(t)\dx{q}\dx{x}\\
	& = -\intOT\int_D M\hpsi u\cdot \nabla_x\phi\dx{q}\dx{x}\dx{s} - \intOT\int_D M\hpsi \prt{(\nabla_x u) q} \cdot \nabla_q \phi\dx{q}\dx{x}\dx{s} \\
	&\quad + \mu\intOT\int_D M\nabla_x\hpsi\cdot\nabla_x\phi\dx{q}\dx{x}\dx{s}
	+ \intOT\int_D M\nabla_q\hpsi\cdot\nabla_q\phi\dx{q}\dx{x}\dx{s}\\
	&\quad \forall\, \phi\in W^{1,1}(0,T;H^s(\Omega\times D))
	\end{align}	
for $s>d+1$. 
\end{definition}
 
\begin{definition}[Weak solution of the Oldroyd-B system in 2D]
	\label{def:OldBweak}
	Let $d = 2$, and let $(u_0,\sigma_0)$ be initial data such that
	\begin{equation}
		u_0 \in L^2_{\mathrm{div}}(\Omega;\R^2),	\quad \sigma_0 \in L^2(\Omega;\R^{2\times 2}),\quad \sigma_0^{\mathrm{T}} = \sigma_0.
	\end{equation}
	A pair $(u,\sigma)$ is a weak solution of the Oldroyd-B system if
	\begin{align}
			u &\in C_{\mathrm{weak}}([0,T];L^2(\Omega;\R^2)) \cap L^2(0,T;V),\\
			\sigma &\in C_{\mathrm{weak}}([0,T];L^2(\Omega;\R^{2\times 2})) \cap L^2(0,T;H^1(\Omega;\R^{2\times 2})),
	\end{align}
	and the equations are satisfied in the following weak sense for each $t\in[0,T]$:
	\begin{align}
		\intOT & u\cdot \partial_t\vartheta \dx{x}\dx{s} + \intOT \prt*{u\otimes u}:\nabla_x\vartheta \dx{x}\dx{s} - \nu \intOT \nabla_x u : \nabla_x \vartheta \dx{x}\dx{s}\\
		&= \intOT \tau : \nabla_x\vartheta\dx{x}\dx{s} + \intO u(t)\cdot\vartheta(t)\dx{x} - \intO u_0\cdot \vartheta(0) \dx{x}\\
		&\forall\, \vartheta \in \set*{ \vartheta \in L^2(0,T;V) \;|\; \partial_t\vartheta\in L^1(0,T;L^2(\Omega;\R^2))},
	\end{align}
and
	\begin{equation}
	\label{eq:OBSigma}
	\begin{aligned}
	-&\int_0^t\intO \sigma : \partial_t\varphi \dx{x}\dx{s} - \intO \sigma_0 : \varphi(0) \dx{x} + \intO \sigma(t) : \varphi(t) \dx{x}\\
	&=-\int_0^t\intO\prt*{u\cdot\nabla_x}\sigma : \varphi\dx{x}\dx{s} - \mu\int_0^t\intO\nabla_x\sigma :: \nabla_x\varphi \dx{x}\dx{s} - 2\int_0^t\intO\prt*{\sigma-\Id} : \varphi\dx{x}\dx{s}\\
	&\;\;\;+\int_0^t\intO\sigma\prt*{\nabla_xu}^{\mathrm{T}} : \varphi \dx{x}\dx{s} + \int_0^t\intO(\nabla_x u) \sigma : \varphi \dx{x}\dx{s}\\
	&\forall\, \varphi\in W^{1,1}(0,T;H^1(\Omega;\R^{2\times 2})).
	\end{aligned}
	\end{equation}
\end{definition}

The question of existence of weak solutions to the Navier--Stokes--Fokker--Planck system defined in this way is open:\ it is completely unknown in the case of $d=3$ space dimensions whether weak solutions, thus defined, exist, while for $d=2$ it was shown by Barrett and S\"uli~\cite{BaSu18} that for initial data, smoother than what has been supposed above, weak solutions do exist globally in time. Their strategy was the following: thanks to earlier work by Barrett and Boyaval~\cite{barrett-boyaval-09}, it is known that there exists a global-in-time weak solution to the two-dimensional Oldroyd-B system~\eqref{eq:OB}. Assuming that $\partial\Omega\in C^{2,\alpha}$, $\alpha\in(0,1)$, and that the initial data additionally satisfy
\begin{equation}
	u_0 \in V,\quad \sigma(\hpsi_0) \in H^1_n(\Omega;\R^{2\times 2}),
\end{equation}
where $H^{1}_n(\Omega;\R^{2\times 2})$ is the completion of $\set*{\zeta\in C^\infty(\overline\Omega;\R^{2\times 2})\;|\; \nabla_x\zeta\cdot\hat n = 0\;\text{on}\;\partial\Omega}$ in the norm of $H^1(\Omega;\R^{2\times 2})$,
one can employ parabolic regularity results to infer additional smoothness of the associated weak solution to~\eqref{eq:OB}:
\begin{align}
	u &\in L^\infty(0,T;V) \cap L^2(0,T;H^2(\Omega;\R^2)) \cap L^{\frac43}(0,T;W^{1,\infty}(\Omega;\R^2)) \cap H^1(0,T;L^2_{\mathrm{div}}(\Omega;\R^2)),\\
	\sigma &\in C([0,T];H^1(\Omega;\R^{2\times 2})) \cap L^2(0,T;H^2(\Omega;\R^{2\times 2})) \cap H^1(0,T;L^2(\Omega;\R^{2\times 2})).
\end{align}
Using the extra regularity of the velocity field (in particular that  $\nabla_x  u\in L^{\frac43}(0,T; L^\infty(\Omega;\mathbb{R}^{2 \times 2})$), one can show global existence of a weak solution $\hpsi$ to the Fokker--Planck equation (via a discrete-in-time approximation and passing to the limit $\Delta t \rightarrow 0_+$ with the time-step). Finally, the identification $\sigma = \sigma(\hpsi)$ is possible, again thanks to the additional regularity of the velocity field. This programme cannot be repeated when $d=3$, because of the lack of proof of the existence of weak solutions to the Oldroyd-B system in three space dimensions. Instead, in~\cite{BaSu18}, the authors show existence of a generalised solution which includes the divergence (with respect to $x$) of a symmetric positive-semidefinite matrix-valued defect measure on the right-hand side of conservation of linear momentum equation \eqref{eq:NSFP}$_1$. However, because of the lack of an energy inequality involving the defect measure, they were unable to identify any conditions under which the defect measure would necessarily vanish, even in two space dimensions. Motivated by these incomplete results, we propose here an extended concept of solution to~\eqref{eq:NSFP}, \textit{generalised dissipative solution}, incorporating into the definition an energy inequality coming from the Oldroyd-B model. We shall then prove that such generalised dissipative solutions to~\eqref{eq:NSFP} exist, and this will enable us to identify sufficient conditions in both $d=2$ and $d=3$ space dimensions under which the defect measure featuring in the definition of generalised dissipative solution vanishes.

\subsection{Formal energy estimates}
There is an energy equality associated with~\eqref{eq:NSFP}. The following calculations are formal, but they will be justified rigorously for the Galerkin approximations of generalised dissipative solutions.
Testing the Navier--Stokes equation with $u$ and the Fokker--Planck equation with $\F'(\hpsi)$, we obtain, respectively
\begin{equation}
\label{eq:NSenergy}
	\ddt \intO \frac12 |u|^2\dx{x} + \nu \intO |\nabla_x u|^2\dx{x} = -\intO \tau : D(u)\dx{x},
\end{equation}
and
\begin{align}
\label{eq:FPenergy}
	\ddt \int_{\Omega\times D}  M\F(\hpsi)\dx{q}\dx{x} &+4\mu\int_{\Omega\times D} M\abs*{\nabla_x\sqrt{\hpsi}}^2 \dx{q}\dx{x} + 4\int_{\Omega\times D} M\abs*{\nabla_q\sqrt{\hpsi}}^2\dx{q}\dx{x}\\
	 &= \intOD M\hpsi \prt*{(\nabla_x u) q}\cdot \nabla_q\F'(\hpsi)\dx{q}\dx{x}\\
	 &= \intOD M \prt*{(\nabla_x u) q}\cdot \nabla_q\hpsi\dx{q}\dx{x}\\
	 &= -\intOD \hpsi\prt*{(\nabla_x u) q}\cdot \nabla_q M\dx{q}\dx{x}\\
	 &= \intOD M\hpsi (q\otimes q) : \nabla_x u \dx{q}\dx{x}\\
	 &= \intO \tau : D(u) \dx{x},
\end{align}
where we have used the observation that $\div_q\prt*{(\nabla_x u) q} = 0$.
Thereby, we obtain the energy equality
\begin{equation}
\begin{split}
\label{eq:NSFPenergy}
	\ddt &\prt*{\intO \frac12 |u|^2\dx{x}  + \int_{\Omega\times D} M\F(\hpsi)\dx{q}\dx{x}} + \nu \intO |\nabla_x u|^2\dx{x}\\
	 &+4\mu\int_{\Omega\times D} M\abs*{\nabla_x\sqrt{\hpsi}}^2 \dx{q}\dx{x} + 4\int_{\Omega\times D} M\abs*{\nabla_q\sqrt{\hpsi}}^2\dx{q}\dx{x} = 0.
\end{split}
\end{equation}

\medskip
There is also a formal energy equality associated with the conformation tensor formulation of the Oldroyd-B system. Contracting equation~\eqref{eq:OB-sigma} with the identity tensor, we obtain
\begin{align}
	\ddt \intO \tr{\sigma} \dx{x} + 2\intO \tr{\prt*{\sigma-\Id}} \dx{x} &= \intO \brk*{(\nabla_x u) \sigma + \sigma (\nabla_x^{\mathrm{T}} u)} : \Id \dx{x}\\
	&= 2\intO \sigma : D(u)\dx{x} = 2\intO \tau : D(u)\dx{x}.
\end{align}
By combining this equality with~\eqref{eq:NSenergy}, we arrive at the formal energy identity
\begin{align}
\label{eq:OBenergy}
	\ddt \prt*{\intO \frac12 |u|^2\dx{x} + \frac{1}{2}\intO \tr{\sigma} \dx{x}} + \nu \intO |\nabla_x u|^2\dx{x} + \intO \tr{\prt*{\sigma-\Id}} \dx{x} = 0.
\end{align}

\begin{remark}
	Equality~\eqref{eq:OBenergy} provides a priori estimates for the Oldroyd-B system only as long as $\tr\sigma > d$, or $\det\sigma >1$. As shown in~\cite{HuLe2007}, this property is propagated from the initial state (at least for smooth solutions). Since in our construction of generalised solutions we rely on the entropy identity~\eqref{eq:NSFPenergy} as the source of uniform bounds, we shall not make any such assumption on the initial conformation tensor, accepting to be lax about the term ``energy'' in reference to~\eqref{eq:OBenergy}.
\end{remark}

\begin{remark}
	We note that one cannot infer a formal energy equality involving the function $\hpsi$ in the $L^2_M(\OD)$ norm. 
The only formal energy equality involving $\hpsi$ is the equality \eqref{eq:NSFPenergy} stated above. We find it remarkable though that the energy equality~\eqref{eq:NSFPenergy} holds, particularly because the term
	\begin{equation}
	 -\intO \tau : D(u) \dx{x},	
	 \end{equation}
appearing in the process of its derivation, is not well defined unless additional regularity of either $\hpsi$ or $u$  is known. A similar, but even worse, problem arises in the Oldroyd-B model in three spatial dimensions, where it is not clear whether the term
	\begin{equation}
	\intOT \prt*{(\nabla_x u) \tau} : \varphi \dx{x}\dx{s}
	\end{equation}	 
	featuring in the weak formulation of the model is well defined -- unless $\tau\in L^2(0,T;L^2(\Omega;\R^{d\times d}))$, which cannot be deduced from any known estimates when $d=3$. This is the main obstacle to proving the global existence of weak solutions to the Oldroyd-B model in three space dimensions.
\end{remark}

\section{Generalised dissipative solutions}
\label{sec:GeneralisedSolutions}

\begin{definition}[Generalised dissipative solutions]
\label{def:GenSol3d}
		Let $(u_0,\hpsi_0)$ satisfy
		\begin{equation}
		\label{eq:Data3d}
		\begin{split}
		u_0 \in L^2_{\mathrm{div}}(\Omega;\R^d),\;\; & \; \mathcal{F}(\hpsi_0) \in L_M^1(\Omega\times D),\\
		\hpsi_0(x,q)\geq 0\quad \mbox{for a.e. $(x,q) \in \Omega \times D$},& \;\;\;\int_D M(q)\hpsi_0(x,q)\dx{q} = 1\;\;\; \text{for a.e. $x\in\Omega$}.
	\end{split}
	\end{equation}
	A quadruple $(u,\hpsi, m_{NS}, m_{OB})$ is a generalised dissipative solution to the Navier--Stokes--Fokker--Planck system~\eqref{eq:NSFP} if
	\begin{equation}
		u \in C_{\mathrm{weak}}([0,T];L^2(\Omega;\R^d)) \cap L^2(0,T;V), 
	\end{equation}
	\begin{equation}
		\hpsi \in L^1(0,T;L_M^1(\Omega\times D))\cap H^1(0,T; M^{-1}[H^s(\Omega\times D]'),\;\; s>d+1,
	\end{equation}
        \begin{equation}
            \quad \mathcal{F}(\hpsi) \in L^{\infty} (0,T;L^1_M(\OD)), \quad \sqrt{\hpsi} \in L^2(0,T;H^1_M(\OD)),
        \end{equation}
        \begin{equation}	
        \hpsi(t,x,q) \geq 0 \quad \mbox{for a.e. $(t,x,q) \in (0,T) \times \Omega \times D$},
        \end{equation}
        \begin{equation}
        \int_D M(q)\hpsi(t,x,q) \dx{q}  = 1\;\;\; \text{for a.e. $(t,x)\in (0,T)\times\Omega$},\\
        \end{equation}
	\begin{equation}
		m_{NS} \in L^\infty(0,T;\mathcal{M}^{+}(\overline\Omega;\R^{d\times d}_{\mathrm{sym}})),\quad m_{OB} \in L^\infty(0,T;\mathcal{M}^+(\overline\Omega)), 
	\end{equation}
	and the following identities are satisfied for all $t\in[0,T]$:
	\begin{align}
	\label{eq:GeneralisedNS3D}	
			\intOT & u\cdot \partial_t\vartheta \dx{x}\dx{s} + \intOT (u \otimes u) : \nabla_x \vartheta \dx{x}\dx{s} - \nu \intOT \nabla_x u : \nabla_x \vartheta \dx{x}\dx{s}\\
		&= \intOT \sigma(\hpsi) : \nabla_x\vartheta\dx{x}\dx{s} + \int_0^t \skp*{m_{NS}(s), \nabla_x\vartheta}\dx{s} + \intO u(t)\cdot\vartheta(t)\dx{x} - \intO u_0\cdot \vartheta(0) \dx{x}\\
		&\forall\, \vartheta \in \set*{ \vartheta \in L^2(0,T;V) \cap L^1(0,T;C^1(\overline\Omega;\mathbb{R}^d))
				\;|\; \partial_t\vartheta\in L^1(0,T;L^2(\Omega;\R^d))};
	\end{align}
	for $s>d+1$,
	\begin{equation}
	\label{eq:GeneralisedFP3D}
	\begin{split}
	\intOT\int_D & M\hpsi\partial_t\phi \dx{q}\dx{x}\dx{s} + \int_{\Omega\times D} M\hpsi_0\phi(0)\dx{q}\dx{x} - \int_{\Omega\times D} M\hpsi(t)\phi(t)\dx{q}\dx{x}\\
	& = -\intOT\int_D M\hpsi u\cdot \nabla_x\phi\dx{q}\dx{x}\dx{s} - \intOT\int_D M\hpsi \prt{(\nabla_x u) q} \cdot \nabla_q \phi\dx{q}\dx{x}\dx{s} \\
	&\quad + \mu\intOT\int_D M\nabla_x\hpsi\cdot\nabla_x\phi\dx{q}\dx{x}\dx{s}
	+ \intOT\int_D M\nabla_q\hpsi\cdot\nabla_q\phi\dx{q}\dx{x}\dx{s}\\ 
	&\quad \forall\, \phi\in W^{1,1}(0,T;H^s(\Omega\times D));
	\end{split}
	\end{equation}
	and the following inequality holds for a.e.\ $t\in(0,T)$:
	\begin{equation}
	\label{eq:GeneralisedEI3D}
	\begin{split}
		\frac12\intO &|u(t)|^2\dx{x} + \frac12\intO\tr{\sigma(\hpsi(t))}\dx{x} + \skp*{m_{OB}(t),\mathbbm{1}_{\overline\Omega}}  + \nu\intOT |\nabla_x u|^2\dx{x}\dx{s}\\
		&+ \intOT\tr{\prt*{\sigma(\hpsi)-\Id}}\dx{x}\dx{s} + \int_0^t \skp*{m_{OB}(s),\mathbbm{1}_{\overline\Omega}}\dx{s}
		\leq \frac12\intO |u_0|^2\dx{x} + \frac12\intO \tr{\sigma(\hpsi_0)}\dx{x}.
	\end{split}
	\end{equation}
	Finally, we require that the following compatibility condition is satisfied by the defect measures $m_{NS}$ and $m_{OB}$:\ there exists a $\zeta\in L^1(0,T)$ such that, for a.e.\ $t\in (0,T)$,
	\begin{equation}
	\label{eq:GeneralisedCompatibility}
		\abs*{\skp*{m_{NS}(t), \nabla_x\vartheta}} \leq \zeta(t)\norm{\vartheta}_{C^1(\overline\Omega;\R^d)}\skp*{m_{OB}(t),\mathbbm{1}_{\overline\Omega}}\quad \forall\, \vartheta \in C^1(\overline\Omega;\R^d).
	\end{equation}
\end{definition}

\begin{remark}
	We note that since we are using divergence-free test functions for the balance of linear momentum equation \eqref{eq:GeneralisedNS3D}, we can freely replace $\tau(\hpsi)$ by $\sigma(\hpsi)$ in the weak formulation. Indeed,
	\begin{equation}
		\intOT \prt*{\int_D M\hpsi\dx{q}} \Id : \nabla_x\vartheta\dx{x}\dx{s} = \intOT \prt*{\int_D M\hpsi\dx{q}} \div_x \vartheta\dx{x}\dx{s} =0.
	\end{equation}
\end{remark}

\begin{remark}
	In the two-dimensional case, we will show existence in a class where the corrected conformation tensor $\bar\sigma := m_{NS} + \sigma(\hpsi)$ satisfies the stress evolution equation associated with the Oldroyd-B system, equation~\eqref{eq:OB-sigma}. In effect, the generalised solution in this case can be thought of as a weak solution of the Oldroyd-B system together with the probability density $\hpsi$ satisfying the Fokker--Planck equation.
\end{remark}

\section{Main results}
\label{sec:MainResults}
The main purpose of this paper is to establish global existence of generalised dissipative solutions as defined in Section~\ref{sec:GeneralisedSolutions}, and discuss some of their properties, including in particular instances when the defect measures appearing in the definition of generalised dissipative solution vanish.

\begin{theorem}[Existence]
\label{thm:Existence}
	Let $d\geq 2$. There exists at least one generalised dissipative solution to the system~\eqref{eq:NSFP} with $m_{OB} = \tr{m_{NS}}$.
	Additionally, this solution is entropy-dissipative in the sense that the following energy inequality is satisfied for a.e.\ $t\in(0,T)$:
    \begin{equation}
    \begin{split}
	\label{eq:LlogLenergy}
	\frac12&\intO |u(t)|^2\dx{x} + \intOD M\F(\hpsi(t))\dx{q}\dx{x} + \nu\intOT |\nabla_x u|^2\dx{x}\dx{s}\\
	 &\quad + 4 \int_0^t\intOD M\prt*{\mu\abs*{\nabla_x\sqrt{\hpsi}}^2 + \abs*{\nabla_q\sqrt{\hpsi}}^2}\dx{q}\dx{x}\dx{s}\\
	 & \leq  \frac12\intO |u_0|^2\dx{x} + \intOD M\F(\hpsi_0)\dx{q}\dx{x}.
	\end{split}
    \end{equation} 
    Furthermore, when $d=2$ and $\sigma(\hpsi_0) \in L^2(\Omega;\R^{2\times 2})$, $\sigma_0^{\mathrm{T}}=\sigma_0$, we also have
 \begin{equation}
     \sigma(\hpsi),\, m_{NS} \in L^q(0,T;L^p(\Omega;\R^{2\times 2})),\quad q=\frac{2p}{p-2},\quad 2\leq p <\infty.
 \end{equation}
	Moreover, the pair $(u, \bar\sigma := m_{NS}+\sigma(\hpsi))$ is a unique weak solution of the Oldroyd-B system according to Definition~\ref{def:OldBweak}.
\end{theorem}
Our motivation for the terminology \textit{entropy-dissipative} in Theorem \ref{thm:Existence} is that the second term in the first line of the inequality \eqref{eq:LlogLenergy} represents the relative entropy of the probability density function $\psi$ with respect to the Maxwellian, $M$, while the expression in the second line of \eqref{eq:LlogLenergy}, the so-called Fisher information, quantifies the dissipation of the relative entropy.

\begin{theorem}[Conditional regularity for $d=3$]
\label{thm:Regularity3D}
	Let $d=3$, and let $(u,\hpsi,m_{NS},m_{OB})$ be a generalised dissipative solution. Suppose that
     \[
    \partial_t u\in L^1(0,T;L^2(\Omega;\R^3))\quad \text{and}\quad \nabla_x u\in L^1(0,T;C(\overline\Omega;\R^{3\times 3}));
    \]
     then, $m_{NS}\equiv 0$, $m_{OB}\equiv 0$, and $(u,\hpsi)$ is a weak solution of the Hookean dumbbell model~\eqref{eq:NSFP}.
\end{theorem}

\begin{corollary}[Conditional regularity for $d=2$]
\label{thm:Regularity2D}
	Let $d=2$ and $\partial\Omega\in C^{2,\alpha}$, $\alpha\in(0,1)$. Suppose further that, in addition to the assumptions of Theorem~\ref{thm:Existence}, 
	$u_0\in V$ and $\hpsi_0$ is such that 
	$\sigma(\hpsi_0) \in W^{1/2,4/3}_n(\Omega;\R^{2\times 2})$. Then, any generalised 
	dissipative solution to~\eqref{eq:NSFP}, as constructed in Theorem~\ref{thm:Existence}, 
	satisfies $m_{NS}\equiv 0$ and $m_{OB} \equiv 0$. Therefore the pair $(u,\hpsi)$ is a weak solution to~\eqref{eq:NSFP} according to Definition~\ref{def:NSFPweak} and the pair $(u,\sigma(\hpsi))$ is the unique weak solution of the Oldroyd-B system according to Definition~\ref{def:OldBweak}.
\end{corollary}

\section{Existence of generalised dissipative solutions}
\label{sec:Existence}

The main difficulty in proving the global existence of generalised dissipative solutions, as defined in Section~\ref{sec:GeneralisedSolutions}, is to derive the equation satisfied by a suitable approximate conformation tensor and guarantee the energy estimate~\eqref{eq:OBenergy}. This information is not furnished by the arguments presented in~\cite{BaSu18}. Here, instead, we shall therefore develop an entirely different approach, which we find more transparent. We shall construct approximate solutions to a truncated (both in the $q$-convective term, and in the configuration domain) problem via a Galerkin method, and then remove the truncations. We adapt to our case the two-step Galerkin approximation used in~\cite{BuMaSu2013} in the context of implicitly constituted kinetic models.

\subsection{Galerkin approximation of a truncated problem}
The Hilbert space $V\cap H^{d+1}(\Omega;\R^d)$, equipped with the inner product of $H^{d+1}(\Omega;\R^d)$, is compactly and densely embedded in $L^2_{\mathrm{div}}(\Omega;\R^d)$. Therefore, by a version of the Hilbert--Schmidt Theorem (cf.\ Lemma 5.1 in \cite{FigueroaSuli2012}), there exists a countable set $\set*{w_i}\subset V\cap H^{d+1}(\Omega;\R^d)$ of eigenfunctions whose linear span is dense in $L^2_{\mathrm{div}}(\Omega;\R^d)$. Furthermore, these vectors form an orthonormal set in $L^2(\Omega;\R^d)$ and an orthogonal set in $H^{d+1}(\Omega;\R^d)$.

Similarly, the Hilbert space $H^1_M(\Omega\times D)\cap H^1(\Omega\times D) = H^1(\Omega\times D)$ (equipped with the standard inner product) is compactly and densely embedded in $L^2_M(\Omega\times D)$. Thus, there exists a countable set $\set*{\phi_i}$ of eigenfunctions in $H^1(\Omega\times D)$ whose linear span is dense in $L^2_M(\Omega\times D)$, which are orthogonal in the inner product of $H^1(\Omega\times D)$ and orthonormal in $L^2_M(\Omega\times D)$.

\medskip
Let $R>0$, choose a smooth function $\chi:[0,\infty)\to [0,1]$ such that
\begin{equation}
	\chi(s) = 1\;\;\text{for}\;\; s \in [0,1],\quad \chi(s) = 0\;\;\text{for}\;\; s>2,
\end{equation}
and define $\chi_R:[0,\infty)\to [0,1]$ by
\begin{equation}
	\chi_R(s) := \chi\prt*{\frac{s}{R}}.
\end{equation}
Using the same cut-off function $\chi$, we define, for a fixed $L>0$, the truncations
\begin{equation}
\label{eq:truncationT}
	T_L(s):= \int_0^s \chi_L(r)\dx{r},\quad \Lambda_L(s) := s\chi_L(s).
\end{equation}

\medskip
We now fix $n,m\in\mathbb{N}$, and pose the following Galerkin problem: find time-dependent coefficients $c_i^{m,n,R,L}$, $d_i^{m,n,R,L}$ such that the functions $u^{m,n,R,L}$ and $\hpsi^{m,n,R,L}$, defined by
\begin{equation}
\begin{split}
	u^{m,n,R,L}(t,x) &:= \sum_{i=1}^m c_i^{m,n,R,L}(t)w_i(x),\\
	\hpsi^{m,n,R,L}(t,x,q) &:= \sum_{i=1}^n d_i^{m,n,R,L}(t)\phi_i(x,q),
\end{split}
\end{equation}
solve
\begin{equation}
\label{eq:NSgalerkin}
\begin{aligned}
	\intO \partial_t u^{m,n,R,L} \cdot w_i \dx{x} - \intO (u^{m,n,R,L}\otimes u^{m,n,R,L}) : \nabla_x w_i \dx{x} &+ \nu\intO \nabla_x u^{m,n,R,L} : \nabla_x w_i \dx{x}\\
	 &= -\intO \tau^{m,n,R,L} : \nabla_x w_i \dx{x},
\end{aligned}
\end{equation}
for all $i=1,\ldots, m$ and a.e.\ $t\in(0,T)$; and
\begin{equation}
\begin{split}
\label{eq:FPgalerkin}
	&\intOD  M\partial_t \hpsi^{m,n,R,L} \phi_i \dx{q}\dx{x} - \intOD Mu^{m,n,R,L}\hpsi^{m,n,R,L} \cdot \nabla_x\phi_i\dx{q}\dx{x}\\
	 &- \intOD M\Lambda_L(\hpsi^{m,n,R,L})\chi_R(|q|)\ \prt*{(\nabla_x u^{m,n,R,L})q} \cdot \nabla_q \phi_i\dx{q}\dx{x}\\
	  &+ \mu\intOD M\nabla_x\hpsi^{m,n,R,L} \cdot \nabla_x \phi_i\dx{q}\dx{x}
	 + \intOD M\nabla_q\hpsi^{m,n,R,L} \cdot \nabla_q \phi_i\dx{q}\dx{x} = 0,
\end{split}
\end{equation}
for all $i=1,\ldots, n$ and a.e.\ $t\in(0,T)$, where
\begin{equation}
\label{eq:truncatedTau}
	\tau^{m,n,R,L} := \int_D M\chi_R(|q|) \nabla_q\prt*{T_L(\hpsi^{m,n,R,L})}\otimes q\dx{q},
\end{equation}
with initial data given by
\begin{align}
	u^{m,n,R,L}(0,x) & = u_0^m(x) := \sum_{i=1}^m (u_0,w_i)_{L^2(\Omega)} w_i(x),\\
	\hpsi^{m,n,R,L}(0,x,q) & = \hpsi_0^{n,L}(x,q) := \sum_{i=1}^n (T_L(\hpsi_0),\phi_i)_{L^2_M(\Omega\times D)} \phi_i(x,q).
\end{align}
Carath\'{e}odory's Existence Theorem (cf., for example, Theorem 1.1 of Chapter 2 in \cite{MR0069338}) provides local-in-time existence of such $u^{m,n,R,L}$, $\hpsi^{m,n,R,L}$ for any fixed $m,n,R,L$. The uniform bounds derived below over the maximal time-interval of local existence, $[0,T_*)$, say, then suffice to extend these solutions to the entire time interval $[0,T]$. We shall therefore assume below that this extension has already taken place and will write $T$ instead of $T_*$ throughout the next section.

\subsection{Passing to the limit $n\to\infty$}
Multiplying the $i^{th}$ equation of~\eqref{eq:NSgalerkin} by $c_i^{m,n,R,L}$ and summing over $i=1,\ldots, m$, we deduce that $u^{m,n,R,L}$ satisfies
\begin{equation}
\label{eq:NSenergy2}
	\ddt \intO \frac12 |u^{m,n,R,L}|^2\dx{x} + \nu \intO |\nabla_x u^{m,n,R,L}|^2\dx{x} = -\intO \tau^{m,n,R,L} : \nabla_x u^{m,n,R,L}\dx{x}.
\end{equation}
Integrating by parts in~\eqref{eq:truncatedTau} we have
\begin{align}
	\tau^{m,n,R,L}(t,x) = -\int_D MT_L(\hpsi^{m,n,R,L})\chi_R \Id \dx{q} &+ \int_D M\chi_R\ T_L(\hpsi^{m,n,R,L})\, (q\otimes q) \dx{q}\\
	 &- \int_D MT_L(\hpsi^{m,n,R,L})\,(\nabla_q\chi_R\otimes q) \dx{q},
\end{align}
whereby we have that
\begin{equation}
	\abs*{\tau^{m,n,R,L}} \leq CL.
\end{equation}
Consequently, using~\eqref{eq:NSenergy2} and Young's inequality, we deduce that
\begin{equation}
	\intO \frac12 |u^{m,n,R,L}(t,x)|^2\dx{x} + \frac{\nu}{2} \intOT |\nabla_x u^{m,n,R,L}(s,x)|^2\dx{x}\dx{s} \leq C(u_0, L).
\end{equation}
It then follows that the coefficients $c_i^{m,n,R,L}$ satisfy
\begin{equation}
	\sup_{t \in (0,T)}\,\sup_{i=1,\ldots,m}\; \abs*{c_i^{m,n,R,L}(t)} + \abs*{\frac{\dd c_i^{m,n,R,L}}{\dd t}(t)} \leq C(u_0, m, L).
\end{equation}
Consequently, by observing the continuous embedding $H^{d+1}(\Omega)\hookrightarrow W^{1,\infty}(\Omega)$, we have that
\begin{equation}\label{eq:uboundmLR}
	\norm*{\nabla_x u^{m,n,R,L}}_{L^\infty((0,T)\times\Omega;\R^{d\times d})} \leq C(m,L).
\end{equation}

Similarly, multiplying the $i^{th}$ equation of~\eqref{eq:FPgalerkin} by $d_i^{m,n,R,L}$ and summing over $i=1,\ldots, n$ shows that $\hpsi^{m,n,R,L}$ satisfies~
\begin{equation}
\begin{split}
\label{eq:FPenergy2}
\ddt &\intOD\frac12 M(\hpsi^{m,n,R,L})^2\dx{q}\dx{x} + \mu \intOD M|\nabla_x\hpsi^{m,n,R,L}|^2\dx{q}\dx{x} + \intOD M|\nabla_q\hpsi^{m,n,R,L}|^2\dx{q}\dx{x}\\
	&=\intOD M \Lambda_L(\hpsi^{m,n,R,L}) \chi_R(|q|) \ \prt*{(\nabla_x u^{m,n,R,L}) q}\cdot\nabla_q\hpsi^{m,n,R,L}\dx{q}\dx{x}\\
	&\leq \prt*{\intOD M(\hpsi^{m,n,R,L})^2 \chi_R(|q|) \ |\nabla_x u^{m,n,R,L}|^2 |q|^2 \dx{q}\dx{x}}^{\frac12}\prt*{\intOD M|\nabla_q\hpsi^{m,n,R,L}|^2\dx{q}\dx{x}}^{\frac12}\\
	&\leq C(m,L,R)\intOD \frac12 M(\hpsi^{m,n,R,L})^2\dx{q}\dx{x} +  \frac12\intOD M|\nabla_q\hpsi^{m,n,R,L}|^2\dx{q}\dx{x},
\end{split}
\end{equation}
where we have used \eqref{eq:uboundmLR} in the transition to the last line.
Therefore, using Gronwall's Lemma, we obtain the bounds
\begin{equation}
	u^{m,n,R,L} \inb L^\infty(0,T;L^2(\Omega;\R^d)) \cap L^2(0,T;V),\;\;\text{uniformly in $n$},
\end{equation}
and
\begin{equation}
	\hpsi^{m,n,R,L} \inb L^\infty(0,T;L^2_M(\Omega\times D)) \cap L^2(0,T;H^1_M(\Omega\times D)),\;\;\text{uniformly in $n$}.
\end{equation}
Here and in the rest of the paper, for a normed linear space $X$ and a sequence $(a_k)_{k \in \mathbb{N}} \subset X$, the notation $a_k \inb X$ signifies that $(a_k)_{k \in \mathbb{N}}$ is a bounded sequence in the norm of $X$. Using these bounds together with \eqref{eq:uboundmLR}, we then have that
\begin{align}
	\intOD M &\Lambda_L(\hpsi^{m,n,R,L})\chi_R(|q|) \ \prt*{(\nabla_xu^{m,n,R,L})q}\cdot \nabla_q \phi_i \dx{q}\dx{x}\\
	 &\leq C(m,L,R)\norm{\hpsi^{m,n,R,L}}_{L^2_M(\Omega\times D)}\norm{\nabla_q\phi_i}_{L^2_M(\Omega\times D;\R^d)},
\end{align}
which we then use to deduce the following $n$-uniform bound on the time derivative:
\begin{equation}
	\partial_t\hpsi^{m,n,R,L} \inb L^2(0,T;M^{-1}[H_M^1(\Omega\times D)]'),\;\;  \text{uniformly in $n$}.
\end{equation}
Furthermore, for each function $\xi\in L^2_M(\OD)$ we have
\begin{align}
	\norm{M\xi}_{[H_M^1(\Omega\times D)]'} &\leq \sup_{\phi\in H^1_M(\OD)}\frac{\abs*{\skp*{M\xi,\phi}_{\prt*{[H^1_M(\Omega \times D)]',H^1_M(\Omega \times D)}}}}{\norm{\phi}_{H_M^1(\Omega\times D)}}\\
	&\leq \sup_{\phi\in H^1_M(\OD)} \frac{\norm{M^{\frac12}\xi}_{L^2(\Omega\times D)}\norm{M^{\frac12}\phi}_{L^2(\Omega\times D)}}{\norm{\phi}_{L^2_M(\Omega\times D)}}\\
	&\leq \norm{\xi}_{L^2_M(\OD)}.
\end{align}
Therefore $L^2_M(\OD)$ embeds continuously into $M^{-1}[H^1_M(\OD)]'$. By the Aubin--Lions Lemma, with the triple of spaces $H^1_M(\OD) \hookrightarrow\hookrightarrow L^2_M(\OD)\hookrightarrow M^{-1}[H^1_M(\OD)]'$, we thus have that
\begin{equation}
\text{the sequence}\;\; (\hpsi^{m,n,R,L})_{n\in\mathbb{N}}\;\;\text{is precompact in}\;\; L^2(0,T;L^2_M(\OD)).
\end{equation}

\medskip
We can therefore deduce the existence of functions $c_i^{m,R,L}, \hpsi^{m,R,L}$ such that, along a suitable subsequence (which we shall never relabel),
	\begin{alignat}{2}
		c_i^{m,n,R,L} &\rightarrow c_i^{m,R,L}\;\; &&\text{weakly$^*$ in $W^{1,\infty}(0,T)$,} \\
		c_i^{m,n,R,L} &\rightarrow c_i^{m,R,L} &&\text{strongly in $C([0,T])$,} \\
		u^{m,n,R,L} &\rightarrow u^{m,R,L} &&\text{strongly in $C([0,T];V\cap H^{d+1}(\Omega;\R^d))$,}\\
		\hpsi^{m,n,R,L} &\rightarrow \hpsi^{m,R,L} &&\text{strongly in $L^2(0,T;L^2_M(\Omega\times D))$,} \\
		\hpsi^{m,n,R,L} &\rightharpoonup \hpsi^{m,R,L}\;\; &&\text{weakly in $L^2(0,T;H_M^1(\Omega\times D))$,}\\
		M\partial_t\hpsi^{m,n,R,L}&\rightharpoonup M\partial_t\hpsi^{m,R,L}\;\; &&\text{weakly in $L^2(0,T;[H_M^1(\Omega\times D)]')$,}
	\end{alignat}
where
\begin{equation}
	u^{m,R,L}(t,x):=\sum_{i=1}^m c_i^{m,R,L}(t)w_i(x).  
\end{equation} 
Let us define
\begin{equation}
	\tau^{m,R,L}(t,x) := \int_D M\chi_R(|q|)\nabla_q\prt*{T_L(\hpsi^{m,R,L})}\otimes q \dx{q}.
\end{equation}
Passing to the limit $n \rightarrow \infty$ in equations~\eqref{eq:NSgalerkin} and~\eqref{eq:FPgalerkin}, it is easy to verify that the functions $u^{m,R,L}$, $\hpsi^{m,R,L}$ satisfy
\begin{equation}
\begin{split}
\label{eq:NSgalerkin2}
	\intO \partial_t u^{m,R,L}\cdot w_i\dx{x} - \intO (u^{m,R,L}\otimes u^{m,R,L}) : \nabla_x w_i\dx{x} &+ \nu\intO \nabla_x u^{m,R,L} : \nabla_x w_i\dx{x} \\
	&= -\intO \tau^{m,R,L} : \nabla_x w_i\dx{x},
\end{split}
\end{equation}
for all $i=1,\ldots, m$ and a.e.\ $t\in(0,T)$; and
\begin{equation}
\begin{split}
\label{eq:FPgalerkin2}
	&\hspace{-5mm}\langle M \partial_t\hpsi^{m,R,L}, \phi \rangle_{([H^1_M(\Omega \times D)]',H^1_M(\Omega \times D))}
	\\
	 &=\intOD M\hpsi^{m,R,L} u^{m,R,L}\cdot \nabla_x\phi\dx{q}\dx{x}
	+ \intOD M\Lambda_L(\hpsi^{m,R,L})\chi_R\ \prt*{(\nabla_x u^{m,R,L})q} \cdot \nabla_q \phi\dx{q}\dx{x} \\
	&\;\; - \mu\intOD M\nabla_x\hpsi^{m,R,L}\cdot\nabla_x\phi\dx{q}\dx{x}
	- \intOD M\nabla_q\hpsi^{m,R,L}\cdot\nabla_q\phi\dx{q}\dx{x}\\
	&\quad \forall\, \phi\in H_M^1(\Omega\times D)\; \text{and a.e.\ $t\in(0,T)$}.
\end{split}	
\end{equation}
Furthermore,
\begin{equation}
\label{eq:FPWeakData}
	\intOD M\hpsi^{m,L,R}(0,x,q)\phi\dx q\dx x = \intOD MT_L(\hpsi_0(x,q))\phi\dx q\dx x\qquad \forall\, \phi \in L^2_M(\Omega \times D).
\end{equation}

Thanks to Theorem II.5.12 on p.99 of \cite{MR2986590}, 
\[ \langle M \partial_t\hpsi^{m,R,L}, \phi \rangle_{([H^1_M(\Omega \times D)]',H^1_M(\Omega \times D))} = \frac{\dd}{\dd t} \intOD M\hpsi^{m,R,L} \phi \dx{q}\dx{x}\quad \forall\, \phi \in H_M^1(\Omega \times D).\] 
Hence, by testing the Fokker--Planck equation~\eqref{eq:FPgalerkin2} with the function $\phi\equiv 1$ and integrating in time over $(0,t)$, we have that
\begin{equation}
\begin{aligned}\label{hpsibd11}
	\intOD M\hpsi^{m,R,L} (t) \dx{q}\dx{x} = \intOD M\hpsi^{m,R,L}(0)\dx{q}\dx{x} &= \intOD MT_L(\hpsi_0)\dx{q}\dx{x}\\
	 &\leq \intOD M\hpsi_0\dx{q}\dx{x} = |\Omega|. 
\end{aligned}	
\end{equation}
Testing~\eqref{eq:FPgalerkin2} with $\phi = \prt{\hpsi^{m,R,L}}_{-}:=\min{\prt{0,\hpsi^{m,R,L}}}$ and invoking again Theorem II.5.12 on p.99 of \cite{MR2986590}, we deduce that
\begin{equation}
\begin{aligned}
	\frac12 \ddt \intOD M\prt{\hpsi^{m,R,L}}_{-}^2 & \dx{q}\dx{x} + \mu\intOD M|\nabla_x\prt{\hpsi^{m,R,L}}_{-}|^2\dx{q}\dx{x} + \intOD M|\nabla_q\prt{\hpsi^{m,R,L}}_{-}|^2 \dx{q}\dx{x} \\
	&= \intOD M\chi_R \Lambda_L(\hpsi^{m,R,L}) ((\nabla_x u^{m,R,L}) q) \cdot \nabla_q \prt{\hpsi^{m,R,L}}_{-} \dx{q}\dx{x} \\
	&\leq \frac{1}{2} \intOD M|\nabla_q (\hpsi^{m,R,L})_{-}|^2\dx{q}\dx{x} + c(m,R,L) \intOD M\prt{\hpsi^{m,R,L}}_{-}^2\dx{q}\dx{x},
\end{aligned}
\end{equation}
whence
\begin{equation}
	\intOD M\prt{\hpsi^{m,R,L}}_{-}^2(t) \dx{q}\dx{x} \leq \mathrm{e}^{c(m,R,L)t} \intOD M\prt{\hpsi^{m,R,L}}_{-}^2(0)\dx{q}\dx{x} = 0,
\end{equation}
and so
\begin{equation}\label{hpsibd2}
	\hpsi^{m,R,L} \geq 0\;\;\text{a.e. on $(0,T) \times \Omega \times D$. \ }
\end{equation}
Finally, testing the Fokker--Planck equation with $\phi(x,q)=\phi(x)$, we get
\begin{align}\label{hpsibd12}
\begin{aligned}
	\left\langle \partial_t\prt*{\int_D M\hpsi^{m,R,L}\dx{q}} ,  \phi \right\rangle_{([H^1(\Omega)]',H^1(\Omega))} =& \intO \prt*{\int_D M\hpsi^{m,R,L}\dx{q}} u^{m,R,L}\cdot \nabla_x\phi \dx{x}\\
	 &\hspace{-1.8cm}- \mu\intO \nabla_x\prt*{\int_D M\hpsi\dx{q}}\cdot\nabla_x\phi \dx{x}\quad \forall\, \phi \in H^1(\Omega).
\end{aligned}
\end{align}
Thus, the \textit{polymer number density}, defined by 
\begin{equation}
\label{eq:MacroDensity1}
	\rho^{m,R,L}(t,x) := \int_D M(q)\hpsi^{m,R,L}(t,x,q)\dx{q},
\end{equation}
is a weak solution of the advection-diffusion equation
\begin{align}
\label{eq:EqnMacroDensity1}
	\partial_t \rho^{m,R,L} + \div_x\prt*{u^{m,R,L} \rho^{m,R,L}} - \mu\Delta_x\rho^{m,R,L} &= 0 \quad \mbox{on $(0,T) \times \Omega$},
\end{align}
subject to the boundary and initial conditions
\begin{align}
    \mu\nabla_x\rho^{m,R,L}\cdot \hat n& = 0\quad \text{on $(0,T)\times \partial\Omega$,}\\
    \rho^{m,R,L}(0,x) &= \int_D M(q) T_L(\hpsi_0(x,q))\dx{q}.
\end{align}
It follows from \eqref{hpsibd11} and \eqref{hpsibd2} on the one hand, and on the other by setting $\phi = \rho^{m,R,L}$ in \eqref{hpsibd12}, using that $\mathrm{div}_x  u^{m,R,L}=0$, and applying again Theorem II.5.12 on p.99 of \cite{MR2986590}, that
\begin{align}
	\sup_{t\in(0,T)}\norm*{\rho^{m,R,L}(t)}_{L^\infty(\Omega)} \leq \norm*{\rho^{m,R,L}(0)}_{L^\infty(\Omega)} \leq \esssup_{x\in\Omega} \int_D M\hpsi_0\dx{q} = 1,\\
	\int_0^T \norm*{\nabla_x\rho^{m,R,L}(t)}_{L^2(\Omega;\R^d)}^2 \dx{t} \leq \norm*{\rho^{m,R,L}(0)}_{L^2(\Omega)}^2 \leq |\Omega|.
\end{align}

\subsection{An entropy estimate}
Choose $\delta > 0$ and define, for $s \in [0,\infty)$, 
\begin{align}
	\mathcal{F}_\delta(s) := (s+\delta)\ln{\prt{s+\delta}} - (s+\delta) + 1
\quad \mbox{and} \quad
	T_{L,\delta}(s):= \int_0^s \frac{\Lambda_L(r)}{r+\delta} \dx{r} = \int_0^s \frac{r\chi_L(r)}{r+\delta} \dx{r}.
\end{align}
The function $\mathcal{F}_{\delta}'(\hpsi^{m,R,L}) = \ln{\prt{\hpsi^{m,R,L}+\delta}}$ is a valid test function for the Fokker--Planck equation. We then obtain
\begin{equation}
\begin{split}
\label{eq:FPdelta}
	\ddt \intOD M \mathcal{F}_\delta(\hpsi^{m,R,L})\dx{q}\dx{x} &+ \mu\intOD M\frac{1}{\hpsi^{m,R,L}+\delta}|\nabla_x\hpsi^{m,R,L}|^2 \dx{q}\dx{x}\\
	 & + \intOD M\frac{1}{\hpsi^{m,R,L}+\delta}|\nabla_q\hpsi^{m,R,L}|^2 \dx{q}\dx{x}\\
	 &\hspace{-4.5cm} = \intOD M\chi_R \Lambda_L(\hpsi^{m,R,L})\frac{1}{\hpsi^{m,R,L}+\delta} ((\nabla_x u^{m,R,L})q)\cdot \nabla_q \hpsi^{m,R,L} \dx{q}\dx{x}\\
	 &\hspace{-4.5cm} = \intOD M\chi_R((\nabla_x u^{m,R,L})q) \cdot \nabla_q\prt*{ T_{L,\delta}(\hpsi^{m,R,L})} \dx{q}\dx{x}\\
	 &\hspace{-4.5cm} = \intOD M\chi_R T_{L,\delta}(\hpsi^{m,R,L})\, (q\otimes q) : \nabla_x u^{m,R,L}\dx{q}\dx{x}\\
	 &- \intOD MT_{L,\delta}(\hpsi^{m,R,L})\,(\nabla_q\chi_R \otimes q) : \nabla_x u^{m,R,L}\dx{q}\dx{x}.
\end{split}
\end{equation}
Using the Monotone Convergence Theorem for the gradient terms on the left-hand side and the convergence $T_{L,\delta}\to T_L$ in $C([0,T])$ for the terms on the right-hand side of~\eqref{eq:FPdelta}, we deduce that, in the limit $\delta\to 0_+$,
\begin{equation}
\begin{split}
\label{eq:FPdelta2}
	\intOD M \mathcal{F}(\hpsi^{m,R,L})(t)\dx{q}\dx{x} &+ 4 \mu\int_0^t \intOD M\abs*{\nabla_x\sqrt{\hpsi^{m,R,L}}}^2\dx{q}\dx{x}\dx{s}\\
	 &+ 4 \int_0^t \intOD M\abs*{\nabla_q\sqrt{\hpsi^{m,R,L}}}^2\dx{q}\dx{x}\dx{s}\\
	 &\hspace{-4.5cm}= \intOD M \mathcal{F}(\hpsi^{m,R,L})(0)\dx{q}\dx{x} + \int_0^t \intOD M\chi_R T_{L}(\hpsi^{m,R,L}) \,(q\otimes q) : \nabla_x u^{m,R,L}\dx{q}\dx{x}\dx{s}\\
	 &- \int_0^t\intOD MT_{L}(\hpsi^{m,R,L})\, (\nabla_q\chi_R \otimes q) : \nabla_x u^{m,R,L}\dx{q}\dx{x}\dx{s}.
\end{split}
\end{equation}
Now, we multiply equation~\eqref{eq:NSgalerkin2} with $c_i^{m,R,L}$, sum over $i=1,\ldots,m$, and integrate by parts to deduce that
\begin{equation}
\begin{split}
\label{eq:NSdelta}
	\frac12\ddt \intO |u^{m,R,L}|^2\dx{x} &+ \nu\intO |\nabla_x^{m,R,L}|^2\dx{x} = -\intO \tau^{m,R,L} : \nabla_x u^{m,R,L}\dx{x}\\
	=& -\intOD M\chi_R T_L(\hpsi^{m,R,L})\, (q\otimes q) : \nabla_x u^{m,R,L}\dx{q}\dx{x}\\
	 &+ \intOD MT_L(\hpsi^{m,R,L})\,(\nabla_q\chi_R\otimes q) : \nabla_x u^{m,R,L}\dx{q}\dx{x}.
\end{split}
\end{equation}
Adding equations~\eqref{eq:FPdelta2} and~\eqref{eq:NSdelta}, we deduce the following entropy identity, valid for a.e.\ $t\in(0,T)$,
\begin{equation}
\begin{split}
\label{eq:GalerkinEntropy}
	\frac12\intO |u^{m,R,L}(t)|^2 \dx{x} &+ \intOD M\mathcal{F}(\hpsi^{m,R,L})(t)\dx{q}\dx{x} + \nu\int_0^t\intO |\nabla_xu^{m,R,L}|^2\dx{x}\dx{s}\\
	 &\hspace{-3cm}+ 4\mu\int_0^t \intOD M\abs*{\nabla_x\sqrt{\hpsi^{m,R,L}}}^2\dx{q}\dx{x}\dx{s} + 4 \int_0^t \intOD M\abs*{\nabla_q\sqrt{\hpsi^{m,R,L}}}^2\dx{q}\dx{x}\dx{s}\\
	 &= \frac12 \intO |u_0^m|^2\dx{x} + \intOD M \mathcal{F}(T_L(\hpsi_0)) \dx{q}\dx{x}.
\end{split}
\end{equation}
Notice that
\begin{align}
	\intOD M\mathcal{F}(T_L(\hpsi_0)) \dx{q}\dx{x} &= \int_{\{T_L(\hpsi_0)\leq e\}} M\mathcal{F}(T_L(\hpsi_0)) \dx{q}\dx{x} + \int_{\{T_L(\hpsi_0) > e\}}M\mathcal{F}(T_L(\hpsi_0)) \dx{q}\dx{x}\\
	&\leq \intOD M\dx{q}\dx{x} + \intOD M\mathcal{F}(\hpsi_0)\dx{q}\dx{x} \leq C.
\end{align}
Thus, the bounds implied by the entropy identity~\eqref{eq:GalerkinEntropy} are uniform in $m,R,L$.

One important consequence of the $L\log L$ bound on $\hpsi$ guaranteed by the entropy identity~\eqref{eq:GalerkinEntropy} is that it implies integrability of the second moment:
\begin{equation}
	\sup_{t\in(0,T)}\ \intOD M \hpsi^{m,R,L} |q|^2 \dx{q}\dx{x} \leq C,
\end{equation}
where the constant $C$ is independent of $m,R,L$. Indeed, the Fenchel--Young inequality (applied to the function $f(s)=\mathrm{e}^s$ and its convex conjugate $f^*(s)=(s\ln s)-s$ for $s \in [0,\infty)$) yields
\begin{equation}
	\hpsi^{m,R,L} \left(\frac14 |q|^2\right)  \leq \prt*{\hpsi^{m,R,L}\ln\hpsi^{m,R,L}} - \hpsi^{m,R,L} + \mathrm{e}^{\frac14|q|^2},
\end{equation}
which gives upon multiplication by $4M$ and integration over $\Omega \times D$ (recall that $M$ is a Gaussian and note that $f^*(s) \leq \mathcal{F}(s)$ for all $s \in [0,\infty)$)
\begin{equation}
\label{eq:FYbound}
	\intOD M \hpsi^{m,R,L}\,|q|^2 \dx{q}\dx{x} \leq 4\intOD M\mathcal{F}(\hpsi^{m,R,L})\dx{q}\dx{x} + \frac{4|\Omega|}{\mathcal{Z}}\int_D \mathrm{e}^{-\frac14 |q|^2}\dx{q}.
\end{equation} 

\subsection{Passing to the limit $L\to\infty$}
The entropy identity~\eqref{eq:GalerkinEntropy} implies that, in particular,
\begin{equation}
	\sup_{t\in(0,T)}\norm{u^{m,R,L}(t)}_{L^2(\Omega;\R^d)} + \int_0^T \norm{\nabla_xu^{m,R,L}(t)}^2_{L^2(\Omega;\R^{d\times d})}\dx{t} \leq C,
\end{equation}
which in turn implies that
\begin{equation}
	\sup_{t\in(0,T)}\sup_{i=1,\ldots,m}\; \abs*{c_i^{m,R,L}(t)} + \abs*{\frac{\dd c_i^{m,R,L}}{\dd t}(t)} \leq C(m).
\end{equation}
It follows that there exist functions $c_i^{m,R}$, $i=1,\ldots,m$, such that (along a subsequence again)
	\begin{alignat}{2}
		c_i^{m,R,L} &\rightarrow c_i^{m,R}\;\; &&\text{weakly$^*$ in $W^{1,\infty}(0,T)$,} \\
		c_i^{m,R,L} &\rightarrow c_i^{m,R} &&\text{strongly in $C([0,T])$,} \\
		u^{m,R,L} &\rightarrow u^{m,R} &&\text{strongly in $C([0,T];V\cap H^{d+1}(\Omega;\mathbb{R}^d))$,}
	\end{alignat}
where
\begin{equation}
	u^{m,R}(t,x):=\sum_{i=1}^m c_i^{m,R}(t)w_i(x).
\end{equation}

Next, we derive strong compactness of the sequence $\hpsi^{m,L,R}$ in $L^1(0,T; L^1_M(\Omega \times D))$. The following arguments are virtually the same as those in~\cite{BuMaSu2013}. From the $L$-uniform estimate
\begin{equation}
	\sup_{t\in(0,T)}\;\intOD M\hpsi^{m,R,L}\prt*{1+\ln{\hpsi^{m,R,L}}} \dx{q}\dx{x} \leq C,
\end{equation}
we infer that the sequence $(\hpsi^{m,R,L})_{L>0}$ is equi-integrable, and therefore there exists a function $\hpsi^{m,R}$ such that
\begin{equation}
	\hpsi^{m,R,L} \rightharpoonup \hpsi^{m,R}\;\;\text{weakly in $L^1(0,T;L^1_M(\OD))$}.
\end{equation}
We wish to `upgrade' this convergence to convergence in norm. To this end, let $\Omega_0\times D_0\subset\subset \OD$ be a Lipschitz domain. Since $M$ is bounded away from zero on $D_0$, we deduce that
\begin{equation}
\label{eq:LocalSobolevSqrt}
	\sup_{t\in(0,T)} \norm*{\sqrt{\hpsi^{m,R,L}}}_{L^2(\Omega_0\times D_0)}^2 + \int_0^T \norm*{\nabla_{x,q}\sqrt{\hpsi^{m,R,L}}}_{L^2(\Omega_0\times D_0;\R^{2d})}^2\dx{q}\dx{x}\dx{t} \leq C.
\end{equation}
Then, by function-space interpolation,
\begin{equation}
\label{eq:LocalSobolev}
	\norm*{\hpsi^{m,R,L}}_{L^{\frac{d+1}{d}}((0,T)\times\Omega_0\times  D_0)} \leq C \quad\text{and}\quad \norm*{\nabla_{x,q}\hpsi^{m,R,L}}_{L^{\frac{2d+2}{2d+1}}((0,T)\times\Omega_0\times  D_0;\R^{2d})} \leq C.
\end{equation}
Moreover, since
\begin{equation}
	\norm{\hpsi^{m,R,L}}_{L^\infty((0,T)\times \Omega;L^1_M(D))} \leq C,
\end{equation}
we see that, for each $q_1\in (1,\infty)$, there exists a $q_2>1$ such that
\begin{equation}
	\norm*{\hpsi^{m,R,L}}_{L^{q_1}((0,T)\times \Omega_0 ;L^{q_2}(D))} \leq C.
\end{equation}
Consequently, there exists a $\delta>0$ such that
\begin{equation}
	\norm*{u^{m,R,L}\hpsi^{m,R,L}}_{L^{1+\delta}((0,T)\times\Omega_0\times D_0;\R^d)} + \norm*{\Lambda_L(\hpsi^{m,R,L})(\nabla_xu^{m,R,L})q}_{L^{1+\delta}((0,T)\times\Omega_0\times D_0;\R^d)}\leq C.
\end{equation}
Now choose any $\alpha\in (0,\min{\prt{\frac12,\frac{\delta}{\delta+1}}})$ and define the following $(1+2d)$-component vector-functions:
\begin{align}
	H^{L} &:= \prt*{M\hpsi^{m,R,L}, M\hpsi^{m,R,L}u^{m,R,L}+M\nabla_x\hpsi^{m,R,L}, M\chi_R\Lambda_L(\hpsi^{m,R,L})\nabla_xu^{m,R,L}q + M\nabla_q\hpsi^{m,R,L}},\\
	Q^L &:= \prt*{\prt*{1+\hpsi^{m,R,L}}^\alpha, 0,\ldots,0}.
\end{align}
Using the above uniform bounds we can deduce that the sequences of these vector-functions converge weakly, respectively in $L^{1+\delta}((0,T)\times\Omega_0\times D_0)$ and $L^{\frac{1}{\alpha}}((0,T)\times\Omega_0\times D_0)$ to
\begin{align}
	H &:= \prt*{M\hpsi^{m,R}, M\hpsi^{m,R}u^{m,R}+M\nabla_x\hpsi^{m,R}, M\chi_R\hpsi^{m,R}\nabla_xu^{m,R}q + M\nabla_q\hpsi^{m,R}}\quad \mbox{and}\\
	Q &:= \prt*{\zeta^{m,R}, 0,\ldots,0}
\end{align}
as $L \rightarrow \infty$, where $\zeta^{m,R}$ denotes the weak limit of $\prt{1+\hpsi^{m,R,L}}^\alpha$.
From the Fokker--Planck equation~\eqref{eq:FPgalerkin2} it is clear that
\begin{equation}
	\div_{(t,x,q)} H^L = 0 \;\;\text{in $(0,T)\times \Omega_0\times D_0$},
\end{equation}
while the bound
\begin{align}
	\int_0^T\int_{\Omega_0\times D_0} |\nabla_{t,x,q}Q^L-(\nabla_{t,x,q}^{\mathrm T}Q^L)|^2 \dx{q}\dx{x}\dx{t} &\leq C \int_0^T\int_{\Omega_0\times D_0} |\nabla_{x,q}(1+\hpsi^{m,R,L})^\alpha|^2\dx{q}\dx{x}\dx{t}\\
	&\leq C \int_0^T\int_{\Omega_0\times D_0} \abs*{\nabla_{x,q}\sqrt{\hpsi^{m,R,L}}}^2\dx{q}\dx{x}\dx{t} \leq C
\end{align}
implies that the curl of $Q^L$ is precompact in $H^{-1}((0,T)\times\Omega_0\times D_0)$. We deduce by the Div-Curl Lemma that
\begin{equation}
	H^L \cdot Q^L \rightharpoonup H \cdot Q\;\; \text{weakly in $L^1((0,T)\times\Omega_0\times D_0)$},
\end{equation}
and in particular
\begin{equation}
	\hpsi^{m,R,L}(1+\hpsi^{m,R,L})^\alpha \rightharpoonup \hpsi^{m,R}\zeta^{m,R}
\end{equation}
as $L \rightarrow \infty$, which implies that
\begin{equation}
	(1+\hpsi^{m,R,L})^{\alpha+1} \rightharpoonup (1+\hpsi^{m,R})\zeta^{m,R}\quad \mbox{as $L \rightarrow \infty$}.
\end{equation}
Since the function $s\mapsto s^{\alpha+1}$ is strictly convex on $[1,\infty)$, we have
\begin{equation}
	(1+\hpsi^{m,R})^{\alpha+1} \leq (1+\hpsi^{m,R})\zeta^{m,R},\quad \text{and so}\quad (1+\hpsi^{m,R})^\alpha \leq \zeta^{m,R}.
\end{equation}
On the other hand, the function $s\mapsto s^\alpha$ is strictly concave on $[1,\infty)$, so we must have
\begin{equation}
\zeta^{m,R} := \overline{(1+\hpsi^{m,R,L})^\alpha} \leq (\overline{1+\hpsi^{m,R,L}})^\alpha = (1+\hpsi^{m,R})^\alpha,
\end{equation}
and therefore
\begin{equation}
  (1+\overline{\hpsi^{m,R,L}})^\alpha  = (1+\hpsi^{m,R})^\alpha = \zeta^{m,R} = \overline{(1+\hpsi^{m,R,L})^\alpha}.
\end{equation}
Then, by strict concavity of the function $s\mapsto s^\alpha$ on $[1,\infty)$, we deduce that, along a subsequence,
\begin{equation}
\label{eq:PointwiseAE}
	\hpsi^{m,R,L} \to \hpsi^{m,R}\;\; \text{a.e.\ in $(0,T)\times\Omega_0\times D_0$}\quad \mbox{as $L \rightarrow \infty$}.
\end{equation}
By a diagonal argument based on an increasing sequence of nested subsets of $\Omega \times D$, we can then replace $(0,T) \times \Omega_0\times D_0$ in \eqref{eq:PointwiseAE} by $(0,T) \times \Omega\times D$. Combining the resulting almost everywhere convergence $\hpsi^{m,R,L} \to \hpsi^{m,R}$ on $(0,T)\times\Omega_0\times D_0$ as $L \rightarrow \infty$ with the weak $L^1(0,T;L^1_M(\OD))$ convergence, and using Vitali's Theorem, we deduce that
\begin{equation}
	\hpsi^{m,R,L} \to \hpsi^{m,R} \;\; \text{strongly in $L^1(0,T;L^1_M(\OD))$}\quad \mbox{as $L \rightarrow \infty$}.
\end{equation}
As the sequence $\hpsi^{m,R,L}$ is uniformly bounded in $L^\infty(0,T;L^1_M(\Omega \times D))$ with respect to $m$, $R$ and $L$, by standard function-space interpolation we then have that
\begin{equation}\label{eq:strongLq}
	\hpsi^{m,R,L} \to \hpsi^{m,R} \;\; \text{strongly in $L^q(0,T;L^1_M(\OD))$}\quad \mbox{as $L \rightarrow \infty$,} \quad \text{for each $q\in[1,\infty)$}.
\end{equation}

Since $T_L$ is a Lipschitz function, it now follows by the Dominated Convergence Theorem that
\begin{equation}
	T_L(\hpsi^{m,R,L}) \to \hpsi^{m,R}\;\; \text{strongly in $L^1(0,T;L^1_M(\OD))$}\quad \mbox{as $L \rightarrow \infty$}.
\end{equation}
Hence, there is a subsequence such that for a.e.\ $t\in(0,T)$ and for all $i=1,\ldots, m$ we have
\begin{equation}
	\intO \tau^{m,R,L} : \nabla_x w_i \dx{x} \to \intO \tau^{m,R} : \nabla_x w_i \dx{x}\quad \mbox{as $L \rightarrow \infty$},
\end{equation}
where
\begin{align}
	\tau^{m,R} := \int_D M\chi_R \hpsi^{m,R}\, (q\otimes q)\dx{q} - \int_D M\hpsi^{m,R}\, (\nabla_q\chi_R\otimes q) \dx{q} - \int_D M\hpsi^{m,R}\chi_R \Id \dx{q}.
\end{align}
Therefore, we can pass to the limit $L\to\infty$ in the approximate Navier--Stokes equation~\eqref{eq:NSgalerkin2} to obtain
\begin{equation}
\begin{split}
\label{eq:NSgalerkin3}
	\intO \partial_t u^{m,R}\cdot w_i\dx{x} - \intO (u^{m,R}\otimes u^{m,R}) : \nabla_x w_i\dx{x} &+ \intO \nabla_x u^{m,R} : \nabla_x w_i\dx{x} \\
	&= -\intO \tau^{m,R} : \nabla_x w_i\dx{x},
\end{split}
\end{equation}
for all $i=1,\ldots, m$ and a.e.\ $t\in(0,T)$.

Next, we shall pass to the limit $L \rightarrow \infty$ in the Fokker--Planck equation~\eqref{eq:FPgalerkin2}; this, however, requires particular care. First, let us prove that
\begin{equation}
\label{eq:weakgradient1}
	M\nabla_{x,q} \hpsi^{m,R,L} \rightharpoonup M\nabla_{x,q}\hpsi^{m,R} \quad \text{weakly in $L^1((0,T)\times\Omega;L^1(D;\R^{2d}))$}
\quad \mbox{as $L \rightarrow \infty$}.
\end{equation}
To this end, we note that
\begin{equation}
\begin{aligned}
\label{eq:UniformL1gradient}
	\intOD M \abs*{\nabla_{x,q}\hpsi^{m,R,L}}\dx{q}\dx{x} &= 2\intOD M\sqrt{\hpsi^{m,R,L}}\ \abs*{\nabla_{x,q}\sqrt{\hpsi^{m,R,L}}}\dx{q}\dx{x}\\
	&\leq 2\prt*{\intOD M\abs*{\nabla_{x,q}\sqrt{\hpsi^{m,R,L}}}^2\dx{q}\dx{x}}^{1/2}\prt*{\intOD M\hpsi^{m,R,L}\dx{q}\dx{x}}^{1/2},
\end{aligned}
\end{equation}
whereby, upon integrating this inequality over $(0,T)$ and using the entropy bound, we deduce that $M\nabla_{x,q}\hpsi^{m,R,L}$ is uniformly bounded in $L^1((0,T) \times \OD;\R^{2d})$. The above inequality also shows that equi-integrability of this sequence follows from that of $\hpsi^{m,R,L}$. Recalling the local bound~\eqref{eq:LocalSobolev}, we can identify the weak limit. This then implies the desired weak convergence property stated in \eqref{eq:weakgradient1}.

Second, thanks to strong compactness of the velocity gradient, the strong convergence result \eqref{eq:strongLq}, and properties of the truncation $\Lambda_L$, we easily see that, as $L \rightarrow \infty$,
\begin{equation}
\label{eq:weakqdiffusion1}
	(\nabla_x u^{m,R,L})\, \Lambda_L(\hpsi^{m,R,L}) \to (\nabla_x u^{m,R})\, \hpsi^{m,R} \;\; \text{strongly in $L^1((0,T)\times \Omega ; L^1_M(D;\R^{d\times d}))$}.
\end{equation}

Third, we deduce a uniform bound on the time derivative of $\hpsi^{m,R,L}$. Using~\eqref{eq:UniformL1gradient} we see that $\nabla_{x,q}\hpsi^{m,R,L}$ is uniformly bounded in $L^2(0,T;L^1(\OD;\R^{2d}))$. Using this bound, the uniform control of the velocity in $L^\infty(0,T;L^2(\Omega;\R^d))\cap L^2(0,T;V)$ and the polymer number density in $L^\infty((0,T)\times\R^d)$, together with the continuous embedding $H^s(\OD)\hookrightarrow W^{1,\infty}(\OD)$, for $s>d+1$, one readily arrives at the following bound on the time derivative:
\begin{equation}
\label{eq:weaktime1}
	\| M\partial_t\hpsi^{m,R,L}\|_{L^2(0,T; [H^s(\Omega \times D)]')}  \leq C, \quad s>d+1,
\end{equation}
uniformly in $m$, $R$ and $L$. We then claim that
\begin{equation}
\label{eq:WeakInTime}
	\hpsi^{m,R,L} \inb C_{\mathrm{weak}}([0,T];L^1_M(\OD))\quad\text{uniformly in $L$}.
\end{equation}
This follows from the above bound on the time derivative in conjunction with the entropy bound
\begin{equation}
	\mathcal{F}(\hpsi^{m,R,L})\inb L^\infty(0,T;L^1_M(\OD)).
\end{equation}
To see this, we note that the Maxwellian-weighted $L\log L$ space on $\OD$ has a separable predual (a certain separable Orlicz space), into which $L^\infty(\OD)$, and therefore also $H^s(\OD)$ for $s>d+1$, is continuously embedded. We can then apply Lemma~3.1(b) from~\cite{BaSu2016Compressible} to deduce~\eqref{eq:WeakInTime}. We refer the reader to the proof of \cite[Theorem~4.1]{BaSu2016Compressible} for the full details.

Using properties~\eqref{eq:weakgradient1}, \eqref{eq:weakqdiffusion1} and~\eqref{eq:WeakInTime}, we can integrate equation~\eqref{eq:FPgalerkin2} in time on the interval $[0,t]$ and pass to the limit $L \rightarrow \infty$ to deduce that, for all $t\in[0,T]$,
\begin{equation}
\begin{split}
\label{eq:FPgalerkin3}
	-\int_0^t\intOD  M \hpsi^{m,R} \partial_t\phi \dx{q}\dx{x}\dx s
	& = \int_0^t\intOD M\hpsi^{m,R} u^{m,R}\cdot \nabla_x\phi\dx{q}\dx{x}\dx s\\
	&\;\;+ \int_0^t\intOD M\hpsi^{m,R}\chi_R(|q|)\ \prt*{(\nabla_x u^{m,R})q} \cdot \nabla_q \phi\dx{q}\dx{x}\dx s \\
	&\;\; - \mu\int_0^t\intOD M\nabla_x\hpsi^{m,R}\cdot\nabla_x\phi\dx{q}\dx{x}\dx s\\
	&\;\; - \int_0^t\intOD M\nabla_q\hpsi^{m,R}\cdot\nabla_q\phi\dx{q}\dx{x}\dx s\\
	&\;\; + \intOD  M \hpsi_0\,\phi(0) \dx{q}\dx{x} - \intOD  M \hpsi^{m,R}(t)\,\phi(t) \dx{q}\dx{x}\\
	&\quad \forall\, \phi\in W^{1,1}(0,T;C_c^\infty(\overline\Omega\times D)).
\end{split}	
\end{equation}

Finally, we would like to pass to the limit in the entropy identity~\eqref{eq:GalerkinEntropy}. To this end, we need to establish the following weak convergence result: 
\begin{equation}
\label{eq:Entropy_dissip_convergence}
	\sqrt{M}\nabla_{x,q}\sqrt{\hpsi^{m,R,L}} \rightharpoonup \sqrt{M}\nabla_{x,q}\sqrt{\hpsi^{m,R}}\;\; \text{weakly in $L^2(0,T;L^2(\OD;\R^{2d}))$}\quad \mbox{as $L \rightarrow \infty$}.	
\end{equation}

From the entropy inequality~\eqref{eq:GalerkinEntropy}, we know that the family  $(\sqrt{M}\nabla_{x,q}\sqrt{\hpsi^{m,R,L}})_{L>0}$ is uniformly bounded in $L^2(0,T;L^2(\OD;\R^{2d}))$. We can therefore pass to a subsequence which converges weakly to some limit, say $S
\in L^2(0,T;L^2(\Omega \times D;\mathbb{R}^{2d}))$. That is,
\begin{equation}
\label{eq:IntermediateWeak1}
    \sqrt{M}\nabla_{x,q}\sqrt{\hpsi^{m,R,L}} \rightharpoonup S\;\;\; \text{weakly in $L^2(0,T;L^2(\OD;\R^{2d}))$}\quad \mbox{as $L \rightarrow \infty$}.
\end{equation}
To identify the weak limit $S$ it then suffices to show that 
\begin{equation}
\label{eq:DissipationTermIdentification}
    \int_0^T\intOD S \cdot\phi \dx q\dx x\dx t = \int_0^T\intOD \sqrt{M}\nabla_{x,q}\sqrt{\hpsi^{m,R}} \cdot\phi \dx q\dx x\dx t,
\end{equation}
for all $\phi\in L^2(0,T;  C_c^\infty(\Omega\times D;\R^{2d}))$. Clearly, by partial integration and noting the strong convergence of $ \sqrt{M \hpsi^{m,R,L}}$ to $ \sqrt{M \hpsi^{m,R}}$ in $L^2(0,T;L^2(\Omega \times D))$ implied by the strong convergence
of $\hpsi^{m,R,L}$ to $\hpsi^{m,R}$ in $L^1(0,T;L^1_M(\Omega \times D))$ (cf. \eqref{eq:strongLq}), we have that
\begin{align*}
\int_0^T\intOD \sqrt{M}\nabla_{x,q}\sqrt{\hpsi^{m,R,L}} \cdot\phi \dx q\dx x\dx t &= - 
\int_0^T\intOD \sqrt{\hpsi^{m,R,L}} \,\nabla_{x,q} \cdot (\sqrt{M} \phi) \dx q\dx x\dx t\\
&= - 
\int_0^T\intOD \sqrt{M \hpsi^{m,R,L}} \,\frac{1}{\sqrt{M}}\nabla_{x,q} \cdot (\sqrt{M} \phi) \dx q\dx x\dx t\\
&\rightarrow -\int_0^T\intOD \sqrt{M \hpsi^{m,R}} \,\frac{1}{\sqrt{M}}\nabla_{x,q} \cdot (\sqrt{M} \phi) \dx q\dx x\dx t\\
&= - 
\int_0^T\intOD \sqrt{\hpsi^{m,R}} \,\nabla_{x,q} \cdot (\sqrt{M} \phi) \dx q\dx x\dx t
\end{align*}
as $L \rightarrow \infty$, for all $\phi\in L^2(0,T; C^\infty_c(\Omega\times D;\R^{2d}))$. This then implies that $S =  \sqrt{M}\nabla_{x,q}\sqrt{\hpsi^{m,R}}$ in $L^2(0,T;\mathcal{D}'(\Omega \times D;\mathbb{R}^{2d}))$, and because $S \in L^2(0,T;L^2(\Omega \times D;\mathbb{R}^{2d}))$ also  $S =  \sqrt{M}\nabla_{x,q}\sqrt{\hpsi^{m,R}}$ in $L^2(0,T;L^2(\Omega \times D;\mathbb{R}^{2d}))$. We thus deduce the desired weak convergence result \eqref{eq:Entropy_dissip_convergence}.

Now, using strong convergence of the velocity, \eqref{eq:Entropy_dissip_convergence} and weak lower semi-continuity of norms (for the gradient terms), Fatou's Lemma (for the entropy term), and the Dominated Convergence Theorem (for the last term)\footnote{The integrand converges a.e.\ and is dominated by the integrable function\\ $M(1\chi_{\{T_L(\hpsi_0)<1\}} + \mathcal{F}(T_L(\hpsi_0))\chi_{\{T_L(\hpsi_0)\geq 1\}}) \leq M(1+\mathcal{F}(\psi_0)$).}, we can pass to the limit in~\eqref{eq:GalerkinEntropy} to obtain
\begin{equation}
\begin{split}
\label{eq:GalerkinEntropy2}
	\frac12\intO |u^{m,R}(t)|^2 \dx{x} &+ \intOD M\mathcal{F}(\hpsi^{m,R})(t)\dx{q}\dx{x} + \nu\int_0^t\intO |\nabla_xu^{m,R}|^2\dx{x}\dx{s}\\
	 &\hspace{-2.5cm}+ 4\mu\int_0^t \intOD M\abs*{\nabla_x\sqrt{\hpsi^{m,R}}}^2\dx{q}\dx{x}\dx{s} + 4 \int_0^t \intOD M\abs*{\nabla_q\sqrt{\hpsi^{m,R}}}^2\dx{q}\dx{x}\dx{s}\\
	 & \leq  \frac12 \intO |u_0^m|^2\dx{x} + \intOD M \mathcal{F}(\hpsi_0) \dx{q}\dx{x}.
\end{split}
\end{equation}

Passing to the limit $L \rightarrow \infty$ in~\eqref{eq:MacroDensity1} and~\eqref{eq:EqnMacroDensity1} we obtain the following initial-boundary-value problem, satisfied in the weak sense:
\begin{align}
\label{eq:EqnMacroDensity2}
	\partial_t \rho^{m,R} + \div_x\prt*{u^{m,R} \rho^{m,R}} - \mu\Delta_x\rho^{m,R} &= 0\quad \text{in $(0,T)\times\Omega$,}\\
    \mu\nabla_x\rho^{m,R}\cdot \hat n &= 0\quad \text{on $(0,T)\times\partial\Omega$,}\\
	\rho^{m,R}(0,x) = \int_D M(q) \hpsi_0(x,q)\dx{q} &= 1\quad \text{for a.e.\ $x\in\Omega$,}
\end{align}
where
\begin{equation}
\label{eq:MacroDensity2}
	\rho^{m,R}(t,x) = \int_D M(q)\hpsi^{m,R}(t,x,q)\dx{q}.
\end{equation}

\subsection{Macroscopic closure}
Having removed the truncation in $\hpsi$ (which was adapted to obtain a cancellation in the $L\log L$ entropy bound above), we are now ready to derive an energy equality for the conformation tensor. Unfortunately, we can no longer test the Fokker--Planck equation~\eqref{eq:FPgalerkin3} with the function $\phi = |q|^2$ (or more generally $\phi(x,q) = (q\otimes q) : \varphi(x)$), as it is not an admissible test function. Therefore, we use instead a sequence of truncated test functions 
\begin{equation}
	\phi_{R'}(x,q) =  T_{R'}(|q|^2)\eta(t), 
\end{equation}
where the truncation operator $T$ is defined in~\eqref{eq:truncationT} and $\eta\in C_c^1([0,T))$ satisfies $\eta(0)=1$, $\partial_t\eta\leq 0$. In particular, we have
\begin{equation}
	\nabla_q T_{R'}(|q|^2) = 2\chi_{R'}(|q|^2)q.
\end{equation}
Since this test function is independent of $x$, there are only two terms on the right-hand side of~\eqref{eq:FPgalerkin3} to look at in detail. 
For the convective term, we write
\begin{equation}
\label{eq:q-convectionR'}
\begin{split}
	\int_0^T\eta\intOD &M\hpsi^{m,R}\chi_R\ \prt*{(\nabla_x u^{m,R})q} \cdot \nabla_q T_{R'}(|q|^2)\dx{q}\dx{x}\dx t\\ 
	&= 2\intOTf \eta \prt*{\int_D M\hpsi^{m,R}\chi_R\ \chi_{R'}(|q|^2)\, q\otimes q\dx{q}} : \nabla_x u^{m,R}\dx{x}\dx t,
\end{split}
\end{equation}
while for the diffusion term we have, after integrating by parts,
\begin{equation}
\label{eq:q-diffusionR'}
\begin{split}
	-2\int_0^T\eta\intOD M\nabla_q\hpsi^{m,R}\cdot q \chi_{R'}(|q|^2)\dx{q}\dx{x}\dx t &\\
        &\hspace{-25mm}= -2\int_0^T\eta\intOD M\hpsi^{m,R}|q|^2 \chi_{R'}(|q|^2) \dx{q}\dx{x}\dx t\\
	&\hspace{-20mm}+2d\int_0^T\eta\intOD M\hpsi^{m,R}\chi_{R'}(|q|^2)\dx{q}\dx{x}\dx t\\
	&\hspace{-20mm}+2\int_0^T\eta\intOD M\hpsi^{m,R}q\cdot \nabla_q\chi_{R'}(|q|^2) \dx{q}\dx{x}\dx t.
\end{split}
\end{equation}
Notice that 
\begin{equation}
	\left| \intOD M\hpsi^{m,R}q\cdot \nabla_q\chi_{R'}(|q|^2) \dx{q}\dx{x} \right| = \left| 2\intOD M\hpsi^{m,R}|q|^2\ \chi_{R'}'(|q|^2) \dx{q}\dx{x}\right| \lesssim \frac{1}{R'}, 
\end{equation}
and therefore, for each fixed $\eta$ as above, the absolute value of the last term of~\eqref{eq:q-diffusionR'} is also $\lesssim 1/{R'}$. 

In each of the other terms we can pass to the limit $R'\to\infty$ using the Monotone Convergence Theorem. Therefore, we deduce that
\begin{equation}
\begin{aligned}
	-\int_0^T\!\!\!\partial_t\eta\intO &\prt*{\int_D M\hpsi^{m,R}|q|^2\dx q}\dx x\dx t + 2\intOTf\!\!\!\eta\brk*{\prt*{\int_D M\hpsi^{m,R}|q|^2\dx{q}} - d\prt*{\int_DM\hpsi^{m,R}}\dx{q}}\dx{x}\dx{t}\\[0.5em]
	 &= \intO \prt*{\int_D M\hpsi_0|q|^2\dx{q}}\dx{x} + 2\int_0^T\eta\intOD M\hpsi^{m,R}\chi_R(|q|)\ (q\otimes q) : \nabla_x u^{m,R} \dx{q} \dx{x}\dx{t}. 
\end{aligned}
\end{equation}
Next, we test the Navier--Stokes equation~\eqref{eq:NSgalerkin3} with $c_i^{m,R}\eta$ and sum over $i=1,\ldots,m$ to deduce that
\begin{equation}
\begin{split}
\label{eq:NSenergyMacro}
	&\frac12\ddt \intO \eta|u^{m,R}|^2\dx{x} - \frac12\intO \partial_t\eta |u^{m,R}|^2\dx x + \nu\intO \eta|\nabla_xu^{m,R}|^2\dx{x} \\
        &\hspace{20mm} = -\intO \eta\tau^{m,R} : \nabla_x u^{m,R}\dx{x}\\
	&\hspace{20mm} = -\intOD \eta M \hpsi^{m,R}\chi_R(|q|) \ (q\otimes q) : \nabla_x u^{m,R}\dx{q}\dx{x}\\
&\hspace{25mm} + \intOD \eta M\hpsi^{m,R}(\nabla_q\chi_R\otimes q) : \nabla_x u^{m,R}\dx{q}\dx{x}.
\end{split}
\end{equation}
Integrating  \eqref{eq:NSenergyMacro} in time, multiplying by 2, summing the last two equalities, and then dividing the resulting equality by 2, we end up with
\begin{equation}
\label{eq:ConformationEnergy}
\begin{split}
	-&\frac12\int_0^T\partial_t\eta\intO \prt*{|u^{m,R}(t)|^2 + \sigma^{m,R}(t)}\dx{x}\dx{t} + \nu\intOTf \eta |\nabla_xu^{m,R}|^2\dx{x}\dx{t} \\
	 &= - \intOTf \eta\tr{(\sigma^{m,R}-\Id)}\dx{x}\dx{t} + \frac12\intO|u^{m}(0)|^2\dx{x} + \frac12\intO \tr{\sigma}(\hpsi_0)\dx{x} + \mathcal{I}^{m,R},
\end{split}
\end{equation}
where we define the approximate conformation tensor
\begin{equation}
\label{eq:conformationMR}
	\sigma^{m,R}(t,x) := \int_D M\hpsi^{m,R}(q\otimes q)\dx{q},
\end{equation}
and the `error term' $\mathcal{I}^{m,R}$ is defined by 
\[\mathcal{I}^{m,R} := \int_0^T\eta\intOD M\hpsi^{m,R} (\nabla_q\chi_R\otimes q) : \nabla_x u^{m,R}\dx{q}\dx{x}\dx{t}.\]
Now, $\mathcal{I}^{m,R}$ can be bounded, independently of $m$, as follows:
\begin{equation}
\begin{split}
	|\mathcal{I}^{m,R}| &\leq \norm{\nabla_xu^{m,R}}_{L^2((0,T)\times\Omega)}\prt*{\int_0^T\!\!\intO\abs*{\int_D M\hpsi^{m,R}|q|\abs*{\nabla_q\chi_{R}}\dx{q}}^2\dx{x}\dx{t}}^{1/2}\\
	 &\lesssim \brk*{\int_0^T\!\!\intO\prt*{\int_D M\hpsi^{m,R}|q|^2\dx{q}}\prt*{\int_D M\hpsi^{m,R}\abs*{\nabla_q\chi_{R}}^2\dx{q}}\dx{x}\dx{t}}^{1/2}\\
	 &\lesssim \frac{1}{R}.
\end{split}
\end{equation}

\subsection{The limit $R\to\infty$}
Passage to the limit $R \to \infty$ can be performed similarly to the passage to the limit $L\to\infty$ before. The main point is that all of the estimates deduced from the entropy inequality~\eqref{eq:GalerkinEntropy2} are independent of $R$. 
In particular, we can guarantee the existence of functions $c_i^{m}$ and $u^m$ such that, as $R \to \infty$, 
	\begin{alignat}{2}
		c_i^{m,R} &\rightarrow c_i^{m}\;\; &&\text{weakly$^*$ in $W^{1,\infty}(0,T)$,} \\
		c_i^{m,R} &\rightarrow c_i^{m} &&\text{strongly in $C([0,T])$,} \\
		u^{m,R} &\rightarrow u^{m} &&\text{strongly in $C([0,T];V\cap H^{d+1}(\Omega;\R^d))$,}
	\end{alignat}
where
\begin{equation}
	u^{m}(t,x):=\sum_{i=1}^m c_i^{m}(t)w_i(x).
\end{equation}
Similarly, there exists a function  $\hpsi^{m}\inb C_{\mathrm{weak}}([0,T];L^1_M(\OD))$ such that, as $R \to \infty$, 
	\begin{alignat}{2}
		\hpsi^{m,R} &\rightarrow \hpsi^{m}\;\; &&\text{strongly in $L^q(0,T;L^1_M(\OD))$,\;\; $q\in[1,\infty)$,} \\
		 M\nabla_{x,q} \hpsi^{m,R} &\rightharpoonup M\nabla_{x,q}\hpsi^{m}\;\; && \text{weakly in $L^1((0,T)\times\Omega;L^1(D;\R^{2d}))$},\\
		M\partial_t \hpsi^{m,R} &\rightharpoonup M\partial_t\hpsi^{m}\;\; && \text{weakly in $L^2(0,T;[H^s(\OD)]')$}.
	\end{alignat}
Extracting a subsequence, we can assume that the first of the above convergences holds almost everywhere in time. These are then sufficient to pass to the limit in the $q$-convective term in the 
Fokker--Planck equation:\ if $\phi\in W^{1,1}(0,T;C_c^\infty(\overline\Omega \times D))$ is a test function with $\supp\phi\subset [0,T]\times\overline\Omega\times B_{r}(0)$, then
\begin{equation}
\begin{aligned}
	&\abs*{\int_0^t\intOD \brk*{M\chi_R\hpsi^{m,R} ((\nabla_xu^{m,R})q)\cdot\nabla_q\phi - M\hpsi^{m}((\nabla_xu^m) q)\cdot\nabla_q\phi} \dx{q}\dx{x}\dx{s}}\\
	&\qquad\qquad\leq \int_0^t\intOD M\chi_R\abs*{\hpsi^{m,R}-\hpsi^m}\abs*{\nabla_xu^{m,R}}|q||\nabla_q\phi|\dx{q}\dx{x}\dx{s}\\
	&\qquad\qquad\quad+ \int_0^t\intOD M\hpsi^{m}\prt*{1-\chi_R}\abs*{\nabla_xu^{m,R}}|q||\nabla_q\phi|\dx{q}\dx{x}\dx{s}\\
	&\qquad\qquad\quad+\int_0^t\intOD M\hpsi^m|q|\abs*{\nabla_xu^{m,R}-\nabla_xu^m}|\nabla_q\phi|\dx{q}\dx{x}\dx{s}\\
	&\qquad\qquad\leq r\norm{\nabla_xu^{m,R}}_{L^\infty(0,T;L^\infty(\Omega;\R^{d\times d}))}\norm{\nabla_q\phi}_{L^2(0,T;L^\infty(\Omega \times D))}\norm*{\hpsi^{m,R}-\hpsi^m}_{L^2(0,T;L^1_M(\OD))}\\
	&\qquad\qquad\quad + r\norm{\nabla_xu^{m,R}}_{L^\infty(0,T;L^\infty(\Omega;\R^{d\times d}))}\int_0^t\intOD M\hpsi^{m}\prt*{1-\chi_R}|\nabla_q\phi|\dx{q}\dx{x}\dx{s}\\
	&\qquad\qquad\quad + r\norm{\nabla_q\phi}_{L^2(0,T;L^\infty(\Omega \times D))}\norm{\hpsi^{m}}_{L^2(0,T;L^1_M(\OD))}\norm{\nabla_xu^{m,R}-\nabla_xu^{m}}_{L^\infty(0,T;H^d(\Omega;\R^d))}.
\end{aligned}
\end{equation}
From the strong convergence of $\hpsi^{m,R}$, the Monotone Convergence Theorem, and strong convergence of $u^{m,R}$, respectively, we see that each of the terms on the right-hand side of the last inequality vanishes as $R\to\infty$. Consequently, passing to the limit $R\to\infty$ we obtain from~\eqref{eq:FPgalerkin3} the following equation:
\begin{equation}
\begin{split}
\label{eq:FPgalerkin4}
	-\int_0^t\!\!\intOD  M \hpsi^{m} \partial_t\phi \dx{q}\dx{x}\dx s
	& = \intOD M\hpsi_0\phi(0)\dx q\dx x - \intOD M\hpsi^m(t)\phi(t)\dx q\dx x\\
	&\hspace{-35mm}+\int_0^t\intOD M\hpsi^{m} u^{m}\cdot \nabla_x\phi\dx{q}\dx{x}\dx s
	+ \int_0^t\intOD M\hpsi^{m} \prt*{(\nabla_x u^{m})q} \cdot \nabla_q \phi\dx{q}\dx{x}\dx s \\
	&\hspace{-35mm} - \mu\int_0^t\intOD M\nabla_x\hpsi^{m}\cdot\nabla_x\phi\dx{q}\dx{x}\dx{s}
	- \int_0^t\intOD M\nabla_q\hpsi^{m}\cdot\nabla_q\phi\dx{q}\dx{x}\dx{s}\\
	&\quad \forall t\in[0,T]\quad\forall\, \phi\in W^{1,1}(0,T; C_c^\infty(\overline\Omega\times D)).
\end{split}	
\end{equation}

However, passing to the limit $R \to \infty$ in the extra stress tensor term $\tau^{m,R}$ appearing in equation~\eqref{eq:NSgalerkin3} requires some care. 
Recall that
\begin{equation}
\label{eq:TauMR}
\tau^{m,R} = \int_D M\chi_R \hpsi^{m,R}(q\otimes q)\dx{q} - \int_D M\hpsi^{m,R}(\nabla_q\chi_R\otimes q) \dx{q} - \int_D M\hpsi^{m,R}\chi_R \Id \dx{q}.
\end{equation}
It is easy to see that, as $R \to \infty$,
\begin{equation}
	\int_D M\hpsi^{m,R}\chi_R \Id \dx{q} \to \int_D M\hpsi^{m} \Id \dx{q}\;\; \text{strongly in $L^1(\Omega;\R^{d\times d})$ for a.e.\ $t\in(0,T)$},
\end{equation}
and
\begin{equation}
	\int_D M\hpsi^{m,R} (\nabla_q\chi_R\otimes q) \dx{q} \to 0\;\; \text{strongly in $L^1(\Omega;\R^{d\times d})$ for a.e.\ $t\in(0,T)$}.
\end{equation}
However, since we only have an $L^\infty(0,T; L^1(\Omega\times D;\mathbb{R}^{d \times d}))$ bound on the conformation tensor $\sigma^{m,R}$ defined in~\eqref{eq:conformationMR}, i.e., 
\begin{equation}
	\sup_{t\in(0,T)}\ \intO \abs*{\int_D M \hpsi^{m,R}\ (q\otimes q)\dx{q}}\dx{x} \leq C,
\end{equation}
we cannot pass to the weak limit in $\tau^{m,R}$ as $R \to \infty$.
To overcome this problem we embed the space $L^\infty(0,T;L^1(\Omega;\R^{d\times d}_{sym}))$ into $L^\infty(0,T;\mathcal{M}(\overline\Omega;\R^{d\times d}_{sym}))$. Then, by the Banach--Alaoglu Theorem, there exists a time-parameterised positive semidefinite measure $\mu_{\sigma}^m \in L^\infty(0,T;\mathcal{M}(\overline\Omega;\R^{d\times d}_{sym}))$ such that, as $R \to \infty$,
\begin{equation}
\sigma^{m,R} \rightharpoonup^* \mu_{\sigma}^m\;\; \text{weakly$^*$ in $L^\infty(0,T;\mathcal{M}(\overline\Omega;\R^{d\times d}_{sym})$)}
\end{equation}
(i.e., in duality with $L^1(0,T; C(\overline\Omega;\R^{d\times d}_{sym}))$.

Now consider the localised approximate conformation tensor
\begin{equation}
	\bar\sigma^{m,R}(t,x) := \int_D M(q)\hpsi^{m,R}(t,x,q)\chi_R(|q|)\,(q\otimes q) \dx q.
\end{equation}
Similarly as above, there exists a positive semidefinite measure $\bar\mu_{\sigma}^m \in  L^\infty(0,T;\mathcal{M}(\overline\Omega;\R^{d\times d}_{sym}))$ such that, as $R \to \infty$, 
\begin{equation}
	\bar\sigma^{m,R} \rightharpoonup^* \bar\mu_{\sigma}^m\;\; \text{weakly$^*$ in $L^\infty(0,T;\mathcal{M}(\overline\Omega;\R^{d\times d}_{sym})$)}.
\end{equation}
Thus, as $R \to \infty$,
\begin{equation}
	\int_0^t\intO \tau^{m,R}(s) : \nabla_x w_i \dx{x}\dx{s} \to \int_0^t\skp*{\bar\mu_{\sigma}^m(s),\nabla_x w_i}\dx{s}\quad \forall\, i \in \{1,\ldots, m\},\;\; \text{a.e. }\; t\in(0,T).
\end{equation}
Therefore, passing to the limit $R\to\infty$ in~\eqref{eq:NSgalerkin3}, we obtain
\begin{equation}
\begin{split}
\label{eq:NSgalerkin4}
	\int_0^t\intO \partial_t u^{m}\cdot w_i\dx{x}\dx s - \int_0^t\intO (u^{m}\otimes u^{m}) : \nabla_x w_i\dx{x}\dx s &+ \nu\int_0^t\intO \nabla_x u^{m} : \nabla_x w_i\dx{x}\dx{s}\\
	&= -\int_0^t\skp*{\bar\mu_{\sigma}^m(s),\nabla_x w_i}\dx{s},
\end{split}
\end{equation}
for all $i=1,\ldots, m$ and a.e.\ $t\in(0,T)$.

We can also pass to the limit $R \rightarrow \infty$ in~\eqref{eq:GalerkinEntropy2} to deduce that
\begin{equation}
\begin{split}
\label{eq:GalerkinEntropy3}
	\frac12\intO |u^{m}(t)|^2 \dx{x} &+ \intOD M\mathcal{F}(\hpsi^{m})(t)\dx{q}\dx{x} + \nu\int_0^t\intO |\nabla_xu^{m}|^2\dx{x}\dx{s}\\
	 &\hspace{-2.5cm}+ 4\mu\int_0^t \intOD M\abs*{\nabla_x\sqrt{\hpsi^{m}}}^2\dx{q}\dx{x}\dx{s} + 4 \int_0^t \intOD M\abs*{\nabla_q\sqrt{\hpsi^{m}}}^2\dx{q}\dx{x}\dx{s}\\
	 & \leq  \frac12 \intO |u_0^m|^2\dx{x} + \intOD M \mathcal{F}(\hpsi_0) \dx{q}\dx{x}.
\end{split}
\end{equation}

Finally, we consider the energy identity~\eqref{eq:ConformationEnergy} for the conformation tensor $\sigma^{m,R}$. Denoting by $\mu_{\tr\sigma}^m$, the trace of the symmetric positive semidefinite matrix-valued measure $\mu_{\sigma}^m$, we can pass to the limit $R\to\infty$ to obtain 
\begin{equation}
\label{eq:ConformationEnergy2}
\begin{split}
	-&\frac12\int_0^T\partial_t\eta\prt*{\intO |u^{m}(t)|^2\dx{x} +\skp*{\mu^m_{\tr\sigma}(t),\mathbbm{1}_{\overline\Omega}}}\dx t + \nu\intOTf \eta |\nabla_xu^{m}|^2\dx{x}\dx{t}\\
	 &= - \int_0^T \eta\brk*{\skp*{\mu_{\tr\sigma}^m(t),\mathbbm{1}_{\overline\Omega}} - d|\Omega|} \dx{t} +\frac12\intO|u^{m}(0)|^2\dx{x} + \frac12\intO \tr{\sigma(\hpsi_0)}\dx{x},
\end{split}
\end{equation}
for each $\eta\in C_c^1([0,T))$ with $\eta(0)=1$ and $\partial_t\eta\leq 0$.

Let us now define
\begin{equation}
	\sigma(\hpsi^m)(t,x) := \int_D M(q)\hpsi^m(t,x,q)\, (q\otimes q) \dx{q}.
\end{equation}
We then have that $\sigma(\hpsi^m) \inb L^\infty(0,T;L^1(\Omega;\R^{d\times d}))$. By identifying this function with a symmetric positive semidefinite matrix-valued measure in $L^\infty(0,T;\mathcal{M}(\overline\Omega;\R^{d\times d}_{sym}))$, we can consider the parametrised defect measures
\begin{equation}
	\bar\mu_{\sigma}^m(t) - \sigma(\hpsi)(t) \quad\text{and}\quad \mu_{\tr\sigma}^m(t) - \tr\sigma(\hpsi^m)(t).
\end{equation}
We will now leverage the energy inequality to deduce additional information about those measures. 
First, we show that the defect measure $\bar{\mu}_\sigma^m - \sigma(\hpsi^m)$ is positive semidefinite, i.e.,
\begin{equation}
\label{eq:PositivityM}
	\skp*{\bar{\mu}_\sigma^m(t) - \sigma(\hpsi^m)(t),\xi\otimes\xi} \geq 0\quad \forall\, \xi\in\R^d\quad \text{for a.e. $t \in (0,T)$}.
\end{equation}
Let $R'>0$ and $\eta\in L^1(0,T)$ be fixed.
Since $\bar\sigma^{m,R} \geq 0$ in the sense of symmetric positive semidefinite matrices, we have
\begin{equation}
	\int_0^T\!\!\!\eta(s)\intO \xi^{\mathrm{T}} \bar\sigma^{m,R} \xi \dx{x}\dx s \geq \int_0^T\!\!\!\eta(s)\intO \xi^{\mathrm{T}}  \prt*{\int_{B_{R'}(0)} M\hpsi^{m,R}\chi_R(|q|)\, (q\otimes q)\dx{q}} \xi \dx{x}\dx s\quad \forall\, \xi \in \mathbb{R}^d.
\end{equation}
We can pass to the limit $R\to\infty$ in this inequality, using the strong convergence of $\hpsi^{m,R}$ to $\hpsi^m$ on the right-hand side, to deduce that 
\begin{equation}
\label{eq:PositivityMIntermediate}
	\int_0^T\eta(s)\skp*{\bar{\mu}_\sigma^m(s),\xi\otimes\xi}\dx s \geq \int_0^T\eta(s)\intO \xi^{\mathrm{T}}  \prt*{\int_{B_{R'}(0)} M\hpsi^{m}\,(q\otimes q)\dx{q}} \xi \dx{x} \dx s\quad \forall\, \xi \in \mathbb{R}^d.
\end{equation}
By the Monotone Convergence Theorem, we can pass to the limit $R'\to\infty$ to obtain
\begin{equation}
 \int_0^T\eta(s)\skp*{\bar{\mu}_{\sigma}^m(s)-\sigma(\hpsi^m)(s),\xi\otimes\xi}\dx s \geq 0 \quad \forall\, \xi \in \mathbb{R}^d.
\end{equation}
Finally, for $\delta>0$, we choose $\eta = \frac{1}{\delta}\mathbbm{1}_{[t,t+\delta]}$ and use Lebesgue's Differentiation Theorem to pass to the limit $\delta\to 0_+$ to deduce~\eqref{eq:PositivityM}.

Similarly, for $R'<R$ and the same choice of $\eta$ as above, we observe
\begin{align}
	&\int_0^T \eta \intO \prt*{\bar\sigma^{m,R} - \int_{B_{R'}(0)} M\hpsi^{m,R}\chi_R(q\otimes q)\dx q} : \nabla_x\vartheta \dx x \dx s\\
	 &\quad = \int_0^T \eta \intO \prt*{\int_{\set{|q|>R'}} M\hpsi^{m,R}\chi_R(q\otimes q) \dx q} : \nabla_x\vartheta \dx x \dx s\\
	 &\quad \leq C\int_0^T \eta \intO |\nabla_x\vartheta| \prt*{\int_{\set{|q|>R'}} M\hpsi^{m,R}|q|^2 \dx q} \dx x \dx s\\
	 &\quad \leq C\norm{\vartheta}_{C^1(\overline\Omega;\R^d)}\int_0^T\eta\intO \prt*{\tr\sigma^{m,R} - \int_{B_{R'}(0)} M\hpsi^{m,R}|q|^2\dx q} \dx x \dx s.
\end{align}
Passing to the limit $R\to \infty$ we obtain
\begin{align}
	\int_0^T & \eta\, \skp*{\bar\mu_\sigma^m, \nabla_x\vartheta} \dx s - \int_0^T\eta\intO\prt*{\int_{B_{R'}(0)} M\hpsi^m(q\otimes q)\dx q} : \nabla_x\vartheta \dx x \dx s\\
	&\leq C\norm{\vartheta}_{C^1(\overline\Omega;\R^d)} \int_0^T \eta \prt*{\skp*{\tr\mu_\sigma^m, \mathbbm{1}_{\overline\Omega}} - \intO\int_{B_{R'}(0)} M\hpsi^m|q|^2\dx q} \dx x \dx s.
\end{align}
Then, passing to the limit $R'\to\infty$ and letting $\delta \rightarrow 0_+$, we deduce that, for a.e.\ $t\in (0,T)$,
\begin{equation}
	\label{eq:CompatibilityM}
		\abs*{\skp*{\bar{\mu}_{\sigma}^m(t)-\sigma(\hpsi^m)(t), \nabla_x\vartheta}} \leq \norm{\vartheta}_{C^1(\overline\Omega;\R^d)}\skp*{\mu_{\tr\sigma}^m(t)-\tr\sigma(\hpsi^m)(t),\mathbbm{1}_{\overline\Omega}}\quad \forall\, \vartheta \in C^1(\overline\Omega;\R^d).
\end{equation}

Next, we consider the energy identity~\eqref{eq:ConformationEnergy2}. Choosing $t\in(0,T)$, $\delta>0$, and a sequence of smooth nonnegative test functions $\eta$ approximating the continuous nonnegative piecewise linear function
\begin{equation}
	\eta_{t,\delta}(s):= \begin{cases}
	1 & \text{for\;} s\in[0,t),\\
	\frac{(t+\delta)-s}{\delta} & \text{for\;} s\in  [t,t+\delta),\\
	0 & \text{for\;} s\in (t+\delta,T],
	\end{cases}
\end{equation}
we can pass to the limit in turn with the approximation of $\eta_{t,\delta}$, and, using Lebesgue's Differentiation Theorem, with $\delta\to0_+$. We thus obtain, for a.e.\ $t\in(0,T)$, the energy inequality
\begin{equation}
\label{eq:ConformationEnergy2M}
\begin{split}
	\frac12\intO |u^m(t)|^2\dx{x} + \frac12 \skp*{\mu^m_{\tr\sigma}(t),\mathbbm{1}_{\overline\Omega}} &+ \nu\intOT |\nabla_xu^m|^2\dx{x}\dx{s} + \int_0^t \brk*{\skp*{\mu^m_{\tr\sigma}(s),\mathbbm{1}_{\overline\Omega}} - d|\Omega|} \dx{s}\\
	 &\leq \frac12\intO|u_0|^2\dx{x} + \frac12\intO \tr{\sigma(\hpsi_0)}\dx{x}.
\end{split}
\end{equation}

Testing the Fokker--Planck equation~\eqref{eq:FPgalerkin4} with the function $\phi(q) = T_R(|q|^2)$ (with the same truncation as before, defined in~\eqref{eq:truncationT}), we obtain
\begin{align}
	\intOD & M\hpsi_0 T_R(|q|^2)\dx{q}\dx{x} - \intOD M\hpsi^m(t) T_R(|q|^2)\dx{q}\dx{x} \\ 
	&= - \intOT\int_D M\hpsi^m \prt{(\nabla_x u^m) q} \cdot \nabla_q T_R(|q|^2)\dx{q}\dx{x}\dx{s}
	+ \intOT\int_D M\nabla_q\hpsi^m\cdot\nabla_qT_R(|q|^2)\dx{q}\dx{x}\dx{s}.
\end{align}
The first term on the right-hand side can be rewritten as
\begin{equation}
	- \intOT\int_D M\hpsi^m \prt{(\nabla_x u^m) q} \cdot \nabla_q T_R(|q|^2)\dx{q}\dx{x}\dx{s} = -2\intOT \prt*{\int_D M\hpsi^m \chi_{R}(|q|^2)\,(q\otimes q)\dx{q}} : \nabla_x u\dx{x}\dx{s},
\end{equation}
while the other term can be written as follows:
\begin{equation}
\begin{split}
	2\intOT\int_D M\nabla_q\hpsi^m\cdot q \chi_{R}(|q|^2)\dx{q}\dx{x}\dx{s} & = 2\intOT\int_D M\hpsi^m |q|^2 \chi_{R}(|q|^2) \dx{q}\dx{x}\dx{s}\\
	&\quad -2d\intOT\int_D M\hpsi^m\chi_{R}(|q|^2)\dx{q}\dx{x}\dx{s}\\
	&\quad -2\intOT\int_D M\hpsi^m q\cdot \nabla_q\chi_{R}(|q|^2) \dx{q}\dx{x}\dx{s}.
\end{split}
\end{equation}
As before, the last term can be bounded as follows:
\begin{equation}
	\intOT\int_D M\hpsi^m q\cdot \nabla_q\chi_{R}(|q|^2) \dx{q}\dx{x}\dx{s} = 2\intOT\int_D M\hpsi^m |q|^2\,\chi_{R}'(|q|^2) \dx{q}\dx{x}\dx{s} \lesssim \frac{1}{R}.
\end{equation}
In each of the other terms we can pass to the limit $R\to\infty$ using the Monotone Convergence Theorem, and thus we deduce that
\begin{equation}
	\intO \tr \sigma(\hpsi_0) \dx{x} - \intO  \tr \sigma(\hpsi^m)(t) \dx{x} = -2\intOT \sigma(\hpsi^m) : \nabla_x u^m \dx{x}\dx{s} + 2 \intOT \tr(\sigma(\hpsi^m) - \Id) \dx{x}\dx{s}.
\end{equation}
Next, we multiply the Navier--Stokes equation~\eqref{eq:NSgalerkin4} by $c_i$ and sum over the index $i$ to derive the equality
\begin{equation}
\begin{split}
	\frac12\intO |u^m(t)|^2\dx{x} - \frac12\intO |u_0|^2\dx{x} + \nu \intOT |\nabla_x u^m|^2\dx{x}\dx{s}
	= - \int_0^t \skp*{\bar{\mu}_\sigma^m(s),\nabla_x u^m}\dx{s}. 
\end{split}
\end{equation}
Consequently, we obtain the energy equality
	\begin{equation}
	\begin{aligned}
	\label{eq:OBEnergyEqualityM}
		\frac12\intO |u^m(t)|^2\dx{x} &+ \frac12\intO \tr{\sigma(\hpsi^m)}(t)\dx{x} + \nu\intOT |\nabla_x u^m|^2\dx{x}\dx{s} + \intOT\tr{\prt{\sigma(\hpsi^m)-\Id}}\dx{x}\dx{s}\\
		&= \frac12\intO |u_0|^2\dx{x} + \frac12\intO \tr{\sigma(\hpsi_0)}\dx{x} - \int_0^t \skp*{\bar{\mu}_\sigma^m(s) - \sigma(\hpsi^m)(s),\nabla_x u^m}\dx{s}. 
	\end{aligned}
	\end{equation}	
Comparing this with the energy inequality~\eqref{eq:ConformationEnergy2M}, we deduce that
\begin{align}
	\frac12\skp*{\mu^m_{\tr\sigma}(t) - \tr\sigma(\hpsi^m)(t),\mathbbm{1}_{\overline\Omega}} &+ \int_0^t \skp*{\mu^m_{\tr\sigma}(s) - \tr\sigma(\hpsi^m)(s),\mathbbm{1}_{\overline\Omega}}\dx{s} \\
	&\leq \int_0^t \skp*{\bar{\mu}_\sigma^m(s) - \sigma(\hpsi^m)(s),\nabla_x u^m}\dx{s}. 
\end{align}	
Inequality~\eqref{eq:CompatibilityM} then implies that
\begin{equation}
\skp*{\mu^m_{\tr\sigma}(t) - \tr\sigma(\hpsi^m)(t),\mathbbm{1}_{\overline\Omega}} \leq 2\int_0^t \norm{\nabla_x u^m(s)}_{C(\overline\Omega;\R^{d\times d})} \skp*{\mu^m_{\tr\sigma}(s) - \tr\sigma(\hpsi^m)(s),\mathbbm{1}_{\overline\Omega}}\dx{s}. 
\end{equation}
Consequently, Gronwall's Lemma implies that the measure $\mu^m_{\tr\sigma}- \tr\sigma(\hpsi^m)$ is identically equal to zero for a.e.\ $t\in(0,T)$. Therefore, using~\eqref{eq:CompatibilityM} we deduce that
\begin{equation}
	\bar{\mu}_\sigma^m = \sigma(\hpsi^m), \quad \mu_{\tr\sigma}^m = \tr\sigma(\hpsi^m), \quad\text{for a.e.\ $t\in(0,T)$}.
\end{equation}
In particular, $\mu_{\tr\sigma}^m = \tr\bar{\mu}_\sigma^m$, and because $\sigma(\hpsi^m) \inb L^\infty(0,T;L^1(\Omega;\R^{d\times d}))$, it also follows that  $\bar{\mu}_\sigma^m \inb L^\infty(0,T;L^1(\Omega;\R^{d\times d}))$ and $\mu_{\tr\sigma}^m
\inb L^\infty(0,T;L^1(\Omega;\R^{d}))$.

\subsection{The limit $m\to\infty$}
As the final step, we pass to the limit $m\to\infty$ in the Galerkin approximation of the Navier--Stokes equation. Again, we will use the uniform bounds obtained from the entropy estimate~\eqref{eq:GalerkinEntropy3} to deduce that $u^m$ is uniformly bounded in $L^\infty(0,T;L^2(\Omega;\R^d))\cap L^2(0,T;V)$.  Then one sees from~\eqref{eq:NSgalerkin4} that
\begin{equation}
	\partial_t u^m \inb L^2(0,T;[V \cap H^{d+1}(\Omega;\mathbb{R}^d)]'),   \quad \text{uniformly in $m$}.
\end{equation}
It follows that there exists a $u\in C_{\mathrm{weak}}([0,T];L^2(\R^d))$ such that
	\begin{alignat}{2}
		u^m &\rightarrow u\;\; &&\text{strongly in $L^2(0,T;L^2(\Omega;\R^d))$,} \\
		\nabla_{x} u^{m} &\rightharpoonup \nabla_{x} u\;\; &&\text{weakly in $L^2(0,T;L^2(\Omega;\R^{d\times d}))$}.
	\end{alignat}
Moreover, since the sequence of measures $\bar{\mu}_{\sigma}^m=\sigma(\hpsi^m)$ is uniformly bounded in $L^\infty(0,T;\mathcal{M}(\overline\Omega;\R^{d\times d}_{sym}))$ (cf.\ the final sentence of the previous section), there is a measure $\bar{\mu}_{\sigma}$ and a non-relabelled subsequence of $(\bar{\mu}_{\sigma}^m)_{m \geq 1}$, such that
\begin{align}
	\bar{\mu}_{\sigma}^m &\rightharpoonup^* \bar{\mu}_{\sigma}\;\;\text{weakly$^*$ in}\; L^\infty(0,T;\mathcal{M}(\overline\Omega;\R^{d\times d}_{sym}))\quad \text{as $m \to \infty$.}
\end{align}

Passing to the limit in the Fokker--Planck equation follows similar lines as before.
In particular, there exists a function $\hpsi\geq 0$ such that, as $m \to \infty$, 
	\begin{alignat}{2}
		\hpsi^{m} &\rightarrow \hpsi\;\; &&\text{strongly in $L^q(0,T;L^1_M(\OD))$,\;\; $q\in[1,\infty)$,} \\
		M\nabla_{x,q} \hpsi^{m} &\rightharpoonup M\nabla_{x,q}\hpsi\;\; && \text{weakly in $L^1((0,T)\times\Omega;L^1(D;\R^{2d}))$},\\
		M\partial_t \hpsi^{m} &\rightharpoonup M\partial_t\hpsi\;\; && \text{weakly in $L^2(0,T;[H^s(\OD)]')$}.
	\end{alignat}
Moreover, by choosing a further subsequence, we can assume that the first convergence property stated above is true almost everywhere on $(0,T) \times \Omega\times D$.
	
We can now pass to the limit in~\eqref{eq:FPgalerkin4} to obtain
	\begin{equation}
	\label{eq:FPfinal}
	\begin{split}
	\intOT\int_D & M\hpsi\partial_t\phi \dx{q}\dx{x}\dx{s} + \int_{\Omega\times D} M\hpsi_0\phi(0)\dx{q}\dx{x} - \int_{\Omega\times D} M\hpsi(t)\phi(t)\dx{q}\dx{x}\\
	& = -\intOT\int_D M\hpsi u\cdot \nabla_x\phi\dx{q}\dx{x}\dx{s} - \intOT\int_D M\hpsi \prt{(\nabla_x u) q} \cdot \nabla_q \phi\dx{q}\dx{x}\dx{s} \\
	&\quad + \mu\intOT\int_D M\nabla_x\hpsi\cdot\nabla_x\phi\dx{q}\dx{x}\dx{s}
	+ \intOT\int_D M\nabla_q\hpsi\cdot\nabla_q\phi\dx{q}\dx{x}\dx{s}\\ 
	&\quad \forall t\in[0,T]\quad\forall\, \phi\in W^{1,1}(0,T; C_c^\infty(\overline\Omega\times D)),
	\end{split}
	\end{equation}
and then by a standard density argument we extend the class of test functions under consideration from $\phi \in H^1(0,T;C_c^\infty(\overline\Omega\times D))$ to $\phi\in H^1(0,T;H^s(\OD))$, $s>d+1$.

We now define the parametrised measure
\begin{equation}
	m_{NS}(t) := \bar\mu_{\sigma}(t) - \sigma(\hpsi)(t),\quad t \in (0,T).
\end{equation}
Then, by passing to the limit $m \to \infty$ in~\eqref{eq:NSgalerkin4} we obtain
\begin{align}
\label{eq:NSfinal}	
		\intOT & u\cdot \partial_t\vartheta \dx{x}\dx{s} + \intOT( u\otimes u) : \nabla_x \vartheta \dx{x}\dx{s} - \nu \intOT \nabla_x u : \nabla_x \vartheta \dx{x}\dx{s}\\
		&= \intOT \sigma(\hpsi) : \nabla_x\vartheta\dx{x}\dx{s} + \int_0^t \skp*{m_{NS}(s), \nabla_x\vartheta}\dx{s} + \intO u(t)\cdot\vartheta(t)\dx{x} - \intO u_0\cdot \vartheta(0) \dx{x}\\
		&\forall\, \vartheta \in \set*{ \vartheta \in L^2(0,T;V)\cap L^1(0,T;C^1(\overline\Omega;\mathbb{R}^d)) \;|\; \partial_t\vartheta\in L^1(0,T;L^2(\Omega;\R^d))},
\end{align}
as required.

Passing to the limit $m\to\infty$ in~\eqref{eq:PositivityMIntermediate}, we deduce that
\begin{equation}
	\int_0^T\eta(s)\skp*{\bar{\mu}_\sigma(s),\xi\otimes\xi}\dx s \geq \int_0^T\eta(s)\intO \xi^{\mathrm{T}} \prt*{\int_{B_{R'}(0)} M\hpsi\, (q\otimes q)\dx{q}} \xi \dx{x}\dx s\quad \forall\, \xi \in \mathbb{R}^d.
\end{equation}
Similarly as before this leads to showing that the defect measure $m_{NS}$ is positive semidefinite, i.e.,
\begin{equation}
\label{eq:Positivity}
	\skp*{m_{NS}(t),\xi\otimes\xi} \geq 0\quad \forall\, \xi\in\R^d\quad \text{for a.e. $t \in (0,T)$}.
\end{equation}

Next, we consider the energy identity~\eqref{eq:ConformationEnergy2} and repeat the argument leading to inequality~\eqref{eq:ConformationEnergy2M}.
Passing to the limit $m\to\infty$ with the same choice of approximations of the indicator function of $[0,t)$ as before, we obtain the energy inequality
\begin{equation}
\label{eq:ConformationEnergy3}
\begin{split}
	\frac12\intO |u(t)|^2\dx{x} + \frac12 \skp*{\tr\bar{\mu}_{\sigma}(t),\mathbbm{1}_{\overline\Omega}} &+ \nu\intOT |\nabla_xu|^2\dx{x}\dx{s} + \int_0^t \brk*{\skp*{\tr\bar{\mu}_{\sigma}(s),\mathbbm{1}_{\overline\Omega}} - d|\Omega|} \dx{s}\\
	 &\leq \frac12\intO|u_0|^2\dx{x} + \frac12\intO \tr{\sigma(\hpsi_0)}\dx{x},
\end{split}
\end{equation}
valid for a.e.\ $t\in(0,T)$.
Defining 
\begin{equation}
	m_{OB}(t) := \tr{\prt*{\bar{\mu}_{\sigma}(t) - \sigma(\hpsi)(t)}}\quad \text{for $t \in (0,T)$}, 
\end{equation}
we can rewrite~\eqref{eq:ConformationEnergy3} to obtain the inequality
\begin{align}
		\frac12\intO &|u(t)|^2\dx{x} + \frac12\intO\tr{\sigma(\hpsi(t))}\dx{x} + \frac12\skp*{m_{OB}(t),\mathbbm{1}_{\overline\Omega}}  + \nu\intOT |\nabla_x u|^2\dx{x}\dx{s}\\
		&+ \intOT\tr{\prt*{\sigma(\hpsi)-\Id}}\dx{x}\dx{s} + \int_0^t \skp*{m_{OB}(s),\mathbbm{1}_{\overline\Omega}}\dx{s}
		\leq \frac12\intO |u_0|^2\dx{x} + \frac12\intO \tr{\sigma(\hpsi_0)}\dx{x},
\end{align}
for a.e.\ $t\in(0,T)$.

The compatibility condition~\eqref{eq:GeneralisedCompatibility} (with $\zeta=const.$) follows in an analogous way as inequality~\eqref{eq:CompatibilityM} before.

Therefore the quadruple $(u,\hpsi, m_{NS}, m_{OB})$ is a generalised dissipative solution of the Hookean dumbbell model according to Definition~\ref{def:GenSol3d}. Furthermore, passing to the limit in~\eqref{eq:GalerkinEntropy3} (via weak lower semi-continuity and Fatou's Lemma) we deduce~\eqref{eq:LlogLenergy}. Finally, passing to the limits $R\to\infty$ and $m\to\infty$ in~\eqref{eq:EqnMacroDensity2}, we deduce that
\begin{equation}
	\int_D M(q)\hat\psi(t,x,q)\dx{q} = 1 \quad \text{for a.e.\ $(t,x)\in [0,T]\times\Omega$}.
\end{equation}

Thus, Theorem~\ref{thm:Existence} is proved for the case $d\geq 3$. The case $d=2$ requires a modification of the macroscopic closure step, and is explained in the next section.\qed

\section{Proof of Theorem~\ref{thm:Existence} in 2D}
\label{sec:Existence2d}

The proof of existence of generalised dissipative solutions for $d=2$ follows essentially the same lines of argument as above. The only deviation is the fact that we can deduce additional regularity of the approximate conformation tensor. Taking advantage of the fact that the uniform $L^2$ estimates below do not require any cancellations between the Fokker--Planck and the Navier--Stokes equations, we derive a macroscopic closure of the Fokker--Planck equation already from equation~\eqref{eq:FPgalerkin2}. Using the additional regularity of the approximate conformation tensor (at each level of approximation), we are able to derive an equation for the limit quantity, which turns out to be the stress equation of the Oldroyd-B system.

Define the approximate conformation tensor
\begin{equation}
	\sigma^{m,R,L}(t,x) := \int_D M\hpsi^{m,R,L}\, (q\otimes q) \dx{q}.
\end{equation}
Let $\varphi \in H^1(\Omega;\R^{2\times 2})$ and set
\begin{equation}
	\phi(x,q) = (q\otimes q) : \varphi(x).
\end{equation}
Testing equation~\eqref{eq:FPgalerkin2} with $\phi$ we obtain, on the left-hand side,
\begin{align*}
	\left\langle \partial_t\hpsi^{m,R,L},\phi\right\rangle_{([H^1_M(\Omega \times D)]',H^1_M(\Omega \times D))} &= \left\langle \partial_t\prt*{\int_D M\hpsi^{m,R,L} q\otimes q\dx{q}} , \varphi\right \rangle_{([H^1(\Omega;\mathbb{R}^{2 \times 2})]',H^1(\Omega;\mathbb{R}^{2 \times 2}))}\\
	& = \left\langle\partial_t\sigma^{m,R,L} , \varphi 
	\right \rangle_{([H^1(\Omega;\mathbb{R}^{2 \times 2})]',H^1(\Omega;\mathbb{R}^{2 \times 2}))},
\end{align*}
whereas on the right-hand side we have, for the terms involving $x$-derivatives,
\begin{equation}
	\intOD M\hpsi^{m,R,L} u^{m,R,L}_k \partial_k\varphi_{ij} (q\otimes q)_{ij} \dx{q}\dx{x} = \intO\prt*{u^{m,R,L}\cdot\nabla_x}\varphi : \sigma^{m,R,L} \dx{x};
\end{equation}
\begin{align}
	-\mu\intOD M\partial_k\hpsi^{m,R,L} \partial_k\prt*{ q_i q_j \varphi_{ij}}\dx{q}\dx{x} &= -\mu\intOD \partial_k\prt*{M\hpsi^{m,R,L} q_i q_j}\partial_k\varphi_{ij}\dx{q}\dx{x} \\
	& = -\mu \intO \nabla_x\sigma^{m,R,L} :: \nabla\varphi\dx{x}.
\end{align}
In the transition to the last line we took advantage of $\hpsi^{m,R,L}$ being in $L^2(0,T;H^1_M(\OD))$ to justify differentiating under the integral sign.
The $q$-convective term can be written as follows:
\begin{align}
	\intOD & M\Lambda_L(\hpsi^{m,R,L})\chi_R(|q|)\, ((\nabla_x u^{m,R,L})q) \cdot \nabla_q\prt*{ (q\otimes q) : \varphi} \dx{q}\dx{x}\\
	 =& \intO \brk*{\int_D M\Lambda_L(\hpsi^{m,R,L})\chi_R(|q|)\, (q\otimes q)\dx{q}}\prt*{\nabla_x u^{m,R,L}}^{\mathrm{T}} : \varphi \dx{x}\\
	 &+\intO \nabla_x u^{m,R,L}\brk*{\int_D M\Lambda_L(\hpsi^{m,R,L})\chi_R(|q|)\, (q\otimes q)\dx{q}} : \varphi \dx{x}. 
\end{align}
Finally, in the diffusion term, we integrate by parts to derive the equality
\begin{align}
	\intOD M\nabla_q\hpsi^{m,R,L} & \cdot \nabla_q\prt*{(q\otimes q) : \varphi} \dx{q}\dx{x} \\
    &= 2\intO \sigma^{m,R,L} : \varphi \dx{x} - 2\intO \brk*{\int_D M\hpsi^{m,R,L}\dx{q}}\Id : \varphi \dx{x}.
\end{align}

We thus arrive at a weak form of an evolution equation satisfied by $\sigma^{m,R,L}$. In particular, thanks to previous estimates yielding $\hpsi^{m,R,L}\in L^2(0,T;H^1_M(\OD))$, we can choose $\varphi = \sigma^{m,R,L}$ as a test function for this equation. To obtain a Gronwall-type inequality, we need to use the Gagliardo--Nirenberg inequality in the following bound:
\begin{align}
	&\abs*{\intO \nabla_x u^{m,R,L}\brk*{\int_D M\Lambda_L(\hpsi^{m,R,L})\chi_R(|q|)\,(q\otimes q)\dx{q}} : \sigma^{m,R,L} \dx{x}}\\
	&\leq \prt*{\intO\abs*{\nabla_xu^{m,R,L}}^2\dx{x}}^{1/2}\prt*{\intO \abs*{\sigma^{m,R,L}}^2 \abs*{\int_D M\Lambda_L(\hpsi^{m,R,L})\chi_R(|q|)\,|q|^2\dx{q}}^2\dx{x}}^{1/2}\\
	&\leq \norm{\nabla_xu^{m,R,L}}_{L^2(\Omega;\R^{2\times 2})}\prt*{\intO \abs*{\sigma^{m,R,L}}^2 \abs*{\int_D M\hpsi^{m,R,L}|q|^2\dx{q}}^2\dx{x}}^{1/2}\\
	&\leq \norm{\nabla_xu^{m,R,L}}_{L^2(\Omega;\R^{2\times 2})}\norm{\sigma^{m,R,L}}_{L^4(\Omega;\R^{2\times 2})}^2\\ 
	&\leq \frac{1}{\mu}\norm{\nabla_xu^{m,R,L}}_{L^2(\Omega;\R^{2\times 2})}^2\norm{\sigma^{m,R,L}}_{L^2(\Omega;\R^{2\times 2})}^2 + \frac{\mu}{4}\norm{\nabla_x\sigma^{m,R,L}}_{L^2(\Omega;\R^{2\times 2})}^2,
\end{align}
where in the transition from the third line to the fourth line we used the pointwise bound
\begin{equation}
	0 \leq \int_D M\hpsi^{m,R,L} |q|^2\dx{q} = \tr\sigma^{m,R,L} \leq \abs*{\sigma^{m,R,L}},
\end{equation}
with $|\sigma^{m,R,L}|$ signifying the Frobenius matrix-norm of the symmetric positive semidefinite matrix $\sigma^{m,R,L}$. We end up with
\begin{equation}
	\ddt \intO \abs*{\sigma^{m,R,L}}^2 \dx{x} + \frac{\mu}{2}\intO \abs*{\nabla_x\sigma^{m,R,L}}^2 \dx{x} \leq C(\Omega) + a(t)\intO \abs*{\sigma^{m,R,L}}^2\dx{x},
\end{equation}
for an integrable nonnegative function $a\in L^1(0,T)$. Using Gronwall's Lemma we deduce that
\begin{equation}
\label{eq:SigmaL2}
	\sup_{t\in(0,T)} \intO \abs*{\sigma^{m,R,L}}^2 \dx{x} + \int_0^t\intO \abs*{\nabla_x\sigma^{m,R,L}}^2 \dx{x}\dx{s} \leq C,
\end{equation}
where the constant $C$ is independent of $m,R,L$.
With the above estimates in hand, we infer the following bound on the time derivative of $\sigma^{m,R,L}$:
\begin{equation}
\label{eq:SigmaTime}
	\|\partial_t\sigma^{m,R,L}\|_{L^2(0,T;[H^1(\Omega;\mathbb{R}^{2 \times 2})]')} \leq C.
\end{equation}

By applying the Aubin--Lions Lemma we deduce that the sequence $(\sigma^{m,R,L})_{L>0}$ is precompact in $L^2(0,T;L^2(\Omega;\R^{2\times 2}))$. Therefore, there is a function $\sigma^{m,R}\in C([0,T];L^2(\Omega;\R^{2\times 2}))$ (and a non-relabelled subsequence) such that, as $L \to \infty$, 
	\begin{alignat}{2}
		\sigma^{m,R,L} &\rightarrow \sigma^{m,R}\;\; &&\text{strongly in $L^2(0,T;L^2(\Omega;\R^{2\times 2}))$},\\
		\nabla_{x} \sigma^{m,R,L} &\rightharpoonup \nabla_{x}\sigma^{m,R}\;\; &&\text{weakly in $L^2(0,T;L^2(\Omega;\R^{2\times 2}))$}.
	\end{alignat}
Since the above bounds are uniform in $R$ and $m$, we can similarly find $\sigma^m$ and $\sigma$ such that
\begin{equation}
   \sigma \in H^1(0,T;[H^1(\Omega;\R^{2\times 2})]') \cap C([0,T];L^2(\Omega;\R^{2\times 2})) \cap L^2(0,T;H^1(\Omega;\R^{2\times 2})), 
\end{equation}
and
\begin{equation}
	\lim_{R\to\infty} \sigma^{m,R} = \sigma^m,\quad \lim_{m\to\infty}\sigma^m = \sigma,
\end{equation}
where the limits are to be understood in the strong sense in $L^2(0,T;L^2(\Omega;\R^{2\times 2}))$ and weakly in $L^2(0,T;H^1(\Omega;\R^{2\times 2}))$.

Passing to the limit with $L\to\infty$ in the equation for $\sigma^{m,R,L}$, we arrive at the following equation for $\sigma^{m,R}$:
\begin{equation}
\label{eq:SigmaMRequation}
\begin{aligned}
	-\intOT\sigma^{m,R} : \partial_t\varphi \dx{x}\dx{s} =& \intO \sigma(\hpsi_0) : \varphi(0) \dx{x} - \intO \sigma^{m,R}(t) : \varphi(t) \dx{x}\\
	 &+\intOT\prt*{u^{m,R}\cdot\nabla_x}\varphi : \sigma^{m,R}\dx{x}\dx{s} - \mu\intOT\nabla_x\sigma^{m,R} :: \nabla_x\varphi \dx{x}\dx{s}\\
	& +\intOT\brk*{\int_D M\hpsi^{m,R}\chi_R\, (q\otimes q) \dx{q}}\prt*{\nabla_xu^{m,R}}^{\mathrm{T}} : \varphi \dx{x}\dx{s}\\
	& + \intOT\nabla_xu^{m,R}\brk*{\int_D M\hpsi^{m,R}\chi_R (q\otimes q) \dx{q}} : \varphi \dx{x}\dx{s}\\
	& - 2\intOT\sigma^{m,R} : \varphi\dx{x}\dx{s} + 2\intOT\brk*{\int_D M\hpsi^{m,R}\dx{q}}\Id:\varphi \dx{x}\dx{s}\\
	&\forall\, \varphi\in C^1([0,T]\times\overline\Omega;\R^{2\times 2})\;\; \text{and }\forall t\in [0,T].
\end{aligned}
\end{equation}
Testing this equation with $\varphi\equiv \Id$ and using~\eqref{eq:NSenergyMacro}, we deduce that 
\begin{equation}
\label{eq:ConformationEnergy2D}
\begin{split}
	\frac12\intO |u^{m,R}(t)|^2\dx{x} &+ \frac12 \intO \tr{\sigma^{m,R}}(t) \dx{x} + \nu\intOT |\nabla_xu^{m,R}|^2\dx{x}\dx{s} + \intOT \tr{(\sigma^{m,R}-\Id)}\dx{x}\dx{s}\\
	 &= \frac12\intO|u^{m}(0)|^2\dx{x} + \frac12\intO \tr{\sigma(\hpsi_0)}\dx{x} + O\prt*{\frac1R}.
\end{split}
\end{equation}
Then, by passing to the limit in this equality with $R\to \infty$, by weak lower-semicontinuity we arrive at the energy inequality satisfied by $\sigma^m$:
\begin{equation}
\label{eq:ConformationEnergy2DM}
\begin{split}
	\frac12\intO |u^m(t)|^2\dx{x} &+ \frac12 \intO \tr{\sigma^m}(t) \dx{x} + \nu\intOT |\nabla_xu^m|^2\dx{x}\dx{s} + \intOT \tr{(\sigma^m-\Id)}\dx{x}\dx{s}\\
	 &\leq \frac12\intO|u(0)|^2\dx{x} + \frac12\intO \tr{\sigma(\hpsi_0)}\dx{x}.
\end{split}
\end{equation}

In a similar way, we can derive an equation for the truncated approximate conformation tensor. Let
\begin{equation}
    \bar\sigma^{m,R,L}(t,x) := \int_D M\hpsi^{m,R,L}\chi_R(|q|)\, (q\otimes q) \dx{q}.
\end{equation}
We now test equation~\eqref{eq:FPgalerkin2} with
\begin{equation}
   	\phi(x,q) = \chi_R(|q|)\, (q\otimes q) : \varphi(x), 
\end{equation}
where $\varphi\in H^1(\Omega;\R^{2\times 2})$.
This yields some additional terms arising from the terms in~\eqref{eq:FPgalerkin2} that involve derivatives in $q$. More precisely, we have
\begin{align}
	\intOD & M\Lambda_L(\hpsi^{m,R,L})\chi_R(|q|) \,((\nabla_x u^{m,R,L})q) \cdot \nabla_q\prt*{\chi_R(|q|)\, (q\otimes q) : \varphi} \dx{q}\dx{x}\\
	 =& \intO \brk*{\int_D M\Lambda_L(\hpsi^{m,R,L})\chi_R^2(|q|)\, (q\otimes q)\dx{q}}\prt*{\nabla_x u^{m,R,L}}^{\mathrm T} : \varphi \dx{x}\\
	 &+\intO \nabla_x u^{m,R,L}\brk*{\int_D M\Lambda_L(\hpsi^{m,R,L})\chi_R^2(|q|)\, (q\otimes q)\dx{q}} : \varphi \dx{x}
	 \\
	 &+ \intOD M\Lambda_L(\hpsi^{m,R,L})\chi_R(|q|)\, \prt*{(\nabla_q(\chi_R(|q|))\otimes q) : \nabla_x u^{m,R,L}}\, \prt*{(q\otimes q) : \varphi} \dx{q}\dx{x},
\end{align}
where the three terms on the right-hand side of this equality arose from applying the product rule for differentiation to $\nabla_q\prt*{\chi_R(|q|)\, (q\otimes q) : \varphi}$ and noting that $\varphi$ is independent of $q$. Furthermore,
\begin{align}
	\intOD & M\nabla_q\hpsi^{m,R,L} \cdot \nabla_q\prt*{\chi_R(|q|)\,(q\otimes q) : \varphi} \dx{q}\dx{x}\\
    &= 2\intO \bar\sigma^{m,R,L} : \varphi \dx{x} - 2\intO \brk*{\int_D M\hpsi^{m,R,L}\chi_R\dx{q}}\Id : \varphi \dx{x}\\
    &\quad-2\intO \varphi : \int_D M\hpsi^{m,R,L}\prt*{(\nabla_q\chi_R\otimes q) + (q\otimes\nabla_q\chi_R)}\dx{q}\dx{x}\\
    &\quad-\intO \varphi : \int_D M\hpsi^{m,R,L}\,(q\otimes q)\, \Delta\chi_R \dx{q}\dx{x} +\intO \varphi : \int_D M\hpsi^{m,R,L}\,(q\otimes q)\, q\cdot\nabla_q\chi_R \dx{q}\dx{x}.
\end{align}
Testing the resulting equation with $\varphi = \bar\sigma^{m,R,L}$, and using the bounds $|\Delta\chi_R|\lesssim R^{-2}$, $|q\cdot \nabla_q\chi_R|\lesssim C$, we arrive at the following uniform bounds:
\begin{align}
\label{eq:BarSigmaL2}
	\sup_{t\in(0,T)} \intO \abs*{\bar\sigma^{m,R,L}}^2 \dx{x} + \int_0^t\intO \abs*{\nabla_x\bar\sigma^{m,R,L}}^2 \dx{x}\dx{s} &\leq C,\\
	 \|\partial_t\bar\sigma^{m,R,L}\|_{L^2(0,T;[H^1(\Omega;\mathbb{R}^{2 \times 2})]')} \leq C.
\end{align}
Consequently, up to extracting subsequences, there are symmetric matrix-functions $\bar\sigma^{m,R}, \bar\sigma^{m}, \bar\sigma$ such that 
\begin{equation}
    \bar\sigma \in H^1(0,T;[H^1(\Omega;\R^{2\times 2})]')\cap C([0,T];L^2(\Omega;\R^{2\times 2}))\cap L^2(0,T;H^1(\Omega;\R^{2\times 2}))\	\quad \text{with}
\end{equation}
\begin{equation}
    \lim_{L\to\infty} \bar\sigma^{m,R,L} = \bar\sigma^{m,R},\quad \lim_{R\to\infty} \bar\sigma^{m,R} = \bar\sigma^m,\quad \lim_{m\to\infty}\bar\sigma^m = \bar\sigma,
\end{equation}
where the limits are to be understood in the strong sense in $L^2(0,T;L^2(\Omega;\R^{2\times 2}))$ and weakly in $L^2(0,T;H^1(\Omega;\R^{2\times 2}))$.
In fact, thanks to the presence of the truncation $\chi_R$ in the definition of $\bar\sigma^{m,R,L}$, we can identify $\bar\sigma^{m,R}$ as
\begin{equation}
    \bar\sigma^{m,R}(t,x) = \int_D M\hpsi^{m,R}\chi_R(|q|)\,(q\otimes q)\dx{q}.
\end{equation}

Passing to the limit $R\to\infty$ in~\eqref{eq:NSgalerkin3} we obtain
\begin{equation}
\begin{split}
\label{eq:NSgalerkin2D}
	\int_0^t\intO \partial_t u^{m}\cdot w_i\dx{x}\dx s - \int_0^t\intO (u^{m}\otimes u^{m}) : \nabla_x w_i\dx{x}\dx s &+ \nu\int_0^t\intO \nabla_x u^{m} : \nabla_x w_i\dx{x}\dx{s}\\
	&= -\int_0^t \bar\sigma^m :\nabla_x w_i\dx{s},
\end{split}
\end{equation}
for all $i=1,\ldots, m$ and a.e.\ $t\in(0,T)$.
Proceeding as in the general case in the previous section, we readily deduce the following properties
\begin{equation}
	\intO\prt*{\bar\sigma^m - \sigma(\hpsi^m)} : (\xi\otimes\xi) \dx x \geq 0\quad\forall\xi\in\R^d,\;\;\text{a.e.}\; t\in(0,T),
\end{equation}
and
\begin{equation}
	\abs*{\intO\prt*{\bar\sigma^m - \sigma(\hpsi^m)} : \nabla_x\vartheta \dx x} \leq \norm*{\vartheta}_{C^1(\overline{\Omega})} \intO \tr\prt{\sigma^m - \sigma(\hpsi^m)} \dx x\quad \forall \vartheta \in C^1(\overline{\Omega}).
\end{equation}
Consequently, by arguing in the same way as before, we use the Fokker--Planck equation~\eqref{eq:FPgalerkin4} and equation~\eqref{eq:NSgalerkin2D} in conjunction with inequality~\eqref{eq:ConformationEnergy2DM} to infer that
\begin{equation}
	\bar\sigma^m = \sigma(\hpsi^m),\quad \sigma^m = \tr\bar\sigma^m.
\end{equation}
Using the strong convergence of $\hpsi^{m}$ and Fatou's Lemma we deduce that, for a.e.\ $(t,x)$,
\begin{equation}
\label{eq:SigmaPointwise}
    |\sigma(\hpsi)| \lesssim \tr\sigma(\hpsi) \leq \liminf_{m\to\infty}\int_D M\hpsi^{m}|q|^2\dx{q} = \tr\bar\sigma.
\end{equation}
Passing to the limit $m\to\infty$ and setting $m_{NS} = \bar\sigma - \sigma(\hpsi)$ and $m_{OB} = \tr m_{NS}$, we see that $m_{NS}$ and $\sigma(\hpsi)$ both belong to $L^q(0,T;L^p(\Omega;\R^{2\times 2}))$ for $q=\frac{2p}{p-2}$ and $p\in[2,\infty)$, as claimed.

\bigskip

To conclude the final part of the proof of the theorem, we derive a macroscopic equation for $\bar\sigma^m = \sigma(\hpsi^m)$. To this end we will test the Fokker--Planck equation~\eqref{eq:FPfinal} by the following test function:
\begin{equation}
	\phi(t,x,q) = T_L(q_i)T_L(q_j)\varphi_{ij}(t,x), \quad \varphi\in H^1(0,T;C^{1}(\overline{\Omega};\R^{2\times 2})).
\end{equation}  
Recall that the truncation $T_L$ is defined in~\eqref{eq:truncationT} and $T_L'(s) = \chi_L(s)$.
Let us denote by $\sigma^L$ the $2 \times 2$ matrix function whose entries are
\begin{equation}
	\sigma^L_{ij} := \int_D M\hpsi^m T_L(q_i)T_L(q_j) \dx q, \quad i,j =1,2.
\end{equation}
Then $|\sigma^L| \leq C\tr\sigma(\hpsi^m)$, so that $\sigma^L$ inherits all the integrability properties of $\sigma(\hpsi^m)$, and
\begin{equation}
	|\nabla_x\sigma^L_{ij}| = \abs*{\int_D M\nabla_x\hpsi^m T_L(q_i)T_L(q_j) \dx q} \lesssim L^2\int_D M|\nabla_x\hpsi^m| \dx q,
\end{equation}
so that $\sigma^L\in L^2(0,T;H^1(\Omega;\R^{d\times d}))$ for each $L>0$. Similarly, $\sigma^L\in C([0,T];L^2(\Omega;\R^{2\times 2}))$ for each $L>0$. Moreover, by the Dominated Convergence Theorem,
\begin{equation}
	\sigma^L \to \sigma(\hpsi^m)\quad \text{as $L\to\infty$\quad in $L^q(0,T;L^1(\Omega;\R^{2\times 2}))$, $q\in[1,\infty)$}.
\end{equation}
We shall now upgrade this convergence by deriving a uniform bound on  $\sigma^L$ in $L^2(0,T;H^1(\Omega;\R^{2 \times 2}))$.
Testing the Fokker--Planck equation by the function $\phi$ as above, we have
\begin{align}
	\intOT & \int_D M\hpsi^m\partial_t\phi \dx{q}\dx{x}\dx{s} + \int_{\Omega\times D} M\hpsi_0\phi(0)\dx{q}\dx{x} - \int_{\Omega\times D} M\hpsi^m(t)\phi(t)\dx{q}\dx{x}\\
	&=\intOT \sigma^L : \partial_t\varphi \dx x \dx s + \intO \prt*{\int_D M\hpsi_0T_L(q_i)T_L(q_j)\dx q} : \varphi(0)\dx x - \intO \sigma^L(t) : \varphi(t) \dx x.
\end{align}

The $q$-diffusion term becomes
\begin{align}
	-\int_0^t\intOD & M\nabla_q\hpsi^m \cdot \nabla_q\prt{T_L(q_i)T_L(q_j)}\varphi_{ij} \dx q \dx x \dx s\\
	 =&  - \int_0^t\intOD M\partial_{q_k}\hpsi^m T_L'(q_i)T_L(q_j)\delta_{ik}\varphi_{ij} \dx q \dx x \dx s\\
	&- \int_0^t\intOD M\partial_{q_k}\hpsi^m T_L'(q_j)T_L(q_i)\delta_{jk}\varphi_{ij} \dx q \dx x \dx s.
\end{align}
Integrating by parts, we have
\begin{align}
	- \int_0^t\intOD M\partial_{q_k}\hpsi^m T_L'(q_k)T_L(q_j)\varphi_{kj} \dx q \dx x \dx s =& \int_0^t\intOD M\hpsi^m T_L''(q_k)T_L(q_j)\varphi_{kj} \dx q\dx x\dx s\\
		&+ \int_0^t\intOD M\hpsi^m (T_L'(q_k))^2\varphi_{kk} \dx q\dx x\dx s\\
		&- \int_0^t\intOD M\hpsi^m T_L'(q_k)T_L(q_j)q_k\varphi_{kj} \dx q\dx x\dx s.
\end{align}
For the first term we observe that
\begin{equation}
	 \abs*{\int_0^t\intOD M\hpsi^m T_L''(q_k)T_L(q_j)\varphi_{kj} \dx q\dx x\dx s} \leq \frac{C}{L} \int_0^t\intOD M\hpsi^m |q| \dx q\dx x\dx s \lesssim \frac{1}{L}.
\end{equation}
For the transport term we write
\begin{align}
	\int_0^t\intOD M\hpsi^m u^m &\cdot \nabla_x \phi \dx q\dx x\dx s = \int_0^t\intO \sigma^L_{ij} u^m\cdot \nabla_x\varphi_{ij} \dx x\dx s,
\end{align}
and finally for the diffusion term
\begin{align}
	\intOT\int_D M\nabla_x\hpsi^m \cdot \nabla_x\phi\dx q\dx x \dx s &= \intOT\int_D M\nabla_x\hpsi^m T_L(q_i)T_L(q_j)\cdot\nabla_x\varphi_{ij}\dx q \dx x \dx s\\
	& = \intOT \nabla_x\sigma^L_{ij} \cdot \nabla_x \varphi_{ij} \dx x \dx s.
\end{align}
We thus arrive at an evolution equation for $\sigma^L$.  Let us now observe the following bounds
\begin{align}
	 \int_0^t\intOD M\hpsi^m T_L''(q_k)T_L(q_j)\varphi_{kj} \dx q\dx x\dx s &\lesssim \frac{1}{L}\intOT|\varphi|\int_D M\hpsi^m|q|\dx q \dx x \dx s \\
	 & \lesssim \frac{1}{L}\norm{\varphi}_{L^2(0,T;L^2(\Omega;\R^{2\times 2}))};
\end{align}
\begin{align}
	\int_0^t\intOD M\hpsi^m (T_L'(q_k))^2\varphi_{kk} \dx q\dx x\dx s &\lesssim \intOT|\varphi|\int_D M\hpsi^m \dx q \dx x \dx s \lesssim \norm{\varphi}_{L^2(0,T;L^2(\Omega;\R^{2\times 2}))};
\end{align}
\begin{align}
	\int_0^t\intOD M\hpsi^m T_L'(q_k)T_L(q_j)q_k\varphi_{kj} \dx q\dx x\dx s &\lesssim \intOT|\varphi|\int_DM\hpsi^m|q|^2\dx q\dx x\dx s\\
	 &\lesssim \norm{\varphi}_{L^2(0,T;L^2(\Omega;\R^{2\times 2}))}\norm{\sigma(\hpsi^m)}_{L^2(0,T;L^2(\Omega;\R^{2\times 2}))};
\end{align}
and
\begin{align}
	\int_0^t\intOD & M\hpsi^m (\nabla_xu^m)_{kl}q_lT_L'(q_k)T_L(q_j)\varphi_{kj}\dx q\dx x\dx s \\
	&\lesssim \intOT |\nabla_x u^m||\varphi|\tr\sigma(\hpsi^m) \dx x\dx s\\
	&\lesssim \norm{\nabla_xu^m}_{L^\infty(0,T;L^\infty(\Omega;\R^{2\times 2}))}\norm{\sigma(\hpsi)}_{L^2(0,T;L^2(\Omega;\R^{2\times 2}))}\norm{\varphi}_{L^2(0,T;L^2(\Omega;\R^{2\times 2}))}.
\end{align}

Since $\sigma^L$ belongs to $L^2(0,T;H^1(\Omega;\R^{2\times 2}))$ for each $L>0$, there is a sequence $(\sigma^{L,n})_{n\in\mathbb{N}} \subset L^2(0,T;C^1(\overline{\Omega};\R^{2\times 2}))$ which converges to $\sigma^L$ in the $L^2(0,T;H^1(\Omega;\R^{2 \times 2}))$ topology. We can then choose $\varphi=\sigma^{L,n}$. These bounds, in conjunction with the uniform bound $\sigma^L\inb L^2(0,T;L^2(\Omega;\R^{2\times 2}))$, show that in the limit $n\to\infty$ each of the above terms is bounded uniformly in $L$. For the terms involving spatial derivatives we observe that
\begin{align}
	 \intOT \sigma^L_{ij} u^m &\cdot \nabla_x\sigma^{L,n}_{ij}\dx x\dx s = \intOT \sigma^L_{ij} u^m\cdot \nabla_x\sigma^L_{ij}\dx x\dx s + \intOT \sigma^L_{ij} u^m\cdot \nabla_x(\sigma^{L,n}_{ij}-\sigma^L_{ij})\dx x\dx s\\
	 &=\intOT \sigma^L_{ij} u^m\cdot \nabla_x(\sigma^{L,n}_{ij}-\sigma^L_{ij})\dx x\dx s\\
	 &\leq \norm{u^m}_{L^\infty(0,T;L^\infty(\Omega;\R^2))}\norm{\sigma(\hpsi)}_{L^2(0,T;L^2(\Omega;\R^{2\times 2}))}\norm{\nabla_x(\sigma^{L,n}_{ij}-\sigma^L_{ij})}_{L^2(0,T;L^2(\Omega;\R^{2}))},
\end{align}
and
\begin{align}
	\intOT \nabla_x\sigma^L_{ij} \cdot \nabla_x \sigma^{L,n}_{ij} \dx x \dx s \to \intOT |\nabla_x\sigma^L_{ij}|^2\dx x \dx s.
\end{align}
By applying an additional mollification in the time variable, we readily deduce that
\begin{align}
	\intOT & \sigma^L : \partial_t\sigma^{L,n}\dx x \dx s + \intO \prt*{\int_D M\hpsi_0T_L(q_i)T_L(q_j)\dx q} : \sigma^{L,n}(0)\dx x - \intO \sigma^L(t) : \sigma^{L,n}(t) \dx x\\
	 &\to \frac12\norm{\sigma^L(t)}^2_{L^2(\Omega;\R^{2\times 2})} - \frac12\norm{\sigma(\hpsi_0)}^2_{L^2(\Omega;\R^{2\times 2})}.
\end{align}
Collecting all the pieces, we arrive at the following uniform bound on $\sigma^L$:
\begin{equation}
	\frac12\norm{\sigma^L(t)}^2_{L^2(\Omega;\R^{2\times 2})} + \mu \int_0^t \intO |\nabla_x\sigma^L(s)|^2\dx x\dx s \leq C + \frac12\norm{\sigma(\hpsi_0)}^2_{L^2(\Omega;\R^{2\times 2})}.
\end{equation}
It follows that, up to a subsequence, we have the weak convergence
\begin{equation}
	\nabla_x\sigma^L \rightharpoonup \nabla_x\sigma(\hpsi) \quad \text{weakly in $L^2(0,T;L^2(\Omega;\R^{2\times 2}))$}.
\end{equation}

We are now ready to pass to the limit $L\to\infty$ in the equation for $\sigma^L$. 
Using the properties of $T_L$ and its derivatives, we obtain
\begin{align}
	- \int_0^t\intOD M\partial_{q_k}\hpsi^m T_L'(q_k)T_L(q_j)\varphi_{kj} \dx q \dx x \dx s \to &-\intOT \sigma(\hpsi^m) : \varphi \dx x \dx s\\
	 &+ \intOT \underbrace{\prt*{\int_D M\hpsi^m\dx{q}}}_{=1}\Id:\varphi \dx x \dx s;
\end{align}
and
\begin{align}
\int_0^t\intOD & M\hpsi^m (\nabla_xu^m)_{kl}q_lT_L'(q_k)T_L(q_j)\varphi_{kj}\dx q\dx x\dx s\\
 &+ \int_0^t\intOD M\hpsi^m (\nabla_xu^m)_{kl}q_lT_L'(q_k)T_L(q_i)\varphi_{ik}\dx q\dx x\dx s\\
 	  \to& \intOT \nabla_xu^m\sigma(\hpsi^m) : \varphi\dx x \dx s + \int_0^t\intO \sigma(\hpsi^m)(\nabla_xu^m)^T : \varphi \dx x \dx s,
\end{align}
where the convergence follows from the Dominated Convergence Theorem, by noting that $\sigma(\hpsi^m)\inb L^\infty(0,T;L^2(\Omega;\R^{2\times 2}))$ and $\nabla_xu^m\inb L^2(0,T;L^2(\Omega;\R^{2\times 2}))$.
For the transport term we write
\begin{align}
	\int_0^t\intO\sigma^L_{ij} u^m\cdot \nabla_x\varphi_{ij} \dx q\dx x\dx s
	&\to \int_0^t\intO \prt{u^m\cdot\nabla_x}\varphi : \sigma(\hpsi^m) \dx x \dx s\\ &= -\int_0^t\intO \prt{u^m\cdot\nabla_x}\sigma(\hpsi^m) : \varphi \dx x \dx s,
\end{align}
where the last equality follows since $u^m$ vanishes on the boundary of $\Omega$. 

We thus arrive at the following equation satisfied by $\sigma(\hpsi^m)$: 
\begin{equation}
\label{eq:SigmaMEquation}
\begin{aligned}
	-\int_0^t\intO \sigma(\hpsi^m) : \partial_t\varphi \dx{x}\dx{s} =& \intO \sigma(\hpsi_0) : \varphi(0) \dx{x} - \intO \sigma(\hpsi^m)(t) : \varphi(t) \dx{x}\\
	&-\int_0^t\intO\prt*{u^m\cdot\nabla_x}\sigma(\hpsi^m) : \varphi\dx{x}\dx{s} - \mu\int_0^t\intO\nabla_x\sigma(\hpsi^m) :: \nabla_x\varphi \dx{x}\dx{s}\\
	& +\int_0^t\intO\sigma(\hpsi^m)\prt*{\nabla_xu^m}^{\mathrm{T}} : \varphi \dx{x}\dx{s} + \int_0^t\intO(\nabla_x u^m) \sigma(\hpsi^m) : \varphi \dx{x}\dx{s}\\
	& - 2\int_0^t\intO\prt*{\sigma(\hpsi^m)-\Id} : \varphi\dx{x}\dx{s}\\
	&\forall t\in[0,T],\;\;\forall\, \varphi\in W^{1,1}(0,T;C^1(\overline{\Omega};\R^{2\times 2})).
\end{aligned}
\end{equation}
Now, using density of the space $C^{1}(\overline{\Omega};\R^{2\times 2})$ in $H^1(\Omega;\R^{2\times 2})$ and passing to the limit $m\to\infty$, we deduce that the pair $(u, \bar{\sigma})$ is a weak solution of the Oldroyd-B system. We postpone the proof of uniqueness to the next section.
Finally, passing to the limit in~\eqref{eq:ConformationEnergy2DM} we obtain the energy inequality
\begin{equation}
\label{eq:ConformationEnergy2DFinal}
\begin{split}
	\frac12\intO |u(t)|^2\dx{x} &+ \frac12 \intO \tr{\bar{\sigma}}(t) \dx{x} + \nu\intOT |\nabla_xu|^2\dx{x}\dx{s} + \intOT \tr{(\bar{\sigma}-\Id)}\dx{x}\dx{s}\\
	 &\leq \frac12\intO|u(0)|^2\dx{x} + \frac12\intO \tr{\sigma(\hpsi_0)}\dx{x}.
\end{split}
\end{equation}
\qed

\begin{corollary}
\label{thm:Integrability2D}
	If $\sigma(\hpsi)$ belongs to $L^\infty(0,T;L^p(\Omega;\R^{2\times 2}))$ for some $p>2$, then, for the generalised solution as constructed above, we have $\sigma(\hpsi) = \bar\sigma$, yielding a weak solution to the 2D NSFP system.
\end{corollary}
\begin{proof}
		With the additional integrability we can repeat the argument from the last part of the previous proof to derive an equation for $\sigma^L$, this time defining its entries via
		\begin{equation}
			\sigma^L_{ij}  :=  \int_D M\hpsi\, T_L(q_i)\,T_L(q_j) \dx q, \quad i,j=1,2.
		\end{equation}
	We arrive at an identical equation with $\hpsi^m$ replaced by $\hpsi$ and $u^m$ replaced by $u$. In deriving the energy estimate for $\sigma^L$ we now cannot use the $W^{1,\infty}$ bound on the velocity. Instead, we can estimate the terms involving the velocity as follows:
\begin{align}
	\int_0^t\intOD & M\hpsi (\nabla_xu)_{kl}q_lT_L'(q_k)T_L(q_j)\varphi_{kj}\dx q\dx x\dx s \\
	&\lesssim \intOT |\nabla_x u||\varphi|\tr\sigma(\hpsi) \dx x\dx s\\
	&\lesssim \norm{\nabla_xu}_{L^2(0,T;L^2(\Omega;\R^{2\times 2}))}\norm{\sigma(\hpsi)}_{L^\infty(0,T;L^p(\Omega;\R^{2\times 2}))}\underbrace{\norm{\varphi}_{L^2(0,T;L^{\frac{2p}{p-2}}(\Omega;\R^{2\times 2}))}}_{\leq C\norm{\varphi}_{L^2(0,T;H^1(\Omega;\R^{2\times 2}))}},
\end{align}	
	and
	\begin{align}
	 \intOT \sigma^L_{ij} u^m &\cdot \nabla_x\sigma^{L,n}_{ij}\dx x\dx s = \intOT \sigma^L_{ij} u\cdot \nabla_x\sigma^L_{ij}\dx x\dx s + \intOT \sigma^L_{ij} u\cdot \nabla_x(\sigma^{L,n}_{ij}-\sigma^L_{ij})\dx x\dx s\\
	 &=\intOT \sigma^L_{ij} u\cdot \nabla_x(\sigma^{L,n}_{ij}-\sigma^L_{ij})\dx x\dx s\\
	 &\leq \norm{u}_{L^2(0,T;L^{\frac{2p}{p-2}}(\Omega;\R^2))}\norm{\sigma(\hpsi)}_{L^\infty(0,T;L^p(\Omega;\R^{2\times 2}))}\norm{\nabla_x(\sigma^{L,n}_{ij}-\sigma^L_{ij})}_{L^2(0,T;L^2(\Omega;\R^{2}))}.
\end{align}
	By passing to the limit $L\to\infty$, it follows that $\sigma(\hpsi	)$ belongs to $L^2(0,T;H^1(\Omega;\R^{2\times 2}))$ and satisfies the Oldroyd-B stress equation~\eqref{eq:OBSigma}, and so does $\bar\sigma$, with the same velocity $u$. Taking the difference of the two formulations and testing by the difference $\bar\sigma - \sigma(\hpsi)$\footnote{Strictly speaking, a mollification in time is required for this testing to be admissible.} we easily deduce that $\bar\sigma = \sigma(\hpsi)$ a.e. (similarly as in the more general calculation below).
\end{proof}

\begin{remark}
The above corollary is stated in terms of conditional integrability of $\sigma(\hpsi)$. Let us point out however that using the pointwise bound~\eqref{eq:SigmaPointwise}, we obtain the same integrability of $\sigma(\hpsi)$ as that of $\bar\sigma$. 
Since the latter satisfies equation~\eqref{eq:OBSigma}, Corollary~\ref{thm:Integrability2D} can be seen as a sufficient condition for existence of weak solutions to the 2D NSFP system in terms of regularity of the solution to the Oldroyd-B system. We leave open the question of minimal conditions for $\bar\sigma$ to belong to $L^\infty L^p$. Observe that a slight modification of the above proof allows for the same conclusion under additional regularity of the velocity gradient -- namely, $\nabla_x u \in L^2(0,T;L^p(\Omega;\R^{2\times 2}))$ for some $p>2$.
\end{remark}

\section{Conditional regularity and uniqueness}
\label{sec:Conditional}

\subsection{Proof of Theorem~\ref{thm:Regularity3D}}

Let us now assume additional regularity of the velocity field, namely that 
\begin{equation}
\label{eq:AdditionalVelocityRegularity}
	\partial_t u \in L^1(0,T;L^2(\Omega;\mathbb{R}^3))\quad \mbox{and}\quad \nabla_x u\in L^1(0,T;C(\bar\Omega;\R^{3 \times 3})).
\end{equation}
Testing the Fokker--Planck equation~\eqref{eq:GeneralisedFP3D} with the function $\phi(q) = T_R(|q|^2)$ (with the same truncation as before, defined in~\eqref{eq:truncationT}), we obtain
\begin{align}
	\intOD & M\hpsi_0 T_R(|q|^2)\dx{q}\dx{x} - \intOD M\hpsi(t) T_R(|q|^2)\dx{q}\dx{x} \\ 
	&= - \intOT\int_D M\hpsi \prt{(\nabla_x u) q} \cdot \nabla_q T_R(|q|^2)\dx{q}\dx{x}\dx{s}
	+ \intOT\int_D M\nabla_q\hpsi\cdot\nabla_qT_R(|q|^2)\dx{q}\dx{x}\dx{s}.
\end{align}
The first term on the right-hand side can be rewritten as
\begin{equation}
	- \intOT\int_D M\hpsi \prt{(\nabla_x u) q} \cdot \nabla_q T_R(|q|^2)\dx{q}\dx{x}\dx{s} = -2\intOT \prt*{\int_D M\hpsi \chi_{R}(|q|^2)\,(q\otimes q)\dx{q}} : \nabla_x u\dx{x}\dx{s},
\end{equation}
while the other term can be written as
\begin{equation}
\begin{split}
	2\intOT\int_D M\nabla_q\hpsi\cdot q \chi_{R}(|q|^2)\dx{q}\dx{x}\dx{s} & = 2\intOT\int_D M\hpsi |q|^2 \chi_{R}(|q|^2) \dx{q}\dx{x}\dx{s}\\
	& -2d\intOT\int_D M\hpsi\chi_{R}(|q|^2)\dx{q}\dx{x}\dx{s}\\
	& -2\intOT\int_D M\hpsi q\cdot \nabla_q\chi_{R}(|q|^2) \dx{q}\dx{x}\dx{s}.
\end{split}
\end{equation}
As before, the last term can be bounded as follows:
\begin{equation}
	\intOT\int_D M\hpsi q\cdot \nabla_q\chi_{R}(|q|^2) \dx{q}\dx{x}\dx{s} = 2\intOT\int_D M\hpsi |q|^2\,\chi_{R}'(|q|^2) \dx{q}\dx{x}\dx{s} \lesssim \frac{1}{R}.
\end{equation}
In each of the other terms we can pass to the limit $R\to\infty$ using the Monotone Convergence Theorem. Therefore we deduce that
\begin{equation}
	\intO \tr \sigma(\hpsi_0) \dx{x} - \intO  \tr \sigma(\hpsi)(t) \dx{x} = -2\intOT \sigma(\hpsi) : \nabla_x u \dx{x}\dx{s} + 2 \intOT \tr(\sigma(\hpsi) - \Id) \dx{x}\dx{s}.
\end{equation}
Next, we test the Navier--Stokes equation~\eqref{eq:GeneralisedNS3D} with $u$ (which is admissible because of the assumed additional regularity~\eqref{eq:AdditionalVelocityRegularity}) to derive the equality
\begin{equation}
\begin{split}
	\frac12\intO |u(t)|^2\dx{x} &- \frac12\intO |u_0|^2\dx{x} + \nu \intOT |\nabla_x u|^2\dx{x}\dx{s} \\
	&= -\int_0^t \sigma(\hpsi) : \nabla_x u\dx{s} - \intOT \skp*{m_{NS}(s),\nabla_x u}\dx{s}. 
\end{split}
\end{equation}
Consequently, we obtain the energy \emph{equality}
	\begin{equation}
	\begin{aligned}
	\label{eq:OBEnergyEquality}
		\frac12\intO |u(t)|^2\dx{x} &+ \frac12\intO \tr{\sigma(\hpsi)}(t)\dx{x} + \nu\intOT |\nabla_x u|^2\dx{x}\dx{s} + \intOT\tr{\prt*{\sigma(\hpsi)-\Id}}\dx{x}\dx{s}\\
		&= \frac12\intO |u_0|^2\dx{x} + \frac12\intO \tr{\sigma(\hpsi_0)}\dx{x} - \int_0^t\skp*{m_{NS},\nabla_x u}\dx{s}.
	\end{aligned}
	\end{equation}	
Comparing this with the energy inequality~\eqref{eq:GeneralisedEI3D}, we deduce that
\begin{align}
	\skp*{m_{OB}(t),\mathbbm{1}_{\overline\Omega}} + \int_0^t \skp*{m_{OB}(s),\mathbbm{1}_{\overline\Omega}}\dx{s} \leq \int_0^t\skp*{m_{NS}(s),\nabla_x u}\dx{s}.
\end{align}	
The compatibility condition implies that
\begin{equation}
\left|\int_0^t\skp*{m_{NS}(s),\nabla_x u}\dx{s}\right| \leq \int_0^t \zeta(s)\norm{\nabla_x u(s)}_{C(\overline\Omega;\R^{3\times 3})} \skp*{m_{OB}(s),\mathbbm{1}_{\overline\Omega}}\dx{s}. 
\end{equation}
Consequently, Gronwall's Lemma implies that $m_{OB}\equiv 0$. Then, by compatibility, we must have that $m_{NS}\equiv 0$. \qed

\subsection{Proof of Corollary~\ref{thm:Regularity2D}}

By assuming the stated additional regularity of the initial datum for the stress tensor, together with the Oldroyd-B stress evolution equation satisfied by $\bar\sigma$, we are in a position to deduce higher regularity for $\bar\sigma$, and subsequently of the velocity. 

Indeed, since $u\in L^4(0,T;L^4(\Omega;\R^2))$ and $\bar\sigma\in L^2(0,T;H^1(\Omega;\R^{2\times 2}))$, we see that 
\begin{equation}
    \prt*{u\cdot\nabla_x}\bar\sigma \in L^{4/3}(0,T;L^{4/3}(\Omega;\R^{2\times 2})).
\end{equation} Similarly, $(\nabla_x u)\bar\sigma\in L^{4/3}(0,T;L^{4/3}(\Omega;\R^{2\times 2}))$. 
It follows from Lemma~3.2 in~\cite{BaSu18} that
\begin{equation}
    \bar\sigma \in L^{4/3}(0,T;W^{2,4/3}(\Omega;\R^{2\times 2})),\quad \div\bar\sigma \in L^{4/3}(0,T;L^4(\Omega;\R^{2\times 2})).
\end{equation}
With this regularity in hand, we can apply Lemma~3.6 in~\cite{BaSu18} to deduce that
\begin{equation}
    \partial_t u \in L^2(0,T;L^2(\Omega;\R^2)),\quad u \in L^{4/3}(0,T;W^{2,q}(\Omega;\R^2)),\quad q\in[1,4).
\end{equation}
We can therefore repeat the argument from the proof of Theorem~\ref{thm:Regularity3D} above to conclude that $m_{NS}$ and $m_{OB}$ vanish and therefore $\bar\sigma = \sigma(\hpsi)$. \qed

\subsection{Proof of uniqueness for the Oldroyd-B system in 2D}

Suppose that $(u^1, \bar\sigma^1)$ and $(u^2,\bar\sigma^2)$ are two weak solutions of the Oldroyd-B system with the same initial data $(u_0, \bar\sigma_0)$.
Using the energy identity~\eqref{eq:OBEnergyEquality} for both solutions we obtain the following equality:
\begin{align}
	&\frac12\intO |u^1(t)-u^2(t)|^2 \dx x + \nu \int_0^t\intO |\nabla_x u^1 - \nabla_x u^2|^2 \dx x\dx s\\ 
	&=
	 \frac12 \intO |u^1(t)|^2 \dx x + \nu \int_0^t\intO |\nabla_x u^1|^2 \dx x \dx s + \frac12\intO |u^2(t)|^2 \dx x + \nu \int_0^t\intO |\nabla_x u^2|^2 \dx x \dx s\\
	 &\quad  - \intO u^1(t) \cdot u^2(t) \dx x - 2\nu \int_0^t\intO \nabla_x u^1 : \nabla_x u^2 \dx x \dx s \\
	&= \frac12 \intO|u_0|^2 \dx x + \frac{1}{2} \intO \tr{\bar\sigma_0} \dx x - \int_0^t\intO \tr{(\bar\sigma^1-\Id)} \dx x \dx s - \frac{1}{2} \intO \tr{\bar\sigma^1(t)}\dx x \\
	&\quad + \frac12\intO |u_0|^2\dx x + \frac{1}{2}\intO \tr{\bar\sigma_0} - \int_0^t\intO \tr{(\bar\sigma^2-\Id)}\dx x\dx s - \frac{1}{2} \intO \tr{\bar\sigma^2(t)}\dx x\\
	&\quad - \intO u^1(t) \cdot u^2(t) \dx x - 2\nu \int_0^t\intO \nabla_x u^1 : \nabla_x u^2\dx x\dx s.
\end{align}
The last two terms are treated as is usual for the Navier--Stokes equations, see for instance~\cite{Wiedemann_survey}, to deduce that
\begin{align}
	\frac12\intO|u_0|^2 \dx x &+ \frac12 \intO|u_0|^2 \dx x - \intO u^1(t) \cdot u^2(t) \dx x - 2\nu \int_0^t\intO \nabla_x u^1 : \nabla_x u^2\dx x\dx s \\
	&= \int_0^t\intO \nabla_x(u^1-u^2) : \prt*{(u^1-u^2)\otimes u^2} \dx x \dx s.
\end{align}
The terms involving the tensors $\bar\sigma^i$, $i=1,2$, can be rewritten using their evolution equations: testing the equation~\eqref{eq:OBSigma} by $\frac{1}{2}\Id$, we obtain, for $i=1,2$,
\begin{equation}
	-\frac{1}{2} \intO \prt*{\bar\sigma^i(t)-\bar\sigma_0}\dx x  - \int_0^t\intO\tr{(\bar\sigma^i-\Id)}\dx x \dx s = 2\int_0^t\intO \bar\sigma^i : \nabla_x u^i \dx x\dx s.
\end{equation}
Therefore, we have
\begin{equation}
\begin{aligned}
\label{eq:RelativeEnergy1}
	\frac12&\intO |u^1(t)-u^2(t)|^2 \dx x + \nu \int_0^t\intO |\nabla_x u^1 - \nabla_x u^2|^2 \dx x \dx s\\
	 & = \int_0^t\intO \nabla_x(u^1-u^2) : \prt*{(u^1-u^2)\otimes u^2} \dx x \dx s + 2\int_0^t\intO \prt*{\bar\sigma^1-\bar\sigma^2} : \nabla_x(u^1-u^2) \dx x \dx s .
\end{aligned}
\end{equation}
The first term on the right-hand side can be bounded as follows, using the Gagliardo--Nirenberg and Young inequalities:
\begin{align}
	\int_0^t & \intO \nabla_x(u^1-u^2) : \prt*{(u^1-u^2)\otimes u^2} \dx x \dx s\\
	 &\leq \int_0^t \norm{\nabla_x(u^1-u^2)}_{L^2(\Omega;\R^{d\times d})}\norm*{u^2}_{L^4(\Omega;\R^d)}\norm*{u^1-u^2}_{L^4(\Omega;\R^d)} \dx s\\
	 &\leq \frac{\delta}{2}\int_0^t\intO |\nabla_x u^1 - \nabla_x u^2|^2 \dx x \dx s + \frac{1}{2\delta}\int_0^t \norm*{u^2}^2_{L^4(\Omega;\R^d)}\norm*{u^1-u^2}^2_{L^4(\Omega;\R^d)} \dx s\\
	 & \leq \delta\int_0^t\intO |\nabla_x u^1 - \nabla_x u^2|^2 \dx x \dx s + \frac{C}{\delta^3}\int_0^t \norm*{u^2}^4_{L^4(\Omega;\R^d)}\intO|u^1-u^2|^2\dx x \dx s, 
\end{align}
where $C$ is a positive constant, independent of $\delta$. For the other term on the right-hand side of~\eqref{eq:RelativeEnergy1} we write, using Young's inequality,
\begin{equation}
\begin{aligned}
	\int_0^t & \intO \prt*{\bar\sigma^1-\bar\sigma^2} : \nabla_x(u^1-u^2)\dx x\dx s\leq \delta \int_0^t\intO |\nabla_x(u^1-u^2)|^2\dx x \dx s + \frac{1}{4\delta} \int_0^t\intO \abs*{\bar\sigma^1-\bar\sigma^2}^2 \dx x \dx s,
\end{aligned}
\end{equation}
and to close the estimate, we require a bound on the square of the difference of $\bar\sigma^1$ and $\bar\sigma^2$.
To this end, we subtract the corresponding equations and test by the difference, which yields
\begin{equation}
\begin{aligned}
\label{eq:RelativeEnergy2}
	&\frac12\intO \abs*{\bar\sigma^1(t)-\bar\sigma^2(t)}^2 \dx x + \mu \int_0^t\intO \abs*{\nabla_x(\bar\sigma^1-\bar\sigma^2)}^2 \dx x\dx s + 2 \int_0^t\intO \abs*{\bar\sigma^1-\bar\sigma^2}^2\dx x \dx s \\
	&= - \int_0^t\intO \prt*{(u^1-u^2)\cdot\nabla_x}(\bar\sigma^1-\bar\sigma^2) : \bar\sigma^2\dx x \dx s \\
	&\quad+ \int_0^t\intO (\nabla_x u^1 \bar\sigma^1+\bar\sigma^1\nabla_x^T u^1) : \bar\sigma^2 + (\nabla_x u^2 \bar\sigma^2+\bar\sigma^2\nabla_x^T u^2) : \bar\sigma^1 \dx x \dx s \\
	&\;\;\; - \int_0^t\intO (\nabla_x u^1 \bar\sigma^1+\bar\sigma^1\nabla_x^T u^1) : \bar\sigma^1 + (\nabla_x u^2 \bar\sigma^2+\bar\sigma^2\nabla_x^T u^2) : \bar\sigma^2 \dx x \dx s.
\end{aligned}
\end{equation}
The first term on the right-hand side can be treated similarly as before, producing
\begin{align}
	\int_0^t & \intO \prt*{(u^1-u^2)\cdot\nabla_x}(\bar\sigma^1-\bar\sigma^2) : \bar\sigma^2\dx x \dx s\\
	 & \leq \delta\int_0^t\intO \abs*{\nabla_x(\bar\sigma^1-\bar\sigma^2)}^2 \dx x \dx s + \delta\int_0^t\intO |\nabla_x u^1 - \nabla_x u^2|^2 \dx x \dx s\\
	 &\quad + \frac{C}{\delta^3}\int_0^t \norm{\bar\sigma^2}^4_{L^4(\Omega;\R^{2\times 2})}\intO|u^1-u^2|^2\dx x \dx s,
\end{align}
where, again, $C$ is a positive constant independent of $\delta$; while the last two terms can be rewritten, using symmetry, as
\begin{align}
	\int_0^t\intO &\nabla_x(u^1-u^2)\bar\sigma^1 : (\bar\sigma^1 -\bar\sigma^2) \dx x \dx s- \int_0^t\intO \nabla_x u^2 (\bar\sigma^1-\bar\sigma^2) : (\bar\sigma^1 - \bar\sigma^2)\dx x \dx s.
\end{align}
The first of the above terms is bounded again using the Gagliardo--Nirenberg inequality and the $L^4(0,T;L^4(\Omega;\R^{2\times 2}))$ regularity of the conformation tensor, while the second term can be 
bounded by integration by parts and the $L^4(0,T;L^4(\Omega;\R^2))$ norm boundedness of $u^2$.
	
Combining~\eqref{eq:RelativeEnergy1} and~\eqref{eq:RelativeEnergy2} together with the above estimates, we obtain, upon choosing $\delta>0$ small enough (in terms of $\nu$ and $\mu$),
\begin{align}
	\intO |u^1-u^2|^2 \dx x &+ \intO \abs*{\bar\sigma^1-\bar\sigma^2}^2 \dx x
	\leq \int_0^t a(t) \intO|u^1-u^2|^2 \dx x \dx s + \int_0^t b(t) \intO\abs*{\bar\sigma^1-\bar\sigma^2}^2 \dx x \dx s, 
\end{align}
for some functions $a,b\in L^1(0,T)$. We conclude by Gronwall's Lemma that $u^1 = u^2$ and $\bar\sigma^1=\bar\sigma^2$ a.e.\ in $(0,T)\times\Omega$. We thus obtain uniqueness for the Oldroyd-B equation. \qed

\section*{{\it{Acknowledgements}}}
This project was initiated while T.D.\ was visiting the Mathematical Institute at the University of Oxford, whose kind hospitality he gratefully acknowledges. T.D.\ was partially supported by the National Science Center (Poland), grant number 2017/27/B/ST1/01569, and the ``New Ideas'' Programme, University of Warsaw, Excellence Initiative Research University.

\end{document}